\documentclass[12pt, a4paper]{amsart}

\usepackage[margin=2.5cm, headheight=13pt, headsep=13pt]{geometry}
\usepackage{newtxtext}
\usepackage[subscriptcorrection]{newtxmath}
\usepackage[dvipsnames]{xcolor}
\usepackage{hyperref}
\usepackage[sort]{cleveref}
\hypersetup{
  colorlinks,
  citecolor=blue,
  filecolor=magenta,
  linkcolor=magenta,
  urlcolor=magenta
}

\usepackage[square,numbers]{natbib}
\usepackage{graphicx}
\usepackage{multicol,multirow}
\usepackage{amsmath,amsfonts}
\usepackage{amsthm}
\usepackage{rotating}
\usepackage{appendix}
\usepackage{ifpdf}
\usepackage[T1]{fontenc}

\usepackage{tikz}
\usepackage{enumitem}

\usepackage[subrefformat=parens, labelfont=up]{subcaption}

\usetikzlibrary{decorations.pathreplacing}

\newtheorem{theorem}{Theorem}[section]
\newtheorem{lemma}[theorem]{Lemma}
\newtheorem{definition}[theorem]{Definition}
\newtheorem{proposition}[theorem]{Proposition}
\theoremstyle{definition}

\newtheorem{remark}[theorem]{Remark}
\newtheorem{example}[theorem]{Example}
\newtheorem{problem}[theorem]{Problem}
\newtheorem{corollary}[theorem]{Corollary}
\numberwithin{equation}{section}

\newcommand{\GG}{\mathcal{G}}
\newcommand{\HH}{\mathcal{H}}
\newcommand{\R}{\mathbb{R}}
\newcommand{\Z}{\mathbb{Z}}

\newcommand\independent{\protect\mathpalette{\protect\independenT}{\perp}}
\def\independenT#1#2{\mathrel{\rlap{$#1#2$}\mkern2mu{#1#2}}}

\title{On unirational varieties with poset parameterizations}
\author{Marina Garrote-L\'{o}pez, Nataliia Kushnerchuk and Liam Solus}

\address[M.~Garrote-L\'{o}pez]{ Universitat Pompeu Fabra, Barcelona, Spain}
\email{marina.garrote@upf.edu}

\address[N.~Kushnerchuk]{Aalto University, Espoo, Finland}
\email{nataliia.kushnerchuk@aalto.fi}

\address[L.~Solus]{Department of Mathematics, KTH Royal Institute of Technology, Stockholm, Sweden}
\email{solus@kth.se}
\date{\today}


\begin{document}

\begin{abstract}
We use partially ordered sets (posets) to provide a canonical parameterization for the Zariski closure of the image of a semialgebraic set under a rational map whose coordinate functions are polynomials with nonnegative integral coefficients. 
The resulting poset parametrization of such a unirational variety allows us to translate several well-studied problems into combinatorics; e.g. reducing the problems to describing the poset associated to the variety.
These problems include, the implicitization problem from algebraic geometry, the toric reparameterization problem, the computation of the linear span of the variety, and the problem of distinguishing two semialgebraic subsets of the same ambient space. 
The technique applies to instances of these problems in several fields, including algebraic geometry, algebraic combinatorics, statistics and applied algebra.
We demonstrate the technique on examples from each field, including degenerate subvarieties of secant varieties, matroid flat varieties -- which generalize toric varieties of edge polytopes, as well as varieties arising in multivariate data analysis and evolutionary biology.
\end{abstract}

\maketitle

\section{Introduction}

A substantial body of research in fields such as applied algebra, computational algebra, real algebraic geometry and optimization is focused on describing the geometry of sets $\varphi(\Theta)\subseteq \R^m$ that are defined as the image of a semialgebraic set $\Theta\subseteq \mathbb{R}^n$ under a rational map 
\begin{equation}
\label{eqn: rational map}
\varphi:\Theta \to \mathbb R^m; \qquad \theta = (\theta_1,\ldots, \theta_n)\mapsto (f_1(\theta), \ldots, f_m(\theta)) = (x_1,\ldots, x_m)=x,
\end{equation}
for rational functions $f_1,\ldots, f_m\in\mathbb R(\theta_1,\ldots, \theta_n)$ defined on $\Theta$. 
The set $\varphi(\Theta)$ often arises from a problem in applied mathematics, where an understanding of the polynomial constraints defining $\varphi(\Theta)$ can be used to develop methodology for the specific application.
For example, $\varphi(\Theta)\subseteq \mathbb{R}^m$ may be a set of points corresponding to distributions in a statistical model defined on the parameter space $\Theta$ \cite{drton2018algebraic, sullivant2023algebraic}, or it may be the feasible region of an optimization problem \cite{blekherman2012semidefinite, gillis2020nonnegative}. 

It is common that the set $\varphi(\Theta)$ is defined with respect to an underlying combinatorial model, such as a graph, matroid or simplicial complex \cite{sullivant2023algebraic}. 
Consequently, understanding the geometry of $\varphi(\Theta)$ amounts to analyzing the resulting interplay between combinatorics and algebraic geometry. 
Embracing this perspective, this paper presents some general techniques for describing semialgebraic sets $\varphi(\Theta)$ where the coordinate functions $f_i$ of the map $\varphi$ are polynomials with nonnegative integral coefficients. 

The restriction to $f_1,\ldots, f_m\in\mathbb Z_{\geq 0}[\theta_1,\ldots, \theta_n]$ is particularly interesting when the coordinate functions are polynomials enumerating combinatorial objects indexed by the monomials $\theta^\alpha = \theta_1^{\alpha_1}\cdots \theta_n^{\alpha_n}$ supporting $f_i$.
When a monomial $\theta^\alpha$ has nonzero coefficient in multiple $f_i$, this induces relations on the generating polynomials $f_1,\ldots, f_m$ that are captured by the algebraic geometry of $\varphi(\Theta)$. 
This occurs, for instance, for unirational varieties associated to matroids, families of degenerate subvarieties of secant varieties and  statistical models for multivariate data analysis (see Section~\ref{section:prelim}). 

In many cases, $\varphi(\Theta)$ is equal to the intersection of an irreducible algebraic variety $\mathcal{V}_\varphi$ of the same dimension as $\varphi(\Theta)$ with a full-dimensional semialgebraic set $\mathcal{K}\subseteq \mathbb R^m$ whose defining inequalities are known. 
Hence, a major goal is to describe the \emph{vanishing ideal} 
\begin{equation}
        \label{eqn: vanishing ideal}
        I_{\varphi} = \{f\in \R[x_1,\ldots, x_m] : f(x) = 0 \textrm{ for all } x\in\varphi(\Theta)\}.
    \end{equation}
of $\varphi(\Theta)$, which is equal to $\ker(\varphi^\ast)\subseteq \mathbb R[x_1,\ldots, x_m]$, where $\varphi^\ast$ is the algebraic pullback of the map $\varphi$. 
In this case, the irreducible variety $\mathcal{V}_\varphi$ is the Zariski closure of $\varphi(\Theta)$; that is, $\mathcal V_\varphi = \overline{\varphi(\Theta)}$.
Since we are interested in identifying polynomial constraints defining the set $\varphi(\Theta)$, the most basic problem to be solved is the \emph{implicitization problem} from algebraic geometry \cite{cox1997ideals}:
\begin{problem}[The implicitization problem]
    \label{prob: colored implicitization}
    Find a generating set (e.g. a basis) for the ideal $I_{\varphi}$. 
\end{problem}
Combining a solution to the implicitization problem for $\mathcal V_\varphi$ with the inequalities defining $\mathcal K$ provides a complete description of $\varphi(\Theta)$ in terms of polynomial constraints. 

Sometimes the variety $\mathcal V_\varphi$ is a toric variety. 
This is desirable since, up to a change of coordinates, the generating set for $I_\varphi$ sought in Problem~\ref{prob: colored implicitization} will consist of relatively simple polynomials (i.e., binomials), which often have a combinatorial interpretation reflecting the original application defining $\varphi(\Theta)$. 
Examples of this include discrete regular exponential families in statistics \cite{diaconis1998algebraic}, where the binomial generators provide a Markov chain for Monte Carlo approximations of p-values, as well as in theoretical physics \cite{juhnke2025triangulations}, where Gröbner bases for a toric ideal are used to compute integrals relevant to cosmological models. 
There are also numerous examples from algebraic combinatorics where Gröbner bases are computed for families of lattices polytopes, and then used to understand triangulations of the polytopes and related enumerative problems \cite{backman2025regular, lam2007alcoved,sturmfels1996grobner}. 
In applied algebra, the value of the toric structure has resulted in a general pursuit aiming to describe when the vanishing ideal of a semialgebraic set $\varphi(\Theta)$ is toric \cite{coons2023symmetrically, duarte2020equations, duarte2025algebraic, gorgen2022staged, misra2021gaussian, misra2022directed,nicklasson2023toric}. 
This problem can generally be referred to as the \emph{toric reparameterization problem}, since it amounts to identifying a semialgebraic set $\Theta'$ 
and a \emph{monomial} map $\varphi'$ for which $\mathcal V_\varphi = \overline{\varphi(\Theta)}$ and $\mathcal V_{\varphi'} = \overline{\varphi'(\Theta')}$ are isomorphic via a change of coordinates. 
\begin{problem}[The toric reparameterization problem]
    \label{prob: toric}
    Decide if there exists a monomial map $\varphi': \Theta'\to \mathbb R^m$ such that the $\mathcal V_\varphi \simeq \mathcal V_{\varphi'}$.  If it exists, give an explicit formula for the map $\varphi'$. 
\end{problem}
While there are computational methods for the existence question in Problem~\ref{prob: toric} \cite{kahle2025efficiently,maraj2026symmetry}, there is a general lack of techniques that give an explicit formula for the monomial map $\varphi'$. 
Provided with an explicit monomial map, lattice point combinatorics can be applied to extract a basis of the ideal, thereby solving Problem~\ref{prob: colored implicitization} and giving new polynomial constraints on $\varphi(\Theta)$ to be used in applications. 

Since not all $\varphi(\Theta)$ admit a toric reparametrization \cite{nicklasson2023toric}, alternative methods for identifying simple, but useful families of polynomials belonging to $I_\varphi$ have been explored.
For instance, we could seek a basis for the degree-1 component of $I_\varphi$ -- which amounts to computing the linear span of the variety $\mathcal V_\varphi$. 
When $\varphi(\Theta)$ is a statistical model, i.e., a set in which each point corresponds to a distribution, 
a linear constraint $p\in I_\varphi$ can be used as the null hypothesis $H_0: p(x) = 0$ in a test to evaluate if the data-generating distribution belongs to the model $\varphi(\Theta)$ based on observed data $x = (x_1,\ldots, x_m)$ \cite{sturma2024testing}.
While this is true for any polynomial in $I_\varphi$, the linearity assumption often simplifies statistical evaluations of the hypothesis test. 
Hence, there is a need for techniques to address the following problem. 
\begin{problem}
    \label{prob: linear span}
    Identify closed-form expressions for a collection of linear forms $h_1,\ldots, h_s\in\mathbb R[x_1,\ldots, x_m]$ whose zero locus is the linear span of $\mathcal V_\varphi$. 
\end{problem}

Classic techniques using elimination theory \cite{cox1997ideals} exist for computationally solving Problems~\ref{prob: colored implicitization},~\ref{prob: toric} and~\ref{prob: linear span} for $\varphi(\Theta)$ when $n$ and $m$ are relatively small. 
However, in most cases, we are interested in solving these problems for infinite families of semialgebraic sets $\mathcal F= \{\varphi_i(\Theta_i): \Theta_i\in \mathbb R^{n_i}, \varphi: \Theta_i\to \mathbb R^{m_i}, i\in\mathbb{Z}_{\geq0}\}$.  
Hence, the goal is to find closed-formed expressions for the solutions to each of these problems for all $i\in \mathbb{Z}_{\geq 0}$.  
This means that (combinatorial) techniques going beyond computation are required. 

Solutions to Problems~\ref{prob: colored implicitization},~\ref{prob: toric} and~\ref{prob: linear span} identify families of polynomials $\mathcal P_i\subset I_{\varphi_i}$ for $\varphi_i(\Theta_i)\in \mathcal F$.  
In the context of statistics, one common application of the polynomials $\mathcal P_i$ is to distinguish models in the family $\mathcal F$; that is, to show that $\varphi_i(\Theta_i)\neq \varphi_j(\Theta_j)$ for all $i,j\in\mathbb Z_{\geq 0}$. 
Specifically, $p\in \mathcal P_i$ can be used to show that $I_{\varphi_i}\neq I_{\varphi_j}$ by proving that $p\in I_{\varphi_i}$ while $p\notin I_{\varphi_j}$ (see \cite{boege2024colored, hollering2021identifiability}, for example). 
Hence, the results of this paper provide general techniques for addressing the so-called \emph{model distinguishability problem}. 

\begin{problem}[The model distinguishability problem]
\label{prob: colored distinguishability}
For a family of semialgebraic sets $\mathcal F= \{\varphi_i(\Theta_i): \Theta_i\in \mathbb R^{n_i}, \varphi: \Theta_i\to \mathbb R^{m_i}, i\in\mathbb{Z}_{\geq0}\}$ decide when $\varphi_i(\Theta_i)\neq \varphi_j(\Theta_j)$ for $i,j\in\mathbb Z_{\geq 0}$. 
\end{problem}

One technique for addressing the model distinguishability problem is to show that the linear span of $\mathcal V_{\varphi_i}$ is different from that of $\mathcal V_{\varphi_j}$. 
Conversely, we can seek solutions to the \emph{linear equivalence problem}:
\begin{problem}[The linear equivalence problem]
    \label{prob: linear equivalence}
    Characterize when $\varphi_i(\Theta_i),\varphi_j(\Theta_j)\in \mathcal F$ have the same linear span. 
\end{problem}
A solution to the Problem~\ref{prob: linear equivalence} serves as an initial step towards solving Problem~\ref{prob: colored distinguishability}. 
In some cases, a solution to a \emph{weak} formulation of Problems~\ref{prob: colored distinguishability} and~\ref{prob: linear equivalence} suffices. 
The weak formulation of Problems~\ref{prob: colored distinguishability} and~\ref{prob: linear equivalence} asks for a solution up to an invertible linear transformation $\mathbb R^{m_i}\simeq \mathbb R^{m_j}$. 

\subsection*{Our Contributions}

This article introduces a combinatorial technique for addressing Problems~\ref{prob: colored implicitization},~\ref{prob: toric},~\ref{prob: linear span},~\ref{prob: colored distinguishability} and~\ref{prob: linear equivalence} when the rational map $\varphi$ has coordinate functions $f_1,\ldots, f_m\in \mathbb Z_{\geq 0}[\theta_1,\ldots, \theta_n]$.  
The key observation is that the nonnegative integral coefficients induce a useful combinatorial structure on the variety $\mathcal V_\varphi$. 
From the coefficient vectors of the coordinate functions, we construct a poset $P_\varphi$ associated to the map $\varphi$. 
The poset $P_\varphi$ is used to identify a new family of polynomials, called \emph{$P_{\varphi}$-invariants}, belonging to the vanishing ideal $I_\varphi$.  
The $P_{\varphi}$-invariants are generally nonlinear, and they arise from regular functions on the variety of $\mathcal V_\varphi$ obtained via a Möbius inversion formula satisfied by $P_{\varphi}$ (given in Theorem~\ref{thm: mobius}). 

It is shown that the poset $P_\varphi$ parametrizes a canonical extension $\hat\varphi: \Theta \to \mathbb R^{P_\varphi}$ of the map $\varphi$ into a generally higher dimensional ambient space, where combinatorial properties of $P_\varphi$ can be used to solve the above problems. 
When the map $\hat\varphi$ is factored as the composition of a linear map $M_{\hat\varphi}$ and a toric (e.g. monomial) map $\phi(\theta)$, the Möbius inversion formula derived in Theorem~\ref{thm: mobius} yields a combinatorial row reduction operation on $M_{\hat\varphi}$. 
From this, we observe that the linear forms defining the linear span of $\mathcal V_\varphi$ are recovered from the \emph{linear} $P_{\hat{\varphi}}$-invariants, meaning that a closed-form description of the linear span of $\mathcal V_\varphi$ is obtained by computing the Möbius function of $P_\varphi$ (Theorem~\ref{thm: lin independence}). 
Hence, Theorem~\ref{thm: lin independence} reduces Problem~\ref{prob: linear span} to describing the combinatorics of the poset $P_\varphi$. 
As an additional benefit, we also obtain a new combinatorial criterion in terms of $P_\varphi$ for solving the toric reparameterization problem (Theorem~\ref{thm: toric condition}). 
When a basis for the toric ideal of $\phi(\theta)$ is known, Theorems~\ref{thm: lin independence} and~\ref{thm: toric condition} combine to yield a closed-form expression for a basis of $I_\varphi$ (see Section~\ref{rem: toric implicitization}). 
Hence, by describing the combinatorics of the poset $P_\varphi$, we obtain new techniques for solving the implicitization problem (Problem~\ref{prob: colored implicitization}), the toric reparametrization problem (Problem~\ref{prob: toric}) and the linear span problem (Problem~\ref{prob: linear span}). 

To demonstrate the applications of this combinatorial methodology, our running examples include the family of colored DAG models \cite{boege2024colored}, where full solutions to the model distinguishability problem (Problem~\ref{prob: colored distinguishability}) and the linear equivalence problem (Problem~\ref{prob: linear equivalence}) remain unknown.  
In this paper, we demonstrate how these techniques can solve these problems on small examples. 
In a follow-up paper \cite{garrote2026posets}, we use the techniques to settle (all five of) these problems for a subfamily of colored DAG models.

\subsection*{Plan of the Paper}
Section~\ref{section:prelim} presents motivating examples for this work.  
These include varieties parametrized using flats of matroids, degenerate subvarieties of secant varieties, the colored DAG models from statistics, and phylogenetic models. 
Section~\ref{subsec: minimal linear subspace} contains preliminary results on the canonical factorization of a polynomial map as a composition of a linear map $M_\varphi$ and a toric map $\phi(\theta)$, giving a reduction of Problems~\ref{prob: linear span} and~\ref{prob: linear equivalence} to matrix algebra.
Section~\ref{sec: posets} introduces the poset $P_\varphi$, derives the Möbius inversion formula (Theorem~\ref{thm: mobius}), and defines the family of $P_{\varphi}$-invariants (Definition~\ref{def: poset invariants}).
Subsection~\ref{subsec: the ideal of the poset} describes the poset parametrization $\hat\varphi: \Theta \to \mathbb R^{P_\varphi}$ extending $\varphi$, and shows that the vanishing ideal $I_\varphi$ is an elimination ideal of $I_{\hat{\varphi}}$. 
Subsection~\ref{subsec: combinatorial linear equiv} describes how the poset $P_\varphi$ gives a combinatorial model for solving Problem~\ref{prob: linear span} in terms of linear $P_{\hat\varphi}$-invariants.  
Subsections~\ref{subsec: combinatorial toric} and~\ref{rem: toric implicitization} show how $P_\varphi$ gives solutions to the toric reparametrization and implicitization problems, respectively.

\section{Motivating examples}\label{section:prelim} 
The assumption in this paper that the coordinate functions in~\eqref{eqn: rational map} satisfy $f_i\in \mathbb Z_{\geq 0}[\theta_1,\ldots, \theta_n]$ amounts to saying that $f_1,\ldots, f_m$ may be enumerating combinatorial objects corresponding to the monomials $\theta^\alpha = \theta_1^{\alpha_1}\cdots \theta_n^{\alpha_n}$ supported on $f_1,\ldots, f_m$. 
When some objects are enumerated by multiple $f_i$ (i.e., $\theta^\alpha$ appears in multiple $f_i$), there will be relations between the generating functions $f_1,\ldots, f_m$, which are reflected in the algebraic geometry of $\varphi(\Theta)$ and its variety $\mathcal V_\varphi = \overline{\varphi(\Theta)}$. 
The techniques developed in this paper aim to clarify these relations, and the associated algebraic geometry. 

This section presents examples from combinatorics, algebraic geometry and applications to serve as our running examples. 
The example in Section~\ref{subsec: matroids} is combinatorial, with the coordinate functions $f_i$ enumerating elements in flats of matroids. 
These \emph{matroid flat varieties} generalize the toric varieties of the well-studied \emph{edge polytopes} \cite{herzog2018binomial}. 
Section~\ref{subsec: secants} considers subvarieties of general determinantal varieties, which are popular examples of secant varieties in algebraic geometry, 
The example in Section~\ref{subsec: models} is both combinatorial and statistical, with coordinate functions $f_i$ enumerating directed graphs in order to parameterize certain statistical models. 
The fourth example, in Section~\ref{subsec: phylogenetics}, arises in phylogenetics. For this example, the techniques in this paper systematically give simple, combinatorial proofs of classic results celebrated in the field.

\subsection{Matroid flat varieties}
\label{subsec: matroids}
Let $M$ be a matroid with ground set $E \subset 2^V$ for a finite set $V$, where $2^V$ denotes the power set of $V$.  
Let $\mathfrak F$ denote the collection of flats of $M$. 
To each $i\in V$, we associate a real-valued parameter $\theta_i$, and for each flat $F\in \mathfrak F$ we define the multivariate generating polynomial
\[
f_F(\theta) = \sum_{e\in F}\prod_{i\in e}\theta_i. 
\]
enumerating the elements belonging to $F$.  
For a collection of nonempty flats $\mathfrak F'\subseteq \mathfrak F$, we obtain a rational map 
\[
\varphi_{M,\mathfrak F'}: \mathbb R^V\to \mathbb R^{\mathfrak F'}; \qquad \theta \mapsto (f_F(\theta))_{F\in \mathfrak F'}. 
\]
The image of $\varphi_{M,\mathfrak F'}$ is an irreducible algebraic variety $\mathcal V_{\varphi_{M,\mathfrak F'}} = \varphi_{M,\mathfrak F'}(\mathbb R^V)$ that we call the \emph{matroid flat variety} for $\mathfrak F'$.  
Since the same element $e\in E$ may appear in multiple flats $F, F'\in \mathfrak F'$, the geometry of $\mathcal V_{\varphi_{M,\mathfrak F'}}$ reflects combinatorial relations among the collection of generating functions $\{f_F : F\in\mathfrak F'\}$.  

A simple example of this construction is when $M$ is a graphic matroid for an undirected graph $G = (V,E)$. 
In this case the edge set $E\subset 2^V$ is the ground set of $M$.  
It is not hard to see that the relations on the generating functions $\{f_F : F\in\mathfrak F'\}$, and hence on the variety $\mathcal V_{\varphi_{M,\mathfrak F'}}$, include relations arising from the subposet of the lattice of flats of $M$ defined by $\mathfrak F'$.  
This connection will be made explicit in Section~\ref{subsec: poset invariants}.  
In the case of graphic matroids, we will see that the varieties $\mathcal V_{\varphi_{M,\mathfrak F'}}$ generalize the (toric) varieties associated to edge polytopes (Example~\ref{ex: toric graphic matroids}), which are well-studied in algebraic and geometric combinatorics \cite[Chapter 5]{herzog2018binomial}.

\subsection{Degenerate subvarieties of general determinantal varieties}
\label{subsec: secants}

Consider the $r$-th secant variety of rank $1$ matrices defined as the image of the following map
\[
\varphi^r_{m,n} : (\mathbb R^m \times\mathbb R^m)^r \to \mathbb R^m  \otimes \mathbb R^n; \qquad (s^{(1)}, t^{(1)}),\ldots,(s^{(r)}, t^{(r)}) \mapsto (\sum_{w=1}^r s^{(w)}_i t^{(w)}_j)_{ij} = (x_{ij}).
\]
Geometrically, the Zariski closure of the image of $\varphi^r_{m,n}$, e.g., the variety $\mathcal V_{\varphi^r_{m,n}}$, consists of all (real) $(m\times n)$ matrices of rank at most $r$.  
Hence, $\mathcal V_{\varphi^r_{m,n}}$ is known as the \emph{general determinantal variety of rank $r$} in algebraic geometry. 

General determinantal varieties (and their extensions to tensors) are widely used, where connections to nonnegative matrix factorizations have value in high-dimensional data analysis \cite{gillis2020nonnegative}. 
In these applications, subvarieties obtained by forcing the parameters $(s^{(1)}, t^{(1)}),\ldots,(s^{(r)}, t^{(r)})$ to satisfy polynomial equations are often used \cite{fu2018identifiability}.
While $\mathcal V_{\varphi^r_{m,n}}$ is nondegenerate, i.e., its linear span is equal to $\mathbb R^m  \otimes \mathbb R^n$, such subvarieties can live in proper linear subspaces of $\mathbb R^m  \otimes \mathbb R^n$. 
Thus, as our second example, we consider subvarieties of $\mathcal V_{\varphi^r_{m,n}}$ obtained by imposing additional linear constraints on the parameters.

For example, for $r=2$, let us choose subsets of indices $U\subseteq [m]$ and $V\subseteq [n]$ and impose the constraints $s^{(1)}_i = s^{(1)}_j$ for $i, j\in U$ and $t^{(2)}_i=t^{(2)}_j$ for $i, j\in V$. 
When $|U|, |V|\geq 2$ there will be linear polynomials in the vanishing ideal. 
In the most restrictive case when $U=[m]$ and $V = [n]$, the variety will be toric. 
These observations will be made (combinatorially) apparent in what follows.

\subsection{Colored DAG models}
\label{subsec: models}

Colored DAG models are semialgebraic sets that arise in the statistical analysis of multivariate data. 
They constitute a combinatorially rich family of semialgebraic sets with applications. 
The fundamental combinatorial structures used to define colored DAG models are special  subgraphs of a DAG $\GG$ known as treks. 
In fact, colored DAG models are parametrized by polynomials that enumerate treks in $\GG$.  

Let $\GG = ([m], E)$ be a simple directed acyclic graph (DAG) with node set $[m] = \{1,\ldots, m\}$ and edge set $E$. 
A \emph{directed path} (from $v_1$ to $v_k$) in $\GG$  is a sequence of nodes $P = (v_1,\ldots, v_k)$ such that $v_i\to v_{i+1}\in E$ for all $i\in[k-1]$. 
For convenience, we will identify the directed path $P$ with its set of edges; i.e., $P =\{v_i \to v_{i+1} : i\in[k-1]\}$. 

\begin{definition}[Treks]
Let $P=\{v_i\to v_{i+1} : i\in[k-1]\}$ and $Q = \{w_i \to w_{i+1} : i\in[\ell -1]\}$ be directed paths in a DAG $\GG$.
The pair $T = \{P,Q\}$ is called a \emph{trek} (between $v_k$ and $w_\ell$ in $\GG$) if the initial vertices of $P$ and $Q$ coincide, i.e. if $v_1=w_1$:
 \[
w_\ell \leftarrow w_{\ell - 1} \leftarrow \cdots \leftarrow  w_2 \leftarrow w_1=v_1\to v_2 \to \cdots \to v_{k-1} \to v_k.
\]
The node $v_1$ is called the \emph{top} of $T$, denoted $\textrm{top}(T)$. 
For $i,j\in[m]$ we let $\mathcal{T}(i,j)$ denote the set of all treks between $i$ and $j$ in $\GG$.   
A trek $T=\{P,Q\}\in \mathcal{T}(i,j)$ is called \emph{simple} if $P$ and $Q$ are vertex-disjoint except for $\textrm{top}(T)$, i.e., $V(P) \cap V(Q) = \{\mathrm{top}(T)\}$. 
\end{definition}

The combinatorics of treks in DAGs is enriched by assigning colors to the edges and vertices, producing the larger family of colored DAGs. 

\begin{definition}[Colored DAG]
\label{def: coloring}
A \emph{coloring} of a DAG $\GG = ([m], E)$ is a pair of surjective maps $c_{n,e} = (c_n: [m]\to [n],c_e: E \to [e])$ for positive integers $n$ and $e$. 
When $n$ and $e$ are understood, we simply write $c$ for $c_{n,e}$, $c_n$ or $c_e$. 
The pair $(\GG, c)$ is called a \emph{colored DAG}.    
\end{definition}

Two important colorings correspond to when the maps $c_n$ and $c_e$ are both bijective (i.e., $n = m$ and $e =|E|$) or when they are both constant (i.e., $n = e = 1$). 

\begin{definition}[The constant coloring and the uncoloring]
\label{def: constant coloring}
Let $\GG = ([m], E)$ be a DAG. 
\begin{enumerate}
    \item The coloring $c^\ast =c_{1,1}$ is called the \emph{constant coloring} of $\GG$.
    \item The coloring $c^\circ = c_{m, |E|}$ is called the \emph{uncoloring} of $\GG$. 
\end{enumerate}
\end{definition}

Every coloring $c$ of a DAG $\GG$ coarsens the constant coloring and refines the uncoloring. 
The colors can be interpreted as identifying vertices and edges with one another, meaning that the same trek (up to a graph isomorphism invariant under the coloring) may appear several times in a colored DAG $(\GG,c)$. 
The treks between nodes $i$ and $j$ in a colored DAG $(\GG, c)$ are enumerated using a multivariate generating polynomial known as a trek polynomial. 
\begin{definition}[Trek polynomial]
\label{def: trek polynomial}
Let $(\GG, c)$ be a colored DAG with $\GG = ([m],E)$ and $c = c_{n, e}$.  
Assign a variable $\omega_k$ to every $k\in[n]$ and a variable $\lambda_{k}$ to every $k\in [e]$. 
Given a trek $T=\{P,Q\}\in \mathcal{T}(i,j)$ we define the \emph{trek monomial}
\[
m_T = \omega_{c(\textrm{top}(T))}\prod_{k\to \ell\in P}\lambda_{c(k\to\ell)}\prod_{k\to \ell\in Q}\lambda_{c(k\to\ell)}. 
\]
The \emph{trek polynomial} for $i,j\in[m]$ is the sum over the trek monomials for all treks in $\mathcal{T}(i,j)$: 
\begin{equation}
\label{eqn: trek polynomial}
p_{i,j}^{(\GG,c)} = \sum_{T\in\mathcal{T}(i,j)}m_T \in \mathbb{Z}_{\geq 0}[\omega_1,\ldots, \omega_n, \lambda_1,\ldots, \lambda_e]. 
\end{equation}   
When the colored DAG $(\GG,c)$ is clear from context, we write $p_{i,j}$ for $p_{i,j}^{(\GG,c)}$. 
We further let 
$
\mathcal{T} = \{m_T : T\in\mathcal{T}(i,j), \, i,j\in[m]\}
$
denote the set of all trek monomials arising from $(\GG,c)$. 
\end{definition}

For any coloring $c$, evaluating $p_{i,j}^{(\GG,c)}$ at the all 1s vector yields the number of treks between $i$ and $j$ in $\GG$. 
Depending on the coloring, the coefficients of $p_{i,j}^{(\GG,c)}$ reveal other combinatorial data. 
For instance, $p_{i,j}^{(\GG,c^\ast)} = \sum_{k\geq 0}\alpha_k\omega\lambda^k\in\mathbb{Z}_{\geq 0}[\omega, \lambda]$ where $\alpha_k$ is the number of treks $T = \{P,Q\}$ between $i$ and $j$ of \emph{length} $k$; i.e., $k = |P| + |Q|$.  The constant coloring also satisfies the following helpful property. 

\begin{lemma}
    \label{lem: degree closure}
    Consider the constantly colored DAG $(\GG, c^\ast)$ for $\GG = ([m], E)$. 
    The polynomial $\sum_{1\leq i\leq j\leq m}p_{i,j}^{(\GG,c^\ast)}$ has support $\mathcal{T} =\{\omega, \omega\lambda, \ldots, \omega\lambda^N\}$, where $N$ is the length of the longest trek in $G$. 
\end{lemma}

\begin{proof} 
    It suffices to show that $\omega\lambda^{s-1}\in\mathcal{T}$ whenever $\omega\lambda^s\in\mathcal{T}$ and $s\geq 1$. 
    If $\omega\lambda^s\in\mathcal{T}$, then there exists a trek $T =\{P,Q\}$ in $\GG$ composed of two directed paths $P=\{v_i\to v_{i+1} : i\in[k-1]\}$ and $Q = \{w_i \to w_{i+1} : i\in[\ell -1]\}$ with $v_1 = w_1$.  
    Since $|P| + |Q| = s \geq 1$, then either $k > 1$ or $\ell > 1$.
    Without loss of generality, we assume $k > 1$, in which case $m_{T'}\in \mathcal{T}$, where $T' = \{P\setminus\{v_{k-1}\to v_k\}, Q\}$. 
\end{proof}

Trek polynomials are used to define semialgebraic sets associated to colored DAGs. 
These sets are known as colored DAG models \cite{boege2024colored}. 
 
\begin{definition}[Colored DAG model]
\label{def: colored DAG model}
For a colored DAG $(\GG,c)$, we define the polynomial map
\begin{equation}
\label{eqn: trek map}
\begin{split}
\varphi_{\GG,c}&: \mathbb{R}^{e}\times \mathbb{R}^n_{>0} \to \mathbb{R}^{m\times m}; \\
\varphi_{\GG,c}&:(\lambda_1,\ldots, \lambda_e,\omega_1,\ldots, \omega_n)\mapsto \left(p_{i,j}^{(\GG, c)}(\lambda_1,\ldots, \lambda_e,\omega_1,\ldots, \omega_n)\right)_{i=1, j=1}^m.
\end{split}
\end{equation}
The image $\mathcal{M}(\GG,c) = \varphi_{\GG,c}\left(\mathbb{R}^{e}\times \mathbb{R}^n_{>0}\right)$, is called the \emph{colored (Gaussian) DAG model} for $(G,c)$. 
\end{definition}
Since $\mathcal{T}(i,j) = \mathcal{T}(j,i)$ for all $i,j\in[m]$, 
any matrix $\Sigma= (\sigma_{ij})_{i=1,j=1}^m\in\mathcal{M}(\GG,c)$ is symmetric (and, in fact, positive definite). 
Hence it can be described via $\binom{m+1}{2}$ variables $\sigma_{ij}$, for $1\leq i \leq j \leq m$. 
So we may identify the image of $\varphi_{\GG,c}$ with a subset of $\R^{\binom{m+1}{2}}$.

\begin{remark}[Statistical interpretation of a colored DAG model]
\label{rem: statistical interpretation}
In statistics, the parameter value $\lambda_{ij}$ is interpreted as the \emph{direct effect} of a stochastic variable $X_i$ on $X_j$ and $\omega_i$ is the variance of an error term associated to the variable $X_i$ (see \cite{drton2018algebraic}).  
The coloring $c(i\to j) = c(k\to \ell)$ captures when two direct effects are the same and $c(i) =c(j)$ encodes when two variables have the same error variances. 
The set $\mathcal{M}(\GG,c)$ is the collection of covariance matrices for a linear Gaussian model where dependencies between the stochastic variables $X_1,\ldots, X_m$ are governed by the parameter equality constraints specified by $c$ and the edge structure of the DAG $\GG$ \cite{boege2024colored}. 
The models $\mathcal{M}(\GG,c)$ are used in the field of causality \cite{pearl2009causality} to capture the cause-effect structure $\GG$ guiding the interactions between the individual variables $X_1,\ldots, X_m$ and the homogeneity relations $c$ on the model parameters \cite{boege2024colored,peters2014identifiability,wu2023partial}. 
\end{remark}

The colored DAG models $\mathcal{M}(\GG,c)$ and $\mathcal{M}(\HH,c')$ can in fact be equal, even if $\GG\neq \HH$ or $c\neq c'$.  
If $\mathcal{M}(\GG,c)= \mathcal{M}(\HH,c')$ but $i\to j$ is an edge in $\GG$ that is not in $\HH$, then interpreting the parameter $\lambda_{ij}$ as a direct effect of $X_i$ on $X_j$ (as in Remark~\ref{rem: statistical interpretation}) can lead to faulty conclusions in real data studies, since the graph $\HH$ lacks this edge while still equaling representing the distribution. 
Hence, for statistical applications, it is important to characterize when $\mathcal{M}(\GG,c)\neq\mathcal{M}(\HH,c')$ \cite{boege2024colored}; i.e., it is important to solve the model distinguishability problem (Problem~\ref{prob: colored distinguishability}) for the family 
\begin{equation}
\label{eqn: colored DAG family}
\mathcal F = \{\mathcal M(\GG,c)=\varphi_{\GG,c}(\mathbb R^{e}\times \mathbb R^{n}_{>0}) : (\GG,c) \textrm{ a colored DAG}\}.
\end{equation}
One technique for approaching Problem~\ref{prob: colored distinguishability} is to solve the implicitization problem for the vanishing ideal $I_{\varphi_{\GG,c}}$ for every colored DAG $(\GG,c)$.  
However, the implicitization problem (Problem~\ref{prob: colored implicitization}), and the related toric reparametrization problem (Problem~\ref{prob: toric}), are both wide-open for the ideals $I_{\GG,c}$. 
Since $\mathcal{M}(\GG,c)$ equals the intersection of the zero locus of the vanishing ideal $I_{\GG,c}$ with the positive definite cone \cite{boege2024colored}, it follows that the model distinguishability problem can be settled by clarifying when $I_{\GG,c}\neq I_{\HH,c'}$. 

One way to show $I_{\GG,c}\neq I_{\HH,c'}$ is to find a polynomial $f\in \R[\sigma_{ij} : 1\leq i\leq j\leq m]$ such that $f\in I_{\GG, c}$ but $f\notin I_{\HH, c'}$. 
This technique does not require a complete solution to Problem~\ref{prob: colored implicitization} in order to solve Problem~\ref{prob: colored distinguishability}, but only a single polynomial $f$.
For example, $f$ could be a linear form. 
While the linear forms satisfied by a graphical model are generally of statistical value \cite{wille2004sparse,magwene2004estimating}, they are not well-understood for colored DAG models. 
In particular, Problems~\ref{prob: linear span} and~\ref{prob: linear equivalence} are also wide-open for the family of colored DAG models.

\subsection{Varieties of phylogenetic models}
\label{subsec: phylogenetics}

Phylogenetic models describe the evolution of discrete genetic states, such as DNA or protein sequences, as Markov processes operating along the branches of a tree or network. In nucleotide substitution models like the Jukes-Cantor model, which assume equal base frequencies and uniform mutation rates, the joint probability distribution of the states observed at the leaves is parameterized by transition matrices assigned to the tree's edges.
A classic problem in algebraic statistics is to find the \emph{phylogenetic invariants}; that is, the polynomials that vanish on the expected pattern frequencies of a specific tree topology \cite{sullivant2023algebraic}.
Linear invariants are particularly valuable since they are robust to rate heterogeneity across the sequences and can be used to distinguish topologies even in the presence of mixtures of trees \cite{lake1987, Casanellas2016, Casanellas2024}.

\begin{figure}[t]
    \centering
    \begin{tikzpicture}[scale=1.2, thick]
        \node[circle, fill=black, inner sep=1.5pt] (u) at (0,0) {};
        \node[circle, fill=black, inner sep=1.5pt] (v) at (1,0) {};
        
        \node (l1) at (-1, 0.7) {1};
        \node (l2) at (-1, -0.7) {2};
        \node (l3) at (2, 0.7) {3};
        \node (l4) at (2, -0.7) {4};
        
        \draw (l1) -- (u) node[midway, above] {$e_1$};
        \draw (l2) -- (u) node[midway, below] {$e_2$};
        \draw (l3) -- (v) node[midway, above] {$e_3$};
        \draw (l4) -- (v) node[midway, below] {$e_4$};
        \draw (u) -- (v) node[midway, above, inner sep=4pt] {$e_5$};
    \end{tikzpicture}
    \caption{An unrooted 4-leaf phylogenetic tree with four pendant edges and one internal edge.}
    \label{fig: 4_leaf_tree}
\end{figure}
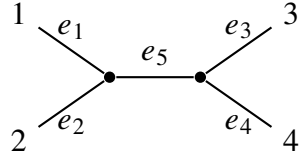

Consider the Jukes-Cantor model on the unrooted 4-leaf phylogenetic tree depicted in Figure~\ref{fig: 4_leaf_tree}. 
Let $\Sigma = \{\textrm{A, C, G, T}\}$ denote the set of four nucleotide states. Under this model, the transition probabilities along each edge $e_i$ for $i \in \{1, \dots, 5\}$ are symmetric, and are parameterized by $a_i$ (the probability of no mutation) and $b_i$ (the probability of a specific mutation), where $a_i+3b_i=1$. These probabilities are collected into a $4 \times 4$ transition matrix $M^{(i)}$ for edge $e_i$ where $M^{(i)}_{y, z} = a_i$ if $y = z$, and $M^{(i)}_{y, z} = b_i$ if $y \neq z$.
Since there are 4 leaves and 4 possible nucleotide states $\{\textrm{A,C,G,T}\}$, there are $4^4 = 256$ possible site patterns, yielding an initial polynomial map $\varphi: \Theta \to \mathbb{R}^{256}$. 

The coordinate function for a specific leaf observation $x = (x_1, x_2, x_3, x_4)$ is obtained by marginalizing over all possible unobserved states $x_u, x_v$ at the internal nodes $u$ and $v$. Assuming a uniform root distribution, the parameterization is:
\begin{equation}
    \varphi_{x}(\theta) = \frac{1}{4} \sum_{x_u} \sum_{x_v} M^{(1)}_{x_u, x_1} M^{(2)}_{x_u, x_2} M^{(5)}_{x_u, x_v} M^{(3)}_{x_v, x_3} M^{(4)}_{x_v, x_4}.
\end{equation}

However, due to the symmetries inherent in the Jukes-Cantor model, many pattern probabilities are strictly equal (e.g., $p_{\texttt{AACC}} = p_{\texttt{CCGG}}$). Grouping these symmetric patterns reduces the parameterization to 15 distinct coordinate polynomials, giving a reduced map $\varphi_{JC}: \Theta \to \mathbb{R}^{15}$.

By factoring out the uniform root distribution, the coordinate functions of this reduced map are polynomials in $\mathbb{Z}_{\geq 0}[a_1, \ldots, a_5, b_1, \ldots, b_5]$. For example, the expected frequency of observing the pattern $\texttt{ACTT}$ (where leaves 1 and 2 observe distinct states, and leaves 3 and 4 share a third state) is governed by the polynomial:
\begin{equation}
    \begin{split}
    p_{\texttt{ACTT}} =\ &a_1 a_3 a_4 b_2 b_5 + a_1 a_5 b_2 b_3 b_4 + 2 a_1 b_2 b_3 b_4 b_5 + a_2 a_3 a_4 b_1 b_5 + a_2 a_5 b_1 b_3 b_4 \\
    &+ 2 a_2 b_1 b_3 b_4 b_5 + a_3 a_4 a_5 b_1 b_2 + a_3 a_4 b_1 b_2 b_5 + a_5 b_1 b_2 b_3 b_4 + 5 b_1 b_2 b_3 b_4 b_5.
    \end{split}
\end{equation}

The Jukes-Cantor model $\varphi_{JC}(\Theta)$ has linear span defined by a pair of linear forms called the \emph{Lake invariants}, which are well-known in phylogenetics.  
A solution to the toric reparametrization problem for $\varphi_{JC}(\Theta)$ is also known, where the reparametrization is deduced using the discrete Fourier transform -- a celebrated technique for toric reparametrization problems for group-based phylgenetic models. 

Since the coordinate functions of $\varphi_{JC}$ have only nonnegative integer coefficients, this phylogenetic model precisely satisfies the conditions of our combinatorial framework. 
In the following, we show that our purely combinatorial approach via posets systematically recovers both the Lake invariants and a toric reparametrization of $\varphi_{JC}(\Theta)$ that does not require the discrete Fourier transform machinery.

\section{Preliminary results on the linear span of \texorpdfstring{\(\mathcal{V}_{\varphi}\)}{}}
\label{subsec: minimal linear subspace}

We now provide some preliminary results regarding the linear span and linear equivalence problems (Problems~\ref{prob: linear span} and~\ref{prob: linear equivalence}). 
This section recalls the canonical factorization of a polynomial map $\varphi: \Theta\to \mathbb R^m$ into the composition of a linear $M_\varphi$ and a toric map $\phi(\theta)$. 
The key tool we obtain from this factorization is Theorem~\ref{thm: min lin subspace}, which describes the linear span of $\varphi(\Theta)$ as the kernel of the matrix $M_\varphi^t$, where the supscript denotes matrix transposition.  
When the coordinate functions $f_1,\ldots, f_m$ of the map $\varphi$ satisfy $f_1,\ldots, f_m\in\mathbb Z_{\geq 0}[\theta_1,\ldots, \theta_n]$, this reduces Problem~\ref{prob: linear equivalence} to showing that a pair of combinatorially-structured matrices have the same kernel (Corollary~\ref{cor: lin equiv reduction}). 
The combinatorial structure of these matrices is the topic of Section~\ref{sec: posets}, where the computation of $\ker(M_\varphi^t)$ is further reduced to describing properties of a certain poset. 

In the following, let $\Theta\subseteq \mathbb R^n$ contain an open subset of $\mathbb R^n$ and 
\[
\varphi: \Theta\to \mathbb R^m; \qquad \theta \mapsto (f_1(\theta),\ldots, f_m(\theta)),
\]
where $f_1,\ldots, f_m\in\mathbb Z_{\geq 0}[\theta_1,\ldots, \theta_n]$. We consider the set of all monomials in the variables $\theta_1,\ldots, \theta_n$ that have a nonzero coefficient in at least one of $f_1,\ldots, f_m$:
\begin{equation}
\label{eqn: support set}
\mathcal T = \{\theta^\alpha = \theta_1^{\alpha_1}\cdots\theta_n^{\alpha_n} : f_i(\theta) = \sum_{\beta\in\mathbb Z_{\geq 0}^n}c_\beta\theta^\beta \textrm{ with } c_\alpha\neq 0 \textrm{ for some } i\in[m]\}. 
\end{equation}
Let $\R^{\mathcal{T}}$ denote the real Euclidean space with one standard basis vector for each monomial $\tau\in \mathcal{T}$, and let $\mathbb{T} = (\tau : \tau \in \mathcal{T})$ denote the column vector of these monomials. 
\begin{definition}[The matrix $M_\varphi$]
For the map $\varphi: \Theta\to \mathbb R^m$ where $\theta \mapsto (f_1(\theta),\ldots, f_m(\theta))$ with $f_1,\ldots, f_m\in\mathbb Z_{\geq 0}[\theta_1,\ldots, \theta_n]$ supported on the monomials in $\mathcal T$ from~\eqref{eqn: support set}, let $c_i = (c_{i,\alpha} : \theta^\alpha \in \mathcal T)$ denote the coefficient vector of $f_i$ in $\mathbb R^\mathcal T$.  Let $M_\varphi\in \mathbb Z^{m\times \mathcal T}_{\geq 0}$ denote the matrix whose rows are $c_1,\ldots, c_m$.  
\end{definition}
Given $\theta\in \Theta$, we have that 
\begin{equation}
\label{eqn: map decomp}
\varphi(\theta) = M_{\varphi}\mathbb{T}.
\end{equation}

The formula in~\eqref{eqn: map decomp} is a standard decomposition of the polynomial map $\varphi$ into a linear part, given by $M_{\varphi}$, and a nonlinear part 
\begin{equation}
    \label{eqn: nonlinear part}
    \phi_{\mathcal{T}}: \Theta \to \R^{\mathcal{T}}; \qquad \phi_{\mathcal{T}}(\theta) = \mathbb{T}.
\end{equation} 
Let $I_{\mathcal{T}} = \ker(\phi_{\mathcal{T}}^\ast)$, where 
\begin{equation}
\label{eqn: nonlinear part dual}
\phi_{\mathcal{T}}^\ast: \R[\rho_\tau: \tau\in\mathcal{T}] \to \R[\theta_1,\ldots, \theta_n]; \qquad 
\phi_{\mathcal{T}}^\ast: \rho_\tau \mapsto \tau
\end{equation} 
denotes the algebraic pullback of the map $\phi_{\mathcal{T}}$, 
and we let $\mathcal V_{\mathcal{T}}$ denote the associated affine toric variety; i.e., the zero locus of the ideal $I_{\mathcal{T}}$. 
We then we have the following result. 
\begin{theorem}
    \label{thm: min lin subspace}
    Let $\Theta$ contain an open subset of $\mathbb R^n$ and 
    $\varphi: \Theta\to \mathbb R^m$ be a polynomial map where $\theta \mapsto (f_1(\theta),\ldots, f_m(\theta))$ with $f_1,\ldots, f_m\in\mathbb Z_{\geq 0}[\theta_1,\ldots, \theta_n]$ supported on the monomials in $\mathcal T$. 
    If the toric ideal $I_{\mathcal{T}}$ is homogeneous with respect to the standard grading then the linear span of $\varphi(\Theta)$ is the zero locus of
    \[
    \left\{\sum_{i=1}^m\alpha_{i}x_i : (\alpha_1,\ldots, \alpha_m)\in \ker(M_{\varphi}^t)\right\}. 
    \]
    In particular, a basis for $\ker(M_{\varphi}^t)$ defines the linear span of $\varphi(\Theta)$. 
\end{theorem}

\begin{proof}
    Suppose that $f = \sum_{i=1}^m\alpha_{i}x_i$ is a linear form vanishing on $\varphi(\Theta)$. 
    Then $f\in I_{\varphi}$, which implies
    \[
    \begin{split}
    0 = f(\varphi(\theta)) = \sum_{i=1}^m\alpha_{i} \varphi(\theta)_{i} = \alpha^tM_{\varphi}\phi_{\mathcal{T}}(\theta),
    \end{split}
    \]
    for all $\theta \in\Theta$,
    where $\varphi(\theta)_{i}$ denotes the $i$-th entry of the vector produced by $\varphi$, and $\phi_{\mathcal{T}}$ is the map defined in~\eqref{eqn: nonlinear part}. 
    It follows that the vector $M_{\varphi}^t\alpha$ is orthogonal to $\phi_{\mathcal{T}}(\theta)$ for all $\theta\in \Theta$. 
    Since $\Theta$ contains an open subset of $\mathbb R^n$ then $\phi_{\mathcal{T}}(\Theta)$ is Zariski dense in $\mathcal V_{\mathcal{T}}$, and the points in $\phi_{\mathcal{T}}(\Theta)$ will linearly span the entire ambient space $\R^{\mathcal{T}}$ whenever $\mathcal V_{\mathcal{T}}$ is nondegenerate (i.e., not contained in any proper linear subspace).
    Since the only vector orthogonal to a spanning set is the zero vector, showing that $\mathcal V_{\mathcal{T}}$ is nondegenerate is sufficient to prove that $M_{\varphi}^t\alpha = 0$; in other words, that $\alpha \in \ker(M_{\varphi}^t)$.
    Since $\mathcal V_{\mathcal{T}}$ is the toric variety of a homogeneous toric ideal under the standard grading defined by a monomial map in which all monomials are distinct, it follows that it is nondegenerate. 
\end{proof}

The following corollary to Theorem~\ref{thm: min lin subspace} is immediate, and it reduces Problem~\ref{prob: linear equivalence} to matrix algebra. 
\begin{corollary}
    \label{cor: lin equiv reduction}
    Let $\varphi(\Theta)$ and $\varphi'(\Theta')$ be two semialgebraic sets fulfilling the conditions from Theorem~\ref{thm: min lin subspace}.  Then $\varphi(\Theta)$ and $\varphi'(\Theta')$ are linearly equivalent if and only if $\ker(M_{\varphi}^t) = \ker(M_{\varphi'}^t)$. 
\end{corollary}

In the remainder of this section, we give more detailed results for two of our running examples from Section~\ref{section:prelim}; i.e., matroid flat varieties and colored DAG models. 

\subsection{The matrix \texorpdfstring{\(M_\varphi\)}{} for matroid flat varieties}
\label{subsec: ideal varieties}
Let $M$ be a matroid with ground set $E\subset 2^V$ for some finite set $V$, and consider the matroid flat varieties $\mathcal V_{M, \mathcal F'}$ defined in Section~\ref{subsec: matroids}. 
Recall that for elements $p_1,\ldots, p_s$ in a poset $P$ with partial order $\preceq$, the \emph{(order) ideal} generated by $p_1,\ldots, p_s$ is the set 
\[
\langle p_1,\ldots, p_s\rangle = \{s\in P : s\preceq p_i \textrm{ for some } i\in [s]\}.
\]
If $P_M$ is the lattice of flats of $M$, we may define the following family of semialgebraic sets associated to $M$, which we call the \emph{ideal varieties} of $M$:
\begin{equation}
\label{eqn: antichain family}
\mathcal F_{M, ideal} = \{\mathcal V_{\varphi_{M, \mathfrak F'}} : \mathfrak F' = \langle F_1,\ldots, F_s\rangle \textrm{ for } F_1,\ldots, F_s \textrm{ an antichain in } P_M\}. 
\end{equation}
For every ideal $\mathfrak F'$ of $P_M$ the variety $\mathcal V_{\varphi_{M, \mathfrak F'}}$ belongs to the family $\mathcal F_{M, ideal}$ since every ideal in a poset is determined by a (unique) antichain. 
For $\mathcal V_{\varphi_{M, \mathfrak F'}}\in \mathcal F_{M, ideal}$ the matrix $M_{\varphi_{M, \mathfrak F'}}$ has rows indexed by $F\in \mathfrak F'$. 
Let $b_F\in \mathbb R^{\mathfrak F'}$ denote the standard basis vectors of $\mathbb R^{\mathfrak F'}$. 
The rows of $M_{\varphi_{M, \mathfrak F'}}$ indexed by singletons $\{e\}\subset \mathfrak F' = \langle F_1,\ldots, F_s\rangle$ are the standard basis vectors $b_{\{e\}}$ for $e\in F_i$ for some $i\in[s]$. 
Since $P_M$ is a geometric lattice, every other row in $M_{\varphi_{M, \mathfrak F'}}$, corresponding to an $F\in \mathfrak F'$, is the sum of the rows indexed by the singletons $\{e\}\preceq F$ in $P_M$.  
In particular, the vectors $b_F - \sum_{\{e\}\preceq F}b_{\{e\}}\in \mathbb R^{\mathfrak F'}$ for $F\in \mathfrak F'$ span the kernel of $M_{\varphi_{M, \mathfrak F'}}$, and the dimension of $\ker(M_{\varphi_{M, \mathfrak F'}})$ equals $|\bigcup_{i=1}^sF_i|$. 
From this, we can recover the following sufficient condition for (weak) linear equivalence of the ideal varieties of $M$. 
\begin{proposition}
    The ideal varieties of $\mathfrak F' = \langle F_1,\ldots, F_s\rangle$ and $\mathfrak F'' = \langle G_1,\ldots, G_t\rangle$ in the poset $P_M$ are linearly equivalent up to a linear transformation of $\mathbb R^{\mathfrak F'}$ if the ideals $\mathfrak F'$ and $\mathfrak F''$ are isomorphic as subposets of $P_M$.   
\end{proposition}

\begin{proof}
    If $\mathfrak F'$ and $\mathfrak F''$ are isomorphic subposets of $P_M$ then the isomorphism induces an invertible linear transformation between $\mathbb R^{\mathfrak F'}$ and $\mathbb R^{\mathfrak F''}$ that takes the kernel vectors described above to their counterparts in $\mathbb R^{\mathfrak F''}$.  
    Hence, up to this linear change of coordinates, the varieties span the same linear subspace of $\mathbb R^{\mathfrak F'}\simeq \mathbb R^{\mathfrak F''}$. 
\end{proof}

The solution to the strong version (as stated in Problem~\ref{prob: linear equivalence}) of linear equivalence for ideal varieties is clear, since each variety lives in a different ambient space. 
However, a full characterization of weak linear equivalence for the ideal varieties  $\mathcal F_{M,ideal}$ may require using properties specific to the matroid $M$.  
This would be nice to see in, for example, the case of graphic matroids. 
Similarly, solving the linear equivalence problem for the more general family $\mathcal F_M = \{\mathcal V_{M, \mathfrak F'} : \mathfrak F'\subseteq \mathfrak F\}$ would also be nice.

\subsection{The matrix \texorpdfstring{\({M_\varphi}\)}{} for colored DAG models}
By Definition~\ref{def: trek polynomial}, for a colored DAG $(\GG,c)$ the trek monomials $m_T$ are monomials in the polynomial ring $\R[\lambda_1,\ldots,\lambda_e, \omega_1,\ldots, \omega_n]$, and $\mathcal{T} = \{m_T : T\in\mathcal{T}(i,j), \, i,j\in[m]\}$ is the set of all trek monomials obtained from $(\GG,c)$. 
Let $(\GG,c)$ be a colored DAG with $\GG =([m],E)$. 
For $i,j\in [m]$, define the row vector $t(i,j)\in \R^{\mathcal{T}}$ where 
\[
t(i,j)_\tau = |\{ T\in \mathcal{T}(i,j) : m_T = \tau\}|
\]
is the number of treks between $i$ and $j$ in $\GG$ having trek monomial $\tau$. These vectors are the rows of $M_{\varphi_{\GG,c}}$.

The formula in~\eqref{eqn: map decomp} decomposes the map $\varphi_{\GG,c}$ from Definition~\ref{def: colored DAG model} as $\varphi_{\GG,c}(\theta) = M_{\varphi_{\GG,c}}\mathbb T$. 
Since every trek has exactly one top vertex, each trek monomial $m_T$ is divisible by exactly one $\omega_i$ with degree $1$. 
Geometrically, this means the exponent vectors of the trek monomials $m_T$ span the affine hyperplane $\omega_1 + \cdots + \omega_n = 1$ in $\R^{n + e}$. 
Consequently, any binomial relation $\prod \rho_T^{u_T} - \prod \rho_T^{v_T} = 0$ in the pullback forces $\sum u_T = \sum v_T$, ensuring that $I_{\mathcal{T}}$ is a homogeneous toric ideal under the standard grading.
Hence, the conditions of Theorem~\ref{thm: min lin subspace} are fulfilled. 
In particular, we have the following result. 

\begin{corollary}
    \label{cor: colored DAG lin equiv reduction}
    The linear span of a colored DAG model $\mathcal M(\GG,c)$ is given by the linear forms 
    $
    \{\sum_{1\leq i\leq j\leq m}\alpha_{ij}\sigma_{ij} : (\alpha_{ij} : 1\leq i\leq j \leq m)\in \ker(M_{\varphi_{\GG,c}}^t)\}. 
    $
    Moreover, two colored DAGs $(\GG,c)$ and $(\HH,c')$ are linearly equivalent if and only if $\ker(M_{\varphi_{\GG,c}}^t) = \ker(M_{\varphi_{\HH, c'}}^t)$. 
\end{corollary}

Familiar geometry appears when the coloring is the constant coloring $c^\ast$. 

\begin{remark}[The rational normal curve and constant colorings] 
    \label{ex: rational normal curve} 
    Let $(\GG,c^\ast)$ be a colored DAG with the constant coloring. 
    It follows from Lemma~\ref{lem: degree closure} that the toric variety $\mathcal V_{\mathcal{T}}$ is the rational normal curve of degree $N$ where $N$ is the length of the longest trek in $\GG$. 
    In particular, by~\eqref{eqn: map decomp} the colored DAG model $(\GG, c^\ast)$ is a linear transformation of the rational normal curve of degree $N$. 
\end{remark}

We illustrate Corollary~\ref{cor: colored DAG lin equiv reduction} with a simple example.

\begin{example}
\label{ex: motivation}
Let $\GG = ([3], E)$ be a graph with $E=\{1\rightarrow2, 2\rightarrow3\}$, and $\HH = ([3], E')$ be a graph with $E'=\{3\rightarrow2, 2\rightarrow1\}$.  
Consider the colored DAGs $(\GG,c^\ast)$ and $(\HH,c^\ast)$, each equipped with the constant coloring $c^\ast$.
The trek polynomials for both graphs live in the ring with two variables $\R[\lambda,\omega]$, and $\mathcal{T}_{\GG} = \mathcal{T}_{\HH}= \{\omega, \omega\lambda, \omega\lambda^2,\omega\lambda^3, \omega\lambda^4\}$.
It is easy to see that
\begin{equation*}
    M_{\varphi_{\GG,c^\ast}} = \begin{pmatrix}
        1 & 0 & 0 & 0 & 0\\
        0 & 1 & 0 & 0 & 0 \\
        1 & 0 & 1 & 0 & 0\\
        0 & 0 & 1 & 0 & 0 \\
        0 & 1 & 0 & 1 & 0\\
        1 & 0 & 1 & 0 & 1
    \end{pmatrix}, \quad
    M_{\varphi_{\HH,c^\ast}} = \begin{pmatrix}
        1 & 0 & 1 & 0 & 1\\
        0 & 1 & 0 & 1 & 0\\
        1 & 0 & 1 & 0 & 0\\
        0 & 0 & 1 & 0 & 0 \\
        0 & 1 & 0 & 0 & 0 \\
        1 & 0 & 0 & 0 & 0
    \end{pmatrix}
\end{equation*}
where rows are given by $t(1, 1), t(1, 2), t(2, 2), t(1, 3), t(2, 3), t(3, 3)$. Notice that $\ker(M_{\varphi_{\GG, c^\ast}}^t)\neq \ker(M_{\varphi_{\HH, c^\ast}}^t)$ with $( 1, 0, -1, 1, 0, 0)^t$ belonging to the first kernel, but not the second. 
This implies that we can distinguish the models with the linear form $f(\Sigma) = \sigma_{11}-\sigma_{22}+ \sigma_{13}$.
\end{example}

Corollary~\ref{cor: colored DAG lin equiv reduction} motivates the computation of a basis for ${M}_{\varphi_{\GG,c}}^t$ for general colored DAGs $(\GG,c)$. 
Computing these kernels for the constant coloring is a good starting point since they already provide a certificate of model distinguishability for colored DAGs. 

\begin{proposition}
    \label{prop: sufficient condition}
    Let $\mathcal{M}(\GG,c)$ and $\mathcal{M}(\HH,c')$ be colored DAGs models of the same dimension.   
    Suppose that $f = \sum\alpha_{ij}\sigma_{ij}$ is a linear form belonging to $I_{\GG, c}$ such that $\alpha =(\alpha_{ij} : 1\leq i \leq j\leq m)\notin\ker(M_{\varphi_{\HH,c^\ast}}^t)$. 
    Then $\mathcal{M}(\GG,c)\neq \mathcal{M}(\HH, c')$. 
\end{proposition}

\begin{proof}
    Since $\alpha\notin \ker(M_{\varphi_{\HH,c^\ast}}^t)$, it follows from Theorem~\ref{thm: min lin subspace} that there exists $\Sigma\in \mathcal{M}(\HH, c^\ast)\subseteq \mathcal{M}(\HH, c')$ such that $f(\Sigma)\neq 0$. 
    Hence, $f\notin I_{\HH, c'}$. 
    Since $f\in I_{\GG, c}$, it follows that $I_{\GG, c} \neq I_{\HH, c'}$. 
\end{proof}

\section{Poset parametrizations}
\label{sec: posets}
Theorem~\ref{thm: min lin subspace} and Corollary~\ref{cor: lin equiv reduction} motivate computing a basis of the kernel of the matrix $M_{\varphi}^t$ for a rational map $\varphi$ as in~\eqref{eqn: rational map} with $f_1,\ldots, f_m\in\mathbb{Z}_{\geq 0}[\theta_1,\ldots, \theta_n]$. 
To solve the problems in the introduction, we would like to obtain a closed-form expression for this basis, and we would like to do this simultaneously for all maps in a family $\mathcal F= \{\varphi_i(\Theta_i): \Theta_i\in \mathbb R^{n_i}, \varphi: \Theta_i\to \mathbb R^{m_i}, i\in\mathbb{Z}_{\geq0}\}$. 
Using the assumption on the coefficients of $f_1,\ldots, f_m$, this section provides a combinatorial theory for doing this using a Möbius inversion formula that holds for an associated poset $P_\varphi$. 
In Subsection~\ref{subsec: the poset}, we define the poset $P_{\varphi}$, and derive the relevant Möbius inversion formula. 
In Subsection~\ref{subsec: poset invariants}, we define the family of polynomials in $I_{\varphi}$, called $P_{\varphi}$-invariants, that are obtained via the Möbius inversion formula. 
Subsection~\ref{subsec: the ideal of the poset} describes a canonical reparametrization $\hat\varphi: \Theta\to \mathbb R^{P_\varphi}$ of $\varphi(\Theta)$ into a (generally higher dimensional) space $\mathbb R^{P_\varphi}$ and makes explicit the relationship between the vanishing ideals $I_\varphi$ and $I_{\widetilde\varphi}$. 
This allows us to formalize how the linear equivalence problem (Problem~\ref{prob: linear equivalence}), toric reparameterization problem (Problem~\ref{prob: toric}), and implicitization problem (Problem~\ref{prob: colored implicitization}) can be treated by deducing properties of $P_{\varphi}$. 
General results describing the reduction of these problems to combinatorics are presented in Subsections~\ref{subsec: combinatorial linear equiv},~\ref{subsec: combinatorial toric} and~\ref{rem: toric implicitization}.
A related connection between finite topologies and the condition $\varphi = \widetilde\varphi$ is presented in Subsection~\ref{subsec: pi-graphs}.

\subsection{The poset \texorpdfstring{\(P_{\varphi}\)}{}}
\label{subsec: the poset}

In the following, we let $\varphi: \Theta \to \mathbb R^m$ be a rational map $\varphi$ as in~\eqref{eqn: rational map} with $f_1,\ldots, f_m\in\mathbb{Z}_{\geq 0}[\theta_1,\ldots, \theta_n]$. 
We let $\mathcal{T}\subset\R[\theta_1,\ldots, \theta_n]$ denote the set of monomials supporting $f_1,\ldots, f_m$ (see~\eqref{eqn: support set}), and $\{c_1,\ldots, c_m\}\subseteq \mathbb Z_{\geq 0}^{\mathcal T}$ denote the coefficient vectors of $f_1,\ldots, f_m$, i.e., 
\[
f_i(\theta) = \sum_{\tau\in\mathcal T}c_{i,\tau}\tau.
\]
\begin{definition}[The poset $P_\varphi$]
For $\{{i_1},\ldots, {i_N}\}$ a nonempty subset of $[m]$ define the \emph{meet} of $c_{i_1},\ldots, c_{i_N}$ as 
\[
\bigwedge_{\ell \in[N]}c_{i_\ell} = (\textrm{min}(c_{i_1,\tau},\ldots, c_{i_N,\tau}) : \tau\in \mathcal{T}).
\]
Let $P_{\varphi}$ be the poset on the set of all meets $\bigwedge_{\ell \in[N]}c_{i_\ell}$ ordered by coordinate-wise domination, i.e., $s' \prec s$ if and only if $s'_\tau \leq s_\tau$ for all $\tau \in \mathcal{T}$.
\end{definition}
The following result is not difficult to show.

\begin{proposition}
\label{prop: min elt}
    The poset $P_\varphi$ is a meet semilattice with a unique minimal element $\hat0\in \mathbb Z^{\mathcal T}_{\geq 0}$. 
    Moreover, $\hat 0 = 0$ if and only if no monomial $\tau\in \mathcal T$ has nonzero coefficient in all coordinate functions $f_1, \ldots, f_m$.
\end{proposition}

For $s^{(1)},\ldots, s^{(\ell)}\in P_{\varphi}$ we let $s^{(1)}\vee \cdots \vee s^{(\ell)} = (\max_{i\in[\ell]}(s_\tau^{(i)}) : \tau \in \mathcal{T})$.
Note that $s^{(1)}\vee \cdots \vee s^{(\ell)}$ need not belong to $P_{\varphi}$.  
However, if it does, then it is the join of $s$ and $s'$.
\begin{proposition}
\label{prop: poset sum}
    Define the function $r(s) = s - \bigvee_{s'\prec s}s'$ for $s\in P_{\varphi}$.
    Then 
    \[
    s = \sum_{s'\preceq s}r(s').
    \]
\end{proposition}
\begin{proof}
    Suppose first that $\hat 0 = 0$ in $P_\varphi$, and fix $\tau\in \mathcal{T}$.  We will show 
    \[
    s_\tau = \sum_{s'\preceq s}r(s')_\tau. 
    \]
    Note that $s_\tau = \max_{s'\preceq s}(s'_\tau)$, and $0 = \min_{s'\preceq s}(s_\tau')$ since $0\preceq s$ by Proposition~\ref{prop: min elt}.
    For $0\leq j \leq s_\tau$, define the smallest element with $\tau$-th coordinate equal to $j$:
    \[
    s^{(j)} = \bigwedge_{s'\preceq s : s_\tau' = j}s'.
    \]
    We claim that $r(u)_\tau =0$ for all $u \neq s^{(j)}$ with $u_\tau = j$.
    To see this, choose $u\preceq s$ with $u_\tau = j$ and $u \neq s^{(j)}$. 
    By definition of $s^{(j)}$ we have that $s^{(j)}\prec u$. 
    Because $s^{(j)} \prec u$ and the poset is ordered by coordinate-wise domination, the maximum $\tau$-th coordinate among all elements strictly preceding $u$ is exactly achieved by $s^{(j)}$.
    Hence, 
    \[
    r(u)_\tau = u_\tau - \bigvee_{s'\prec u}s'_\tau = u_\tau - s_\tau^{(j)} = j - j = 0. 
    \]
    Since $j$ was arbitrary, the only $s'\preceq s$ with $r(s')_\tau$ possibly not equal to $0$ are $s^{(0)},\ldots, s^{(s_\tau)}$. 
    We now compute the values $r\left(s^{(0)}\right)_\tau,...,r\left(s^{(s_\tau)}\right)_\tau$.

    We first note that $s^{(j)}$ need not exist for all $0\leq j\leq s_\tau$.
    Specifically, it is possible that no $s'\preceq s$ has $s_\tau' = j$. 
    However, we know that $s^{(0)}$ and $s^{(s_\tau)}$ exist. 

    Let $0=j_0< j_1<\cdots < j_\ell = s_\tau$ denote the indices for which $s^{(j_i)}$ exists. 
    We claim that
    \begin{equation}
    \label{eqn:fvals}
    \begin{split}
    r\left(s^{(0)}\right)_\tau &= 0, \qquad \textrm{and}\\
    r\left(s^{(j_i)}\right)_\tau &= j_i - j_{i-1} \qquad \textrm{for all $i= 1,\ldots, \ell$.}
    \end{split}
    \end{equation}
    To see this, we first show that $s^{(0)} \prec s^{(j_1)}\prec \cdots \prec s^{(j_{\ell-1})} \prec s^{(s_\tau)}$. 
    For the sake of contradiction, suppose there exist $0 \leq t < k \leq s_\tau$ such that $s^{(j_t)}$ and $s^{(j_k)}$ are incomparable in the poset $P_{\varphi}$. 
    Since $P_{\varphi}$ is a meet semilattice then $s^{(j_t)}\wedge s^{(j_k)}\in P_{\varphi}$ with $s^{(j_t)}\wedge s^{(j_k)}\preceq s^{(j_t)}, s^{(j_k)}\preceq s$.
    Moreover, 
    \[
    \left(s^{(j_t)}\wedge s^{(j_k)}\right)_\tau = \min\left(s^{(j_t)}_\tau, s^{(j_k)}_\tau\right) = s^{(j_t)}_\tau = j_t. 
    \]
    So it must be that $s^{(j_t)} = s^{(j_t)}\wedge s^{(j_k)}$ (by definition of $s^{(j)}$). 

    We can now see~\eqref{eqn:fvals} holds. 
    Namely, since $s^{(j_{i-1})}\prec s^{(j_i)}$, and there is no $s'\prec s^{(j_i)}$ with $s_\tau'\geq j_{i-1}+1$, we obtain $r(s^{(j_i)})_\tau = j_i - j_{i-1}$.
    The fact that $r(s^{(0)})_\tau = 0$ is also clear since no $s'\prec s^{(0)}$ can have $s_\tau'>0$. 

    Since \eqref{eqn:fvals} holds, and $r(s')_\tau = 0$ for all $s'\notin\{s^{(j_1)},\ldots,s^{(j_{\ell})}\}$, we obtain
    \[
    \sum_{s'\preceq s}r(s')_\tau = \sum_{i=1}^\ell r(s^{(j_i)})_\tau = \sum_{i=1}^\ell (j_i - j_{i-1}) = s_\tau,
    \]
    which completes the proof. Finally, in case $\hat 0 \neq 0$, the proof is analogous to the discussion above, the only difference being the defining $s^{(j)}$ for $\hat{0}_\tau \leq j\leq s_\tau$ as there is no element $s'$ with $s'_{\tau}< \hat{0}_\tau$.
\end{proof}

We now define two functions that map the elements of the poset $P_{\varphi}$ to specific polynomials in $\R[\theta_1,\ldots, \theta_n]$. Using the column vector of monomials $\mathbb{T}$ defined in Section~\ref{subsec: minimal linear subspace}, let $f, g: P_{\varphi} \to \R[\theta_1,\ldots, \theta_n]$ be defined via the dot product:
\begin{equation}
\label{eqn: mobius functions}
    f(s) = r(s) \mathbb{T},\qquad
    g(s) = s\mathbb{T}.
\end{equation} 
We also let $\mu(s',s)$ denote the \emph{Möbius function} of the poset $P_{\varphi}$ which is defined recursively as
\begin{equation}
    \label{eqn: mobius_definition}
    \begin{split}
        \mu(s,s) &= 1, \quad \textrm{for all } s\in P_{\varphi}\\
        \mu(s,u) & = - \sum_{s\preceq t\prec u}\mu(s,t) \quad \textrm{for all } s \prec u \textrm{ in } P_{\varphi}. 
    \end{split}
\end{equation}

We then have the following theorem via the M\"obius inversion formula \cite[Proposition~3.7.1]{stanley2011enumerative}. 
\begin{theorem}
    \label{thm: mobius}
    For all $s\in P_{\varphi}$
    \begin{equation}
    \label{eqn:mobius}
    f(s) = \sum_{s'\preceq s}\mu(s',s)g(s').
    \end{equation}
\end{theorem}
\begin{proof}
    It follows from Proposition~\ref{prop: poset sum} that
$
g(s) = \sum_{s'\preceq s} f(s')
$
for all $s\in P_{\varphi}$. 
By the Möbius inversion formula \citep[Proposition~3.7.1]{stanley2011enumerative}, we have that 
$
f(s) = \sum_{s'\preceq s}\mu(s',s)g(s').
$
\end{proof}

\begin{definition}[Residue vector]
    \label{def: residue vector}
    For $s\in P_{\varphi}$, the vector and polynomial 
    \[
    r(s) = s - \bigvee_{s'\prec s}s'\in \Z^{\mathcal{T}}_{\geq 0} \qquad \textrm{and} \qquad r(s)\mathbb{T} = \sum_{\tau\in \mathcal{T}}r(s)_{\tau}\tau\in \mathbb{Z}_{\geq 0}[\theta_1,\ldots, \theta_n]
    \]
    are called, respectively, the \emph{residue vector} and \emph{residue polynomial} of $s$. 
\end{definition}

When the coordinate functions $f_1,\ldots, f_m$ of the map $\varphi$ are enumerating combinatorial objects, then the residue vectors have a combinatorial interpretation.  
For instance, when $\varphi = \varphi_{\GG,c}$ for a colored DAG $(\GG,c)$, a vector $s\in P_{\varphi}$ is the multiplicity vector for a multiset of treks in $(\GG,c)$. 
 The residue vector $r(s)$ counts the new treks introduced by $s$, that is, the treks appearing in $s$ that remain after removing the maximum multiplicity of those treks found in any $s'$ strictly preceding $s$ in $P_{\varphi}$.

\begin{example}[Residue vectors of a colored DAG]
\label{ex: poset}
    Consider the colored DAGs $(\GG,c^\ast)$ from Example \ref{ex: motivation}. The elements of the poset $P_{\varphi_{\GG,c^\ast}}$ are the multiplicity vectors corresponding to the valid multisets of treks with $\mathbb{T} = (\omega, \omega\lambda, \omega\lambda^2,\omega\lambda^3, \omega\lambda^4)^t$. Figure \ref{fig: poset Ex G123} displays the Haase diagram of this poset, with the minimal element $\hat 0 = {0} = (0,0,0,0,0)$ at the bottom.
    Next to each element $s\in P_{\varphi_{\GG,c^\ast}}$, we display the corresponding residue vector $r(s)$ in red. 
    Consider the element $s = (1,0,1,0,0)$.
    Its residue vector is 
    $$r(s) = s - s_1 \vee s_2 \vee {0} = s - s_1 \vee s_2,$$
    where $s_1 = (1,0,0,0,0)$ and $s_2 = (0,0,1,0,0)$ are elements strictly preceding $s$. The coordinate-wise maximum of $s_1$ and $s_2$ is 
    $$s_1 \vee s_2 = \max((1,0,0,0,0), (0,0,1,0,0)) = (1,0,1,0,0).$$ 
    Thus, the residue vector of $s$ is $r(s) = s - s_1 \vee s_2 = (0,0,0,0,0)$.

    The Möbius function $\mu(s', s)$ can also be computed. 
    For instance, for the same element $s = (1,0,1,0,0)$, the lower ideal consists of $s, s_1, s_2$, and ${0}$. Applying the recursive definition yields $\mu(s,s) = 1$, $\mu(s_1,s) = -1$, $\mu(s_2,s) = -1$, and $\mu({0}, s) = 1$. These integer values will serve as the coefficients for the algebraic invariants defined in the next section.
    
     \begin{figure}[t]
    \begin{subfigure}[b]{0.3\textwidth}
    \centering
    \begin{tikzpicture}
            \node[circle, draw, fill=red!50, minimum size=1pt, inner sep=1pt] (0) at (0, 0) {1};
            \node[circle, draw, fill=red!50, minimum size=1pt, inner sep=1pt] (1) at (1, -0.5) {2};
            \node[circle, draw, fill=red!50, minimum size=1pt, inner sep=1pt] (2) at (2, 0.5) {3};

             \draw[->, blue!70, >=stealth] (0) edge (1);
             \draw[->, blue!70, >=stealth] (1) edge (2);

             
    \end{tikzpicture}
    \caption{}
    \label{fig: path graph}
    \end{subfigure}
    \hfill
    \begin{subfigure}[b]{0.6\textwidth}
        \centering
        {\tiny
        \begin{tikzpicture}[scale = 1.1]
            \node (0) at (0, 0) {(0,0,0,0,0)};
            \node (1) at (-2, 1) {(0,1,0,0,0)};
            \node (2) at (0, 1) {(0,0,1,0,0)};
            \node (3) at (2, 1) {(1,0,0,0,0)};
            \node (4) at (-2, 2) {(0,1,0,1,0)};
            \node (5) at (2, 2) {(1,0,1,0,0)};
            \node (6) at (2, 3) {(1,0,1,0,1)};
             \draw[-,  >=stealth] (0) edge (1);
             \draw[-,  >=stealth] (0) edge (2);
             \draw[-,  >=stealth] (0) edge (3);
             \draw[-,  >=stealth] (1) edge (4);
             \draw[-,  >=stealth] (2) edge (5);
             \draw[-,  >=stealth] (3) edge (5);
             \draw[-,  >=stealth] (5) edge (6);
%
            \node (r1) at (-2 - 0.55, 1 - 0.3) {\textcolor{red}{(0,1,0,0,0)}};
            \node (r2) at (0 + 0.55, 1 - 0.3) {\textcolor{red}{(0,0,1,0,0)}};
            \node (r3) at (2 + 0.55, 1 - 0.3) {\textcolor{red}{(1,0,0,0,0)}};
            \node (r4) at (-2 - 0.55, 2 - 0.3) {\textcolor{red}{(0,0,0,1,0)}};
            \node (r5) at (2 + 0.55, 2 - 0.3) {\textcolor{red}{(0,0,0,0,0)}};
            \node (r6) at (2 + 0.55, 3 - 0.3) {\textcolor{red}{(0,0,0,0,1)}};          
    \end{tikzpicture}
    }
    \caption{}
    \label{fig: poset for graph}
    \end{subfigure}
        \caption{(a) The colored graph $(\GG,c^\ast)$. (b) The poset $P_{\varphi_{\GG,c^\ast}}$ for $(\GG,c^\ast)$.}
        \label{fig: poset Ex G123}
    \end{figure}
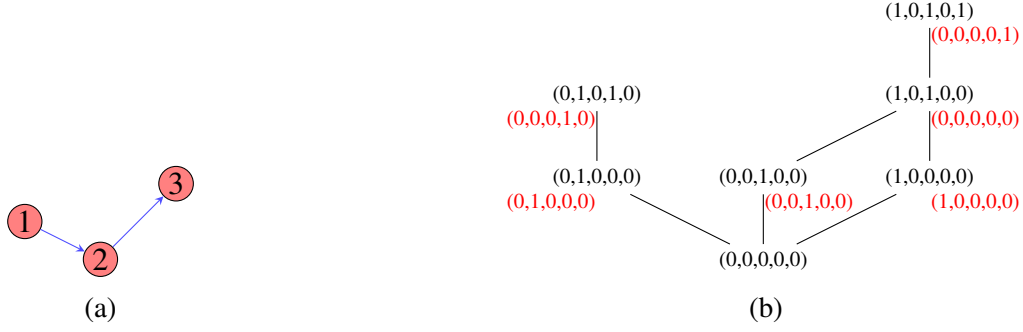
\end{example}

\begin{example}[Residue polynomials of ideal varieties of graphic matroids]
    \label{ex: poset for graphic matroid}
    Let $M_G$ be the graphic matroid for an undirected graph $G = (V,E)$, and let $\mathcal V_{M_G,\mathfrak F'}$ be an ideal variety of $M_G$ as defined in Section~\ref{subsec: ideal varieties}. 
    Since $\mathfrak F'$ is an ideal in the lattice of flats $P_{M_G}$ it follows that the atoms in the poset $P_{\varphi_{M, \mathfrak F'}}$ are precisely the standard basis vectors $b_{\{ij\}}\in \mathbb R^{\mathfrak F'}$ indexed by the set of singletons $\{ij\}$ where $\{i,j\}\in E$ is an edge appearing in some $F\in \mathfrak F'$. 
    Since $P_M$ is atomic, so is $P_{\varphi_{M, \mathfrak F'}}$. 
    It follows that the residue polynomial $r(F)\mathbb T$ of $F$ equals $0$ whenever $F$ is not an atom and $r(F)\mathbb T = \theta_i\theta_j$ whenever $F = \{ \{i, j\}\}$ is an atom. 

    Note that when $\mathfrak F'\subseteq \mathfrak F$ is not an ideal in $P_M$, it is possible that $P_{\varphi_{M_G, \mathfrak F'}}$ does not have the property that every element covering $0$ is a standard basis vector index by an atom in $P_M$. 
    In this case, the poset $P_{\varphi_{M_G, \mathfrak F'}}$, and its residue polynomials, are more complex (interesting). 
\end{example}

\begin{example}[Residue vectors of some degenerate subvarieties of secants]
    \label{ex: poset for a secant subvariety}
    In this example we consider the subvarieties of a general determinantal variety introduced in Section~\ref{subsec: secants} for $m=n=3$ and $r=2$. To simplify the notation, we denote the parameters as $s, t, p, q\in \mathbb R^3$. Before adding any constraints on the parameters, $\mathbb T = ( s_1t_1,\ldots,s_3t_3,  p_1q_1, \ldots, p_3q_3)^t$. 
    The elements of the poset $P_{\varphi_{3,3}^2}$ are 0-1 vectors of length $18$, with exactly two non-zero entries, as the parametrization functions are binomials. Note that in this case no two binomials share a monomial. 
    Thus, the meet of any two coefficient vectors is the zero vector, and the poset is an antichain of length $18$ covering $0$. Hence, trivially, $r(s) = s$ for all $s\in P_{\varphi_{3,3}^2}$.

    As soon as the constraints are introduced, the length of $\mathbb T$ decreases, and the poset changes in the following way. If $s_1 =s_2=s_3 = s$ and $q_1 =q_2=q_3 = q$, the vector of possible monomials becomes $\mathbb T = (st_1, st_2, st_3, p_1q, p_2q, p_3q)$, and we denote the parametrization as
    \begin{equation}
    \label{eqn: nataliias variety 1}
    \varphi_1 = \varphi^2_{3, 3}|_{\substack{s_1=s_2=s_3\\
    q_1=q_2=q_3}}: (s, t_1, t_2, t_3, p_1, p_2, p_2, q)\mapsto (st_j+ p_iq)_{ij} = (x_{ij})
    \end{equation}
    The Haase diagram of this poset is presented in Figure \ref{fig: poset Ex secant 1}. The minimal element of the poset is $\hat{0} = (0, 0, 0, 0, 0, 0)$. The elements in the upper cover of $\hat{0}$ are the only ones to have non-zero residue vectors. Every vector $c$ coming from the parametrization map has exactly two elements in the lower cover. 

    When fewer constraints are imposed, the poset changes. Consider the map $\varphi^2_{3, 3}$ with $s_1=s_2$ and $q_1=q_2$. The parametrization is given as 
    \begin{equation}
    \label{eqn: nataliias variety 2}
    \varphi_2 = \varphi^2_{3, 3}|_{\substack{s_1=s_2\\
    q_1=q_2}}.
    \end{equation}
    In this case, the vector $\mathbb T$ has length 12. The Haase diagram of the poset $P_{\varphi_2}$ is presented in Figure \ref{fig: poset Ex secant 2}. 
    Here the notation $c_{ij}$ corresponds to the coefficient vector of the $ij$-th coordinate function of $\varphi_2$. 
    The minimal element $\hat{0}$ is again the zero vector.

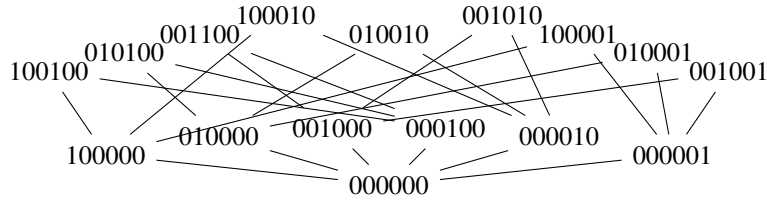
\begin{figure}[t]
        \centering
        \begin{tikzpicture}[scale=0.50]
            \node (0) at (0, 0) {\footnotesize 000000};
            \node (1) at (-7.5, 0.8) {\footnotesize 100000};
            \node (2) at (-4.5, 1.3) {\footnotesize 010000};
            \node (3) at (-1.5, 1.5) {\footnotesize 001000};
            \node (4) at (1.5, 1.5) {\footnotesize 000100};
            \node (5) at (4.5, 1.3) {\footnotesize 000010};
            \node (6) at (7.5, 0.8) {\footnotesize 000001};

            \node (7) at (-9, 3) {\footnotesize 100100};
            \node (8) at (-7, 3.5) {\footnotesize 010100};
            \node (9) at (-5, 4) {\footnotesize 001100};
            \node (10) at (-3, 4.5) {\footnotesize 100010};
            \node (11) at (0, 4) {\footnotesize 010010};
            \node (12) at (3, 4.5) {\footnotesize 001010};
            \node (13) at  (5, 4) {\footnotesize 100001};
            \node (14) at (7, 3.5) {\footnotesize 010001};
            \node (15) at (9, 3){\footnotesize 001001};
             \draw[-,  >=stealth] (0) edge (1);
             \draw[-,  >=stealth] (0) edge (2);
             \draw[-,  >=stealth] (0) edge (3);
             \draw[-,  >=stealth] (0) edge (4);
             \draw[-,  >=stealth] (0) edge (5);
             \draw[-,  >=stealth] (0) edge (6);
             
             \draw[-,  >=stealth] (1) edge (7);
             \draw[-,  >=stealth] (1) edge (10);
             \draw[-,  >=stealth] (1) edge (13);
             \draw[-,  >=stealth] (2) edge (8);
             \draw[-,  >=stealth] (2) edge (11);
             \draw[-,  >=stealth] (2) edge (14);
             \draw[-,  >=stealth] (3) edge (9);
             \draw[-,  >=stealth] (3) edge (12);
             \draw[-,  >=stealth] (3) edge (15);
             \draw[-,  >=stealth] (4) edge (7);
             \draw[-,  >=stealth] (4) edge (8);
             \draw[-,  >=stealth] (4) edge (9);
             \draw[-,  >=stealth] (5) edge (10);
             \draw[-,  >=stealth] (5) edge (11);
             \draw[-,  >=stealth] (5) edge (12);
             \draw[-,  >=stealth] (6) edge (13);
             \draw[-,  >=stealth] (6) edge (14);
             \draw[-,  >=stealth] (6) edge (15);
    \end{tikzpicture}
        \caption{ The poset $P_{\varphi_1}$ of the subvariety of the determinantal variety $\mathcal V_{\varphi^2_{3,3}}$ defined by $\varphi_1$ in~\eqref{eqn: nataliias variety 1}.}
        \label{fig: poset Ex secant 1}
    \end{figure}

    \begin{figure}[t]
        \centering
        \begin{tikzpicture}
            \node (0) at (0, 0) { $\hat{0}$};
            
            \node (1) at (-5.2, 2) {$c_{13}$};
            \node (2) at (-4.4, 2) {$c_{23}$};
            \node (3) at (-3.6, 2) {$c_{33}$};
            \node (4) at (-2.8, 2) {$c_{32}$};
            \node (5) at (-2, 2) {$c_{31}$};

            \node (6) at (1, 2.5) {$c_{11}$};
            \node (7) at (2, 2.5) {$c_{12}$};
            \node (8) at (3, 2.5) {$c_{21}$};
            \node (9) at (4, 2.5) {$c_{22}$};

            \node (10) at (0, 1) {$c_{11}\wedge c_{12}$};
            \node (11) at (1.8, 1) {$c_{11}\wedge c_{21}$};
            \node (12) at (3.6, 1) {$c_{12}\wedge c_{22}$};
            \node (13) at (5.4, 1) {$c_{21}\wedge c_{22}$};

            \draw[-,  >=stealth] (0) edge (1);
            \draw[-,  >=stealth] (0) edge (2);
            \draw[-,  >=stealth] (0) edge (3);
            \draw[-,  >=stealth] (0) edge (4);
            \draw[-,  >=stealth] (0) edge (5);

            \draw[-,  >=stealth] (0) edge (10);
            \draw[-,  >=stealth] (0) edge (11);
            \draw[-,  >=stealth] (0) edge (12);
            \draw[-,  >=stealth] (0) edge (13);

            \draw[-,  >=stealth] (10) edge (6);
            \draw[-,  >=stealth] (10) edge (7);
            \draw[-,  >=stealth] (11) edge (6);
            \draw[-,  >=stealth] (11) edge (8);
            \draw[-,  >=stealth] (12) edge (7);
            \draw[-,  >=stealth] (12) edge (9);
            \draw[-,  >=stealth] (13) edge (8);
            \draw[-,  >=stealth] (13) edge (9);
        \end{tikzpicture}
        \caption{The poset $P_{\varphi_2}$ of the subvariety of the  determinantal variety $\mathcal V_{\varphi^2_{3,3}}$ defined by $\varphi_2$ in~\eqref{eqn: nataliias variety 2}.}
        \label{fig: poset Ex secant 2}
    \end{figure}
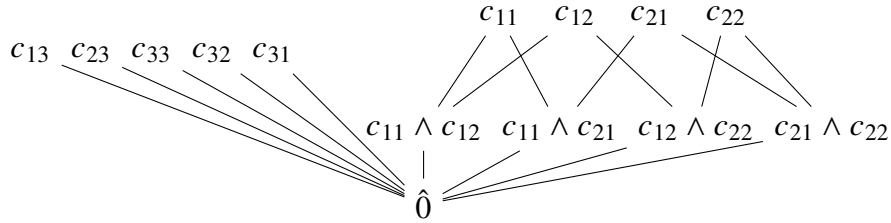

\end{example}

\subsection{\texorpdfstring{$P_{\varphi}$}{P}-invariants}
\label{subsec: poset invariants}
The Möbius inversion formula for the poset $P_{\varphi}$ may be used to obtain polynomial constraints on the semialgebraic set $\varphi(\Theta)$; i.e., polynomials belonging to the vanishing ideal $I_{\varphi}$ in~\eqref{eqn: vanishing ideal}. 

\begin{definition}[$P_{\varphi}$-invariant]
    \label{def: poset invariants}
    Let $\varphi:\Theta\to \mathbb R^m$ be a rational map with coordinate functions $f_1,\ldots, f_m\in\mathbb Z_{\geq 0}[\theta_1,\ldots, \theta_n]$, and let $s\in P_\varphi$.
    If there exists a polynomial $p\in\mathbb{R}[x_1,\ldots, x_m]$ and rational functions $\beta_s,\ \alpha_{s'}\in\mathbb{R}(x_1,\ldots, x_m)$ for all $s'\preceq s$ in $P_{\varphi}$ whose denominators do not vanish on $\varphi(\Theta)$ such that
    \[
    \beta_s(\varphi(\theta)) = f(s)(\theta) \qquad \text{and} \qquad \alpha_{s'}(\varphi(\theta)) = g(s')(\theta)
    \]
    for all $\theta \in \Theta$, and 
    \begin{equation}
    \label{eqn: poset invariant}
    q = p\left(\beta_s - \sum_{s'\preceq s}\mu(s',s)\alpha_{s'}\right)\in \mathbb R[x_1,\ldots, x_m],
    \end{equation} 
    then the polynomial $q$ is called a \emph{$P_{\varphi}$-invariant}. 
\end{definition}

\begin{lemma}
    \label{lem: in the vanishing ideal}
    If $q\in \mathbb R[x_1,\ldots, x_m]$ is a $P_\varphi$-invariant, then $q\in I_\varphi$.
\end{lemma}

\begin{proof}
    Let $q \in \mathbb{R}[x_1,\ldots, x_m]$ be a $P_\varphi$-invariant. To show $q \in I_\varphi$, we evaluate $q$ at a point $\varphi(\theta)$ for an arbitrary $\theta \in \Theta$ and show that $q(\varphi(\theta))=0$. That is, take 
    $$q(\varphi(\theta)) = p(\varphi(\theta))\left(\beta_s(\varphi(\theta)) - \sum_{s'\preceq s}\mu(s',s)\alpha_{s'}(\varphi(\theta))\right).$$

    By the conditions in Definition~\ref{def: poset invariants}, the denominators of the rational functions $\beta_s$ and $\alpha_{s'}$ do not vanish on $\varphi(\Theta)$, meaning they are well-defined at $\varphi(\theta)$. Substituting the identities $\beta_s(\varphi(\theta)) = f(s)(\theta)$ and $\alpha_{s'}(\varphi(\theta)) = g(s')(\theta)$ yields:
    $$q(\varphi(\theta)) = p(\varphi(\theta))\left(f(s)(\theta) - \sum_{s'\preceq s}\mu(s',s)g(s')(\theta)\right).$$
    By the Möbius inversion formula (Theorem~\ref{thm: mobius}), we have $f(s) = \sum_{s'\preceq s}\mu(s',s)g(s')$. Therefore, the expression inside the parentheses evaluates identically to zero for all $\theta \in \Theta$.
    Since $p \in \mathbb{R}[x_1,\ldots, x_m]$ is a polynomial, we have that $q(\varphi(\theta)) = p(\varphi(\theta)) \cdot 0 = 0$ for all $\theta \in \Theta$, which proves $q \in I_\varphi$.
\end{proof}

In other words, a $P_{\varphi}$-invariant is a polynomial in the vanishing ideal $I_{\varphi}$ of the semialgebraic set $\varphi(\Theta)$ produced by clearing denominators in a regular function $\beta_s - \sum_{s'\preceq s}\mu(s',s)\alpha_{s'}$ on the variety of $I_{\varphi}$ obtained from the Möbius inversion formula~\eqref{eqn:mobius}. 

One condition guaranteeing the existence of $P_\varphi$-invariants for $\varphi(\Theta)$ comes from birational geometry and, more recently, algebraic statistics. 
We say that $\varphi(\Theta)$ satisfies \emph{global rational identifiability} \cite{sullivant2023algebraic} if there is a rational function $\psi$ such that $\psi\circ\varphi(\theta) = \theta$ for all $\theta \in \Theta$. 

\begin{proposition}
    \label{prop: existence of poset invariants}
    If $\varphi(\Theta)$ satisfies global rational identifiability then there exist $P_{\varphi}$-invariants. 
\end{proposition}

\begin{proof}
It suffices to show that for $s\in P_\varphi$, there exists a choice for the polynomial $p\in \mathbb R[x_1,\ldots, x_m]$ and rational functions $\beta_s,\alpha_{s'}\in\mathbb R(x_1,\ldots, x_m)$ satisfying the conditions of Definition~\ref{def: poset invariants}. 
To this end, we take $\beta_s = \psi^\ast(f(s))$ and $\alpha_s = \psi^\ast(g(s))$ for all $s\in P_\varphi$, where $\psi^\ast$ is the algebraic pullback of $\psi$. 
Since $\psi^\ast: \mathbb R[\theta_1,\ldots, \theta_n]\to S^{-1}\mathbb R[x_1,\ldots, x_m]$ is a ring homomorphism such that the polynomials in the localizing monoid $S$ do not vanish on $\varphi(\Theta)$, then $\beta_s = h_{s}/d_s$ and $\alpha_{s'} = h_{s', g}/d_{s',g}$ for polynomials $h_s,d_s,h_{s',g},d_{s',g}\in\mathbb R[x_1,\ldots, x_m]$ with $d_s, d_{s',g}$ not vanishing on $\varphi(\Theta)$. 
Moreover, the functional relations $\beta_s(\varphi(\theta)) = f(s)(\theta)$ and $\alpha_{s'}(\varphi(\theta)) = g(s')(\theta)$ for all $\theta\in\Theta$ are immediate from the definition of global rational identifiability. 
Hence, taking $p = d_s\prod_{s'\preceq s}d_{s',g}$ produces a $P_\varphi$-invariant.
\end{proof}

\begin{example}
    \label{ex: colored DAGs birational}
The colored DAG models $\mathcal M(\GG,c)$ satisfy global rational identifiability, 
meaning that there exist rational functions $h_i$ for every vertex $i\in[m]$ of $\GG =([m],E)$ and $h_{ij}$ for every edge $i\to j\in E$ such that if $\Sigma\in\mathcal{M}(\GG,c)$ satisfies $\Sigma = \varphi_{\GG,c}(\lambda_1,\ldots, \lambda_e,\omega_1,\ldots, \omega_n)$ then $h_i(\Sigma) = \omega_i$ and $h_{ij}(\Sigma) = \lambda_{ij}$. 
That is, there exist rational functions in the $\sigma_{ij}$-coordinates that recover the parameters yielding $\Sigma$, for every $\Sigma\in\mathcal{M}(\GG,c)$. 
These rational functions are ratios of determinants
\begin{equation}
    \label{eqn: identifying functions}
    \omega_{i\mid A}(\Sigma) = \frac{|\Sigma_{A\cup\{i\}, A\cup\{i\}}|}{|\Sigma_{A,A}|} \qquad \lambda_{ij\mid A}(\Sigma) = \frac{|\Sigma_{\{i\}\cup A\setminus \{i\}, \{j\}\cup A\setminus\{i\}}|}{|\Sigma_{A,A}|},
\end{equation}
where $\Sigma_{A, B}$ denotes the submatrix of $\Sigma$ with row indices belonging to $A\subseteq[m]$ and column indices belonging to $B\subseteq[m]$. 
The choices for the sets $A$ for which the functions in~\eqref{eqn: identifying functions} yield global rational identifiability are characterized in \cite[Theorem 5]{boege2024colored}. 
For example, by \cite[Lemma 4]{boege2024colored}, we may always choose $A = \textrm{pa}_{\GG}(j)$, where $\textrm{pa}_{\GG}(j) =\{ k\in[m] : k\to j\in E\}$ denotes the set of \emph{parents} of $j$ in $\GG$; e.g., for $\Sigma = \varphi_{\GG,c}(\omega_1,\ldots, \omega_n, \lambda_1,\ldots, \lambda_e)\in\mathcal M(\GG,c)$, we have 
\begin{equation}
\label{eqn: standard identification map}
\omega_{i\mid \textrm{pa}(i)}(\Sigma) = \omega_{i} \textrm{ and } \lambda_{ij\mid \textrm{pa}(j)}(\Sigma) =\lambda_{ij}  \textrm{ for all $i\in [m]$ and $i\to j\in E$.}
\end{equation}
Fixing a choice of the sets $A$ in~\eqref{eqn: identifying functions} that satisfies the conditions of \cite[Theorem 5]{boege2024colored}, we obtain a rational inverse $\psi_{\GG,c}$ to the map $\varphi_{\GG,c}$ in~\eqref{eqn: trek map} that defines the model $\mathcal{M}(\GG,c)$. 
The algebraic pullback of $\psi_{\GG,c}$ is a homomorphism between local rings 
\begin{equation}
    \label{eqn: inverse map}
    \psi_{\GG,c}^\ast: \mathbb{R}[\lambda_1,\ldots, \lambda_e,\omega_1,\ldots, \omega_n] \to U^{-1}\mathbb{R}[\sigma_{ij}: 1\leq i\leq j\leq m]
\end{equation}
where $U$ is a multiplicatively closed set  generated by the principal minors $|\Sigma_{A,A}|$ appearing as denominators in~\eqref{eqn: identifying functions}. 
\end{example}

The following example computes some explicit $P_\varphi$-invariants for colored DAG models using the construction in Example~\ref{ex: colored DAGs birational}. 

\begin{example}
    \label{ex: poset invariants can be nonlinear}
    Consider the colored DAG $(\GG,c^\ast)$ in Figure~\ref{fig: path graph}. 
    The poset $P_{\varphi_{\GG,c^\ast}}$ is depicted in Figure~\ref{fig: poset for graph}, with the residue vector of each element in red. 
    Let $\psi_{\GG,c^\ast}^\ast$ be the ring homomorphism specified by the formulas in~\eqref{eqn: standard identification map}, and let $s = (1,0,1,0,0)$. 
    Note that $s$ is the coefficient vector of the trek polynomial $p_{2,2}^{(\GG,c^\ast)}$.
    It covers the coefficient vectors for $p_{1,3}^{(\GG,c^\ast)}$ and $p_{1,1}^{(\GG,c^\ast)}$, which in turn cover $0$.  
    Since
    \[
    \psi_{\GG,c^\ast}^\ast(p_{2,2}^{(\GG,c^\ast)}) = \sigma_{22}, \quad \psi_{\GG,c^\ast}^\ast(p_{1,3}^{(\GG,c^\ast)}) = \sigma_{13}, \quad \psi_{\GG,c^\ast}^\ast(p_{1,1}^{(\GG,c^\ast)}) = \sigma_{11}, \quad \psi_{\GG,c^\ast}(0) = 0, 
    \]
    and $f(s) = 0$, we obtain the $P_{\varphi_{\GG,c^\ast}}$-invariant 
    \[
    \psi_{\GG,c^\ast}^\ast(f(s)) - \sum_{s'\preceq s}\mu(s',s)\psi_{\GG,c^\ast}^\ast(g(s')) = 0 - \sigma_{22} + \sigma_{13} + \sigma_{11} - 0. 
    \]
    In particular, $\sigma_{11} - \sigma_{22} + \sigma_{13}\in I_{\GG,c^\ast}$, recovering the observation in Example~\ref{ex: motivation}.

    Note that Definition~\ref{def: poset invariants} does not require us to directly apply a rational inverse. We only need that the rational functions $\beta_s,\alpha_s$ agree with $f(s), g(s)$ at all points in the model. 
    Moreover, $P_{\varphi_{\GG,c^\ast}}$-invariants need not be linear. 
    For example, consider the colored DAG $(\GG,c^\ast)$ in Figure~\ref{fig: complete graph}, whose poset $P_{\varphi_{\GG,c^\ast}}$ is depicted in Figure~\ref{fig: poset for complete graph}. 
    Let $s = (1,0,2,1,1)$ be the top-left element of $P_{\varphi_{\GG,c^\ast}}$. 
    In this case, $f(s) = \omega\lambda^2 + \omega\lambda^4$, and $g(s) = p_{3,3}^{(\GG,c^\ast)}$. 
    It can be checked that there are exactly three elements $s'\prec s$ with $\mu(s',s)\neq 0$, and these are 
    \[
    s^{(1)} =(1,0,1,0,0), \quad s^{(2)} =(0,0,1,1,0), \quad s^{(3)} = (0,0,1,0,0).
    \]
    We also have that $g(s^{(1)}) = p_{1,2}^{\GG,c^\ast}$. 
    Hence, we can take $\alpha_{s} = \sigma_{33}$, $\alpha_{s^{(1)}} = \sigma_{12}$, and we can apply the parameter identification formulas $\omega = \sigma_{11}$ and $\lambda = \sigma_{12}/\sigma_{11}$ to all $f(s)$, $g(s^{(2)})$ and $g(s^{(3)})$ to produce the rational functions $\beta_s$, $\alpha_{s^{(2)}}$ and $\alpha_{s^{(3)}}$.  
    Taking $p = \sigma_{11}^3$, recovers the $P_{\varphi_{\GG,c^\ast}}$-invariant
    \[
    \begin{split}
        \sigma_{11}^3\left(\beta_s - \alpha_s + \alpha_{s^{(1)}} + \alpha_{a^{(2)}} - \alpha_{s^{(3)}}\right) &= \sigma_{11}^3((\sigma_{11}\left(\frac{\sigma_{12}}{\sigma_{11}}\right)^2 + \sigma_{11}\left(\frac{\sigma_{12}}{\sigma_{11}}\right)^4) - \sigma_{33} + \sigma_{12}\\ 
        & + (\sigma_{11}\left(\frac{\sigma_{12}}{\sigma_{11}}\right)^2 + \sigma_{11}\left(\frac{\sigma_{12}}{\sigma_{11}}\right)^3) - \left(\frac{\sigma_{12}}{\sigma_{11}}\right)^2),\\
        &= \sigma_{12}^4 + \sigma_{11}\sigma_{12}^3 + \sigma_{11}^2\sigma_{12}^2 + \sigma_{11}^3(\sigma_{12} - \sigma_{33}).
    \end{split}
    \]
    Since our chosen rational functions $\beta_s, \alpha_s$ agree with the corresponding $f(s)$ and $g(s)$ at all points in the model, it follows that this polynomial belongs to the vanishing ideal.  That is, 
    \[
    \sigma_{12}^4 + \sigma_{11}\sigma_{12}^3 + \sigma_{11}^2\sigma_{12}^2 + \sigma_{11}^3(\sigma_{12} - \sigma_{33})\in I_{\GG,c^\ast}.
    \]
    
    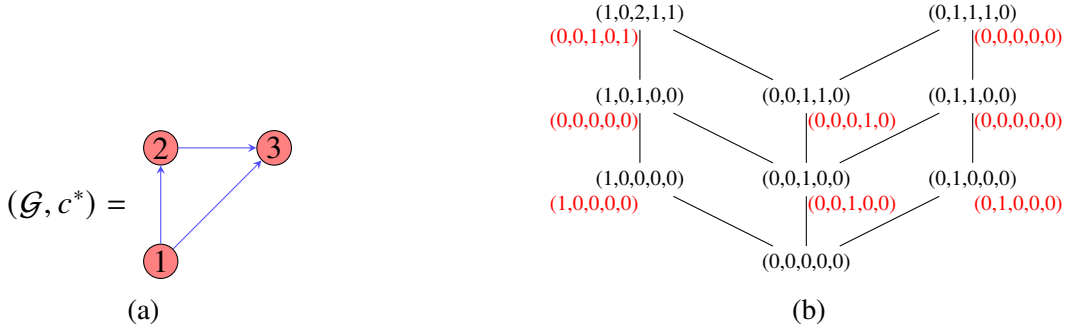
\begin{figure}[t]
    \begin{subfigure}[b]{0.3\textwidth}
    \centering
    \begin{tikzpicture}
            \node[circle, draw, fill=red!50, minimum size=1pt, inner sep=1pt] (0) at (0, 0) {1};
            \node[circle, draw, fill=red!50, minimum size=1pt, inner sep=1pt] (1) at (0, 1.5) {2};
            \node[circle, draw, fill=red!50, minimum size=1pt, inner sep=1pt] (2) at (1.5, 1.5) {3};

             \draw[->, blue!70, >=stealth] (0) edge (1);
             \draw[->, blue!70, >=stealth] (1) edge (2);
             \draw[->, blue!70, >=stealth] (0) edge (2);

            \node at (-1.25,0.75) {$(\GG,c^\ast) =$};
             
    \end{tikzpicture}
    \caption{}
    \label{fig: complete graph}
    \end{subfigure}
    \hfill
    \begin{subfigure}[b]{0.6\textwidth}
        \centering
        {\tiny
        \begin{tikzpicture}[scale = 1.1]
            \node (0) at (0, 0) {(0,0,0,0,0)};
            \node (1) at (-2, 1) {(1,0,0,0,0)};
            \node (2) at (0, 1) {(0,0,1,0,0)};
            \node (3) at (2, 1) {(0,1,0,0,0)};

            \node (4) at (-2, 2) {(1,0,1,0,0)};
            \node (7) at (0, 2) {(0,0,1,1,0)};
            \node (5) at (2, 2) {(0,1,1,0,0)};

            \node (6) at (2, 3) {(0,1,1,1,0)};
            \node (8) at (-2, 3) {(1,0,2,1,1)};

             \draw[-,  >=stealth] (0) edge (1);
             \draw[-,  >=stealth] (0) edge (2);
             \draw[-,  >=stealth] (0) edge (3);

             \draw[-,  >=stealth] (1) edge (4);
             \draw[-,  >=stealth] (2) edge (4);
             \draw[-,  >=stealth] (2) edge (5);
             \draw[-,  >=stealth] (2) edge (7);
             \draw[-,  >=stealth] (3) edge (5);

             \draw[-,  >=stealth] (5) edge (6);
             \draw[-,  >=stealth] (4) edge (8);
             \draw[-,  >=stealth] (7) edge (8);
             \draw[-,  >=stealth] (7) edge (6);


            \node (r1) at (-2 - 0.55, 1 - 0.3) {\textcolor{red}{(1,0,0,0,0)}};
            \node (r2) at (0 + 0.55, 1 - 0.3) {\textcolor{red}{(0,0,1,0,0)}};
            \node (r7) at (0 + 0.55, 2 - 0.3) {\textcolor{red}{(0,0,0,1,0)}};
            \node (r3) at (2 + 0.55, 1 - 0.3) {\textcolor{red}{(0,1,0,0,0)}};
            \node (r4) at (-2 - 0.55, 2 - 0.3) {\textcolor{red}{(0,0,0,0,0)}};
            \node (r5) at (2 + 0.55, 2 - 0.3) {\textcolor{red}{(0,0,0,0,0)}};
            \node (r6) at (2 + 0.55, 3 - 0.3) {\textcolor{red}{(0,0,0,0,0)}};
            \node (r8) at (-2 - 0.55, 3 - 0.3) {\textcolor{red}{(0,0,1,0,1)}};
             
    \end{tikzpicture}
    }
    \caption{}
    \label{fig: poset for complete graph}
    \end{subfigure}
        \caption{(a) The colored graph $(\GG,c^\ast)$. (b) The poset $P_{\varphi_{\GG,c^\ast}}$ for $(\GG,c^\ast)$.}
        \label{fig: poset Ex G123 complete}
    \end{figure}
\end{example}

\begin{remark}
    \label{rem: poset invariants value}
    The identification of families of polynomials belonging to the vanishing ideal of a semialgebraic set is an important endeavor in the fields of applied and computational algebra as well as algebraic statistics. 
    This is because such polynomials can be used to distinguish the vanishing ideals of two semialgebraic sets, and hence solve model distinguishability problems in statistics (e.g. Problem~\ref{prob: colored distinguishability}).  
    For colored DAG models, the only known family of polynomials belonging to the vanishing ideal $I_{\varphi_{\GG,c}}$ are those listed in \cite[Definition 7]{boege2024colored}. 
    The family of $P_{\varphi_{\GG,c^\ast}}$-invariants is a new family of polynomial constraints, providing new tools for solving the model distinguishability problems. 
\end{remark}

While satisfying global rational identifiability is sufficient to guarantee the existence of $P_\varphi$-invariants, it is not necessary.  

\begin{example}
    \label{ex: matroid poset invariants}
    Consider the graphic matroid $M_G$ for the $4$-cycle $G = ([4],E)$ with $E=\{\{1,2\},\{1,3\},\{2,4\}, \{3,4\}\}$. 
    The poset for the matroid flat variety $\mathcal V_{\varphi_{M_G, \mathfrak F}}$ is shown in Figure~\ref{fig: 4-cycle poset}. 
    In this case, for every non-singleton flat $F\in \mathfrak F$, the rational functions $\alpha_F = x_F$ satisfies the conditions in the Definition~\ref{def: poset invariants}.  
    Moreover, by Example~\ref{ex: poset for graphic matroid}, we know that $r(F) = 0$.  
    Hence, we may take $\beta_F = 0$ to obtain the (linear) $P_{\varphi_{M_G, \mathfrak F}}$-invariants
    \[
    0 - \sum_{F'\preceq F \textrm{ in $P_{M_G}$}}\mu(F', F)x_{F'}\in I_{\varphi_{M_G, \mathfrak F}}. 
    \]
    However, the map $\varphi_{M_G, \mathfrak F}$ does not satisfy global rational identifiability. 
    It can be checked that $\varphi_{M_G, \mathfrak F}$ is the transformation $(\theta_1,\theta_2,\theta_3,\theta_4)\mapsto (\theta_1\theta_2,\theta_1\theta_3, \theta_2\theta_4,\theta_3\theta_4)$ taking $\mathbb R^4$ to the $4$-dimensional linear subspace of $\mathbb R^{\mathfrak F}$ defined by these linear $P_{\varphi_{M_G, \mathfrak F}}$-invariants. 
    But this map is not invertible. 
    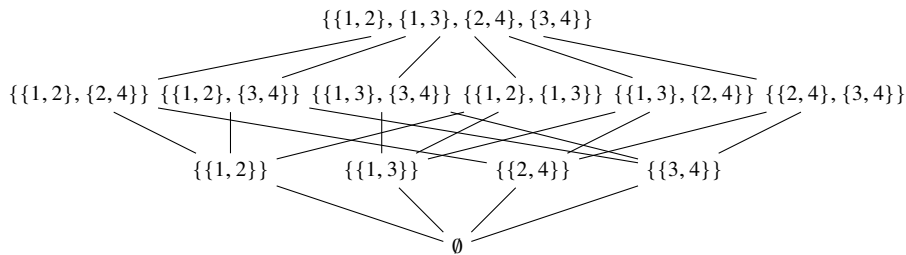
\begin{figure}[b]
        \centering
        {\tiny
        \begin{tikzpicture}
            \node (0) at (0, 0) {$\emptyset$};
            \node (1) at (-3, 1) {$\{\{1,2\}\}$};
            \node (2) at (-1, 1) {$\{\{1,3\}\}$};
            \node (3) at (1, 1) {$\{\{2,4\}\}$};
            \node (4) at (3, 1) {$\{\{3,4\}\}$};

            \node (5) at (-5, 2) {$\{\{1,2\}, \{2,4\}\}$};
            \node (6) at (-3, 2) {$\{\{1,2\}, \{3,4\}\}$};
            \node (7) at (-1, 2) {$\{\{1,3\}, \{3,4\}\}$};
            \node (8) at (1, 2) {$\{\{1,2\}, \{1,3\}\}$};
            \node (9) at (5, 2) {$\{\{2,4\}, \{3,4\}\}$};
            \node (10) at (3, 2) {$\{\{1,3\}, \{2,4\}\}$};
            
            \node (12) at (0, 3) {$\{\{1,2\}, \{1,3\}, \{2,4\}, \{3,4\}\}$};

             \draw[-,  >=stealth] (0) edge (1);
             \draw[-,  >=stealth] (0) edge (2);
             \draw[-,  >=stealth] (0) edge (3);
             \draw[-,  >=stealth] (0) edge (4);
             
             \draw[-,  >=stealth] (1) edge (5);
             \draw[-,  >=stealth] (3) edge (5);

             \draw[-,  >=stealth] (1) edge (6);
             \draw[-,  >=stealth] (4) edge (6);

             \draw[-,  >=stealth] (2) edge (7);
             \draw[-,  >=stealth] (4) edge (7);

             \draw[-,  >=stealth] (1) edge (8);
             \draw[-,  >=stealth] (2) edge (8);

             \draw[-,  >=stealth] (3) edge (9);
             \draw[-,  >=stealth] (4) edge (9);

             \draw[-,  >=stealth] (2) edge (10);
             \draw[-,  >=stealth] (3) edge (10);

             \draw[-,  >=stealth] (5) edge (12);
             \draw[-,  >=stealth] (6) edge (12);
             \draw[-,  >=stealth] (7) edge (12);
             \draw[-,  >=stealth] (8) edge (12);
             \draw[-,  >=stealth] (9) edge (12);
             \draw[-,  >=stealth] (10) edge (12);


             
    \end{tikzpicture}
    }
    \caption{The lattice of flats $P_{M_G}$ for the graphic matroid $M_G$ of the $4$-cycle $G = ([4], E)$ where $E = \{\{1,2\}, \{1,3\}, \{2,4\}, \{3,4\}\}$.}
    \label{fig: 4-cycle poset}
    \end{figure}
\end{example}

As shown in Example~\ref{ex: matroid poset invariants}, 
particularly nice $P_{\varphi}$-invariants arise when the rational functions $\beta_s$ and $\alpha_s$ belong to $\mathbb{R}[x_1,\ldots, x_m]$. 
This happens, for example, when $f(s) = 0$ and $s' = c_{i_{s'}}$, the coefficient vector of a coordinate function $f_{i_{s'}}(\theta)$ for each $s'\preceq s$.
When $\Theta$ contains an open subset of $\mathbb R^n$, we have that $f(s)(\theta) = 0$ for all points $\theta\in\Theta$ if and only if the residue polynomial $f(s)$ is the $0$-polynomial in  $\R[\theta_1,\ldots, \theta_n]$.
Moreover, since $s' = c_{i_{s'}}$ then we may take $g(s') = f_{i_{s'}}(\theta)$.  
In particular, the resulting $P_{\varphi}$-invariant is a linear polynomial in the vanishing ideal $I_{\varphi}$. 

\begin{definition}
    \label{def: linear poset invariants}
    A $P_{\varphi}$-invariant is \emph{linear} if the rational functions $\beta_s$  and $\alpha_{s'}$ in~\eqref{eqn: poset invariant} 
    are linear forms. 
\end{definition}

\begin{example}[Marginal independence constraints]
\label{ex: marginal}
The linear $P_{\varphi_{\GG,c}}$-invariants are a generalization of the linear constraints on classical Gaussian DAG models; i.e. the models $\mathcal{M}(\GG,c^\circ)$ where $c^\circ$ denotes the uncoloring from Definition~\ref{def: constant coloring}. 
Let $G=([m],E)$ be a DAG with at least one edge and $c$ be any coloring of $\GG$. 
For any $i,j\in [m]$ the vector  $s = t(i,j)$ satisfies $\psi_{\GG,c}^\ast(g(s)) = \sigma_{ij}$, according to the map~\eqref{eqn: trek map}. 
Suppose there is no trek between $i$ and $j$ in $\GG$, then $s= t(i,j) = 0$.  
Since $s = {0}$, the residue polynomial evaluates to $f(s) = 0$. 
Hence, by Theorem~\ref{thm: mobius}
\[
    0 = \psi_{\GG,c}^\ast(f(s)) = \sum_{s'\preceq s}\mu(s',s)\psi_{\GG,c}^\ast(g(s')) = \mu(0,0)\sigma_{ij} = \sigma_{ij}.
\]
Thus, Theorem~\ref{thm: mobius} recovers the only linear constraints that hold for uncolored graphical models; i.e., those corresponding to the marginal independence of $X_i$ and $X_j$ in $\mathcal{M}(G,c^\circ)$.
In statistics, the constraint $\sigma_{ij} = 0$ on the model $\mathcal{M}(\GG,c)$ says that all distributions on $X_1,\ldots, X_m$ with covariance matrix $\Sigma\in\mathcal{M}(\GG,c)$ satisfy the marginal independence relation $X_i\independent X_j$. 
These constraints have been used in applications to learn graphical representations of the dependence structure on multivariate data, called marginal or unconditional dependence graphs \cite{deligeorgaki2024combinatorial, markham2022transformational}. 
The linear model equivalence problem for uncolored DAGs $(\GG,c^\circ)$ is equivalent to deducing when two DAGs encode the same marginal independence constraints, and a solution to this problem is known \cite{deligeorgaki2024combinatorial, markham2022transformational}.
The linear $P_{\varphi_{\GG,c}}$-invariants generalize the marginal independence constraints for colored DAG models with arbitrary coloring.
\end{example}

By Theorem~\ref{thm: min lin subspace}, the linear $P_{\varphi}$-invariants are linear forms whose coefficient vectors belong to the kernel of the matrix $M_{\varphi}^t$. 
The computation of the linear $P_{\varphi}$-invariants will help us solve the linear equivalence problem (Problem~\ref{prob: linear equivalence}) via Corollary~\ref{cor: lin equiv reduction}. 
This computation involves deducing when the $\alpha_s$ are linear forms, as well as extracting combinatorial data from the poset $P_{\varphi}$, including a formula for the Möbius function $\mu(s',s)$ and a characterization of the residue vectors $r(s)$ for $s\in P_{\varphi}$. 
Generally speaking, we are interested in answering the following:
\begin{problem}\label{prob: P-invariants}
    Let $\varphi:\Theta\to\mathbb R^m$ have coordinate functions $f_1,\ldots, f_m\in\mathbb Z_{\geq 0}[\theta_1,\ldots, \theta_n]$. 
    \begin{enumerate}
        \item When can we take $\alpha_s\in \R[x_1,\ldots, x_m]$ for all $s\in P_{\varphi}$?  When is $\alpha_s$ a linear form?  If $\varphi(\Theta)$ satisfies global rational identifiability, when does $\psi^\ast(g(s))$ satisfy these conditions for all $s\in P_\varphi$?
        \item What is a closed-form formula for the Möbius function $\mu(s',s)$ for $P_{\varphi}$?
        \item What are the residue vectors $r(s)$ for $s\in P_{\varphi}$?  For which $s$ do we have $r(s) = 0$?
    \end{enumerate}
\end{problem}

Problem~\ref{prob: P-invariants} is particularly of interest for the colored DAG models $\mathcal M(\GG,c)$. 
We note a general solution to the latter part of Problem~\ref{prob: P-invariants}~(3). 

\begin{lemma}
\label{lem: lower cover}
    Let $\varphi:\Theta\to\mathbb R^m$ be a rational map with coordinate functions $f_1,\ldots, f_m\in\mathbb Z_{\geq 0}[\theta_1,\ldots, \theta_n]$. 
    Then, $r(s) = 0$ for $s\in P_{\varphi}$ if and only if $s = \bigvee_{i=1}^ks_i$, where $\{s_i\}_{i=1}^k$ is the lower cover of $s$ in $P_{\varphi}$. 
\end{lemma}
\begin{proof}
By definition~\ref{def: residue vector}, $r(s) = s - \bigvee_{s'\prec s} s'$. We can partition the elements strictly preceding $s$ into its lower cover $\{s_1, \ldots, s_k\}$ and the remaining strict predecessors:
    \[
    r(s) = s-\bigvee_{s'\prec s} s' = s-\left(\bigvee_{i=1}^k s_i \right) \vee \left(\bigvee_{s'\prec s, s'\notin \{s_i\}} s'\right).
    \] 
    Note that every strict predecessor $s' \notin \{s_i\}$ must be bounded above by at least one element in the lower cover (i.e., $s' \preceq s_i$ for some $i$). Consequently,
    $
    \bigvee_{s'\prec s, s\notin \{s_i\}} s'\preceq \bigvee_{i=1}^k s_i.
    $
    Therefore, the second term is redundant in the maximum, and we have that
    \[
    r(s) = s-\bigvee_{s'\prec s} s' = s-\left(\bigvee_{i=1}^k s_i \right) \vee \left(\bigvee_{s'\prec s, s\notin \{s_i\}} s'\right)= s-\bigvee_{i=1}^k s_i.
    \] 
    It follows immediately that $r(s) = 0$ if and only if $s = \bigvee_{i=1}^k s_i$.
\end{proof}

We note that Lemma~\ref{lem: lower cover} is a characterization of $s\in P_{\varphi}$ for which $r(s) = 0$ in terms of the poset $P_{\varphi}$.  
Hence, to solve our problem for a semialgebraic set $\varphi(\Theta)$, we need one of two things:  either a sufficient description of the poset $P_{\varphi}$, or a combinatorial interpretation of the elements $s\in P_{\varphi}$ with $r(s) = 0$ in terms of $\varphi$.  
In particular, there remains work to be done in order to apply Lemma~\ref{lem: lower cover} to a given family of semialgebraic subsets $\mathcal F= \{\varphi_i(\Theta_i): \Theta_i\in \mathbb R^{n_i}, \varphi: \Theta_i\to \mathbb R^{m_i}, i\in\mathbb{Z}_{\geq0}\}$. 

\subsection{\texorpdfstring{$\pi$}{pi}-systems}
\label{subsec: pi-graphs}
A simple condition under which Problem~\ref{prob: P-invariants}~(1) has a positive answer is when we can take $\alpha_s = x_{i_s}$ for some $i_s\in[m]$, for all $s\in P_\varphi$. 
This occurs when the ground set of $P_\varphi$ is a subset of $\{0,c_1,\ldots, c_m\}$. 

\begin{definition}
    \label{def: pi-graph}
    Let $\varphi:\Theta\to\mathbb R^m$ be a rational map with coordinate functions $f_1,\ldots, f_m\in\mathbb Z_{\geq 0}[\theta_1,\ldots, \theta_n]$ with respective coefficient sequences $c_1,\ldots, c_m\in \mathbb Z_{\geq 0}^\mathcal{T}$.
    The poset $P_\varphi$ is a called a \emph{$\pi$-system} if its ground set is a subset of $\{0,c_1,\ldots, c_m\}$. 
\end{definition}

\begin{example}
    \label{ex: matroid pi-systems}
    Let $\mathcal V_{\varphi_{M, \mathfrak F'}}$ be an ideal variety for the matroid $M$, as defined in Section~\ref{subsec: ideal varieties}. 
    Then $P_{\varphi_{M, \mathfrak F'}}$ is a $\pi$-system.
    This is because $P_{\varphi_{M, \mathfrak F'}}$ is isomorphic to $\mathfrak F'$ as a subposet of $P_M$, since $\mathfrak F'$ is an order ideal in $P_M$. 
    For special families of matroids, such as graphic matroids, it would be interesting to fully characterize the subsets $\mathfrak F'\subseteq \mathfrak F$ of flats for which $P_{\varphi_{M, \mathfrak F'}}$ is a $\pi$-system.
\end{example}

Characterizing the $P_{\varphi_i}$ that are $\pi$-systems for a family of semialgebraic sets $\mathcal F= \{\varphi_i(\Theta_i): \Theta_i\in \mathbb R^{n_i}, \varphi: \Theta_i\to \mathbb R^{m_i}, i\in\mathbb{Z}_{\geq0}\}$ amounts to deducing when $\bigwedge_{j\in S}c_j$ is either equal to $0$ or a coefficient vector $c_k$ for a coordinate function $f_k$ of $\varphi_i$, for all $S\in 2^{[m]}\setminus\emptyset$.

\begin{figure}[t]
    \begin{subfigure}[b]{0.2\textwidth}
    \centering
    \begin{tikzpicture}
            \node[circle, draw, fill=red!50, minimum size=1pt, inner sep=1pt] (0) at (0, 0) {1};
            \node[circle, draw, fill=red!50, minimum size=1pt, inner sep=1pt] (1) at (1, -0.5) {2};
            \node[circle, draw, fill=red!50, minimum size=1pt, inner sep=1pt] (2) at (2, 0.5) {3};

             \draw[->, blue!70, >=stealth] (0) edge (1);
             \draw[->, blue!70, >=stealth] (1) edge (2);
             
    \end{tikzpicture}
    \caption{}
    \label{fig: colored a}
    \end{subfigure}
    \hfill
    \begin{subfigure}[b]{0.2\textwidth}
    \centering
    \begin{tikzpicture}
            \node[circle, draw, fill=red!50, minimum size=2pt, inner sep=2pt] (0) at (1,0.5) {1};
            \node[circle, draw, fill=red!50, minimum size=2pt, inner sep=2pt] (1) at (1,3.5) {2};
            \node[circle, draw, fill=red!50, minimum size=2pt, inner sep=2pt] (2) at (1.5,1.5) {3};
            \node[circle, draw, fill=red!50, minimum size=2pt, inner sep=2pt] (3) at (1.5,2.5) {4};
            \node[circle, draw, fill=red!50, minimum size=2pt, inner sep=2pt] (4) at (2.5,2) {5};

            \draw[->, blue!70, >=stealth] (0) edge (2);
            \draw[->, blue!70, >=stealth] (2) edge (4);
            \draw[->, blue!70, >=stealth] (1) edge (3);
            \draw[->, blue!70, >=stealth] (3) edge (4);
    \end{tikzpicture}
    \caption{}
    \label{fig: colored b}
    \end{subfigure}
    \hfill
    \begin{subfigure}[b]{0.2\textwidth}
    \centering
    \begin{tikzpicture}
            \node[circle, draw, fill=red!50, minimum size=2pt, inner sep=2pt] (0) at (0,0) {1};
            \node[circle, draw, fill=red!50, minimum size=2pt, inner sep=2pt] (1) at (0,1.5) {2};
            \node[circle, draw, fill=red!50, minimum size=2pt, inner sep=2pt] (2) at (1.5,1.5) {3};

            \draw[->, blue!70, >=stealth] (0) edge (1);
            \draw[->, blue!70, >=stealth] (0) edge (2);
            \draw[->, blue!70, >=stealth] (1) edge (2);
    \end{tikzpicture}
    \caption{}
    \label{fig: colored c}
    \end{subfigure}
    \hfill
    \begin{subfigure}[b]{0.3\textwidth}
    \centering
    \centering
    \begin{tikzpicture}

            \node[circle, draw, fill=red!50, minimum size=1pt, inner sep=1pt] (2) at (0, 0) {2};
            \node[circle, draw, fill=lime!60, minimum size=1pt, inner sep=1pt] (5) at (0, 3) {5};
            \node[circle, draw, fill=purple!50, minimum size=1pt, inner sep=1pt] (10) at (1.25,3.5) {11};
            \node[circle, draw, fill=cyan!50, minimum size=1pt, inner sep=1pt] (0) at (-1, 1.5) {1};
            \node[circle, draw, fill=yellow!50, minimum size=1pt, inner sep=1pt] (1) at (-2, 1.5) {10};
            \node[circle, draw, fill=cyan!50, minimum size=1pt, inner sep=1pt] (3) at (0.25, 1.5) {4};
            \node[circle, draw, fill=cyan!50, minimum size=1pt, inner sep=1pt] (4) at (1.5, 1.5) {3};
            \node[circle, draw, fill=cyan!50, minimum size=1pt, inner sep=1pt] (6) at (-1.2, 3.5) {7};
            \node[circle, draw, fill=gray!50, minimum size=1pt, inner sep=1pt] (7) at (0.2, 3.8) {8};
            \node[circle, draw, fill=pink!50, minimum size=1pt, inner sep=1pt] (8) at (-1, 2.5) {6};
            \node[circle, draw, fill=cyan!50, minimum size=1pt, inner sep=1pt] (9) at (1, 2.5) {9};

            \draw[->, black!80, >=stealth] (6) edge (7);
            \draw[->, brown!80, >=stealth] (7) edge (10);
            \draw[->, blue!80, >=stealth] (6) edge (5);
            \draw[->, magenta!80, >=stealth] (1) edge (8);
            \draw[->, orange!80, >=stealth] (8) edge (5);
            \draw[->, violet!80, >=stealth] (1) edge (2);
            \draw[->, red!80, >=stealth] (0) edge (5);
            \draw[->, blue!80, >=stealth] (0) edge (2);
            \draw[->, blue!80, >=stealth] (3) edge (2);
            \draw[->, blue!80, >=stealth] (3) edge (5);
            \draw[->, green!80, >=stealth] (4) edge (9);
            \draw[->, blue!80, >=stealth] (9) edge (5);
            \draw[->, blue!80, >=stealth] (4) edge (2);
            
    \end{tikzpicture}
    \caption{}
    \label{fig: colored d}
    \end{subfigure}
    
\caption{Some colored DAGs with the constant coloring.  (a) (b) are $\pi$-graphs.  (a), (b) and (c) have linearly independent nonzero residue vectors. 
} 
\label{fig: colored DAGs}
\end{figure}
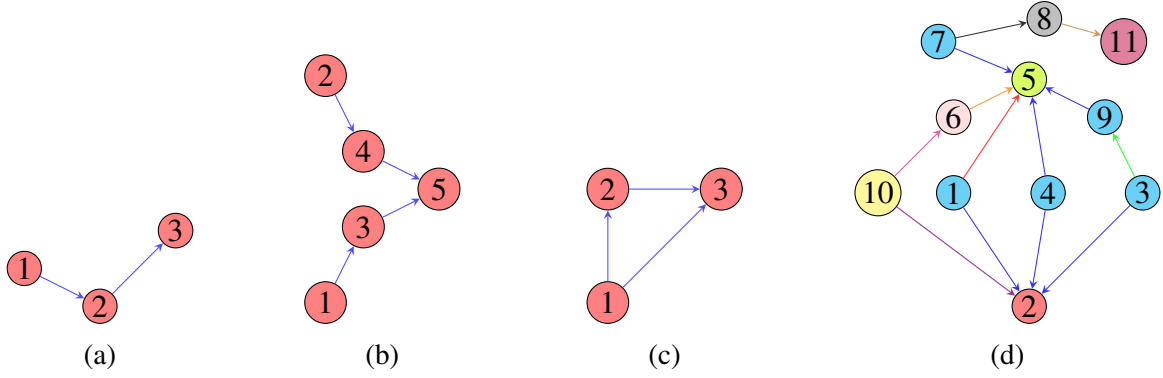

\begin{example}
    \label{ex: pi-graph and non-pi-graph}
    For a colored DAG $(\GG,c)$,  $P_{\varphi_{\GG,c}}$ is a $\pi$-system if and only if for all nonempty sets $\{(i_1,j_1),\ldots, (i_k,j_k)\}\subset [m]\times [m]$ there exists $(i,j)$ such that $t(i,j) = t(i_1,j_1)\wedge \cdots \wedge t(i_k,j_k)$. 
    We call $(\GG,c)$ a \emph{$\pi$-graph} if $P_{\varphi_{\GG,c}}$ is a $\pi$-system, and we say $s\in P_{\varphi}$ is \emph{realized} by $(\GG,c)$ if $s\in \{t(i,j): i,j\in [m]\}\cup\{0\}$. 
    Consider the colored DAGs (a), (b) and (c) in Figure~\ref{fig: colored DAGs}. 
    It is easy to check that (a) and (b) are $\pi$-graphs.
    For graph (c), the poset contains the vectors $t(2,2)\wedge t(1,3) = (0,0,1,0,0)$ and $t(3,3)\wedge t(2,3) = (0,0,1,1,0)$ which are the coefficient vectors of the trek polynomials $\omega\lambda^2$ and $\omega\lambda^2 + \omega\lambda^3$ respectively.  
    It is also easy to check that no trek polynomial $p_{i,j}^{(\GG,c^\ast)}$ in (c) is equal to either of these two polynomials. 
    Therefore, the vectors are not realized by $(\GG,c)$. 
    Hence, (c) is not a $\pi$-graph.
\end{example}


For a colored DAG, being a $\pi$-graph is a sufficient condition to guarantee that the inversion formula \eqref{eqn:mobius} will yield linear constraints satisfied by the model $\mathcal{M}(G,c)$ whenever $r(s) = 0$. 
This motivates the following problem. 

\begin{problem}
\label{prob: ground set}
    Characterize the colored DAGs $(\GG,c)$ that are $\pi$-graphs. 
\end{problem}

More generally, the same question can be asked for other families of semialgebraic sets, such as for the matroid flat varieties or degenerate subvarieties of determinantal varieties in Section~\ref{section:prelim}.  
However, finite $\pi$-systems are closely related to finite topologies, which are known to be difficult to enumerate \cite{kleitman1975asymptotic}. 
So solving Problem~\ref{prob: ground set} for colored DAGs will depend on the structure of the graphs $(\GG,c)$. 
Similarly, solving the same problem for other families of semialgebraic sets will likely depend on the combinatorial properties of the specific objects enumerated by their coordinate functions.

\subsection{The ideal of the poset \texorpdfstring{$P_{\varphi}$}{}}
\label{subsec: the ideal of the poset}
When $P_\varphi$ is not a $\pi$-system, the vanishing ideal $I_{\varphi}$ of $\varphi(\Theta)$ is an elimination ideal of an ideal associated to the poset $P_{\varphi}$.  
This connection underlies the  way in which the poset $P_\varphi$ combinatorially encodes the kernel of the matrix $M_\varphi^t$ from Section~\ref{subsec: minimal linear subspace}. 


For a rational map $\varphi: \Theta\to \mathbb R^m$ with coordinate functions $f_1,\ldots, f_m\in\mathbb Z_{\geq 0}[\theta_1,\ldots, \theta_n]$, we let  $c_1,\ldots, c_m$ denote their coefficient vectors living $\mathbb Z_{\geq 0}^{\mathcal T}$ where $\mathcal T$ is the set of monomials defined in~\eqref{eqn: support set}. 
Let $2^{[m]}$ denote the power set of $[m]$, and let $\mathbb{R}[P_{\varphi}]$ denote the polynomial ring $\mathbb{R}[z_S : S\in 2^{[m]}\setminus\{\emptyset\}]$.
Define the ring homomorphism
\begin{equation}
    \label{eqn: poset map}
    \begin{split}
    \hat\varphi^\ast:&\ \mathbb{R}[P_{\varphi}] \to \mathbb{R}[\theta_1,\ldots, \theta_n]; \\ 
    \hat\varphi^\ast:&\ z_S\mapsto \bigwedge_{i\in S}c_i \cdot \mathbb{T}
    \end{split}
\end{equation}
where $\mathbb{T} = (\tau : \tau\in \mathcal{T})$ is a column vector.  
Let $I_{{\hat\varphi}} = \ker(\hat\varphi^\ast)$. 
Note that $\{i\}\in 2^{[m]}\setminus\{\emptyset\}$ for all $i\in [m]$.
Hence, 
$
\mathbb{R}[x_1,\ldots, x_m] \cong \mathbb{R}[z_{\{1\}},\ldots, z_{\{m\}}] \subset \R[P_{\varphi}].
$
Note further that $\hat\varphi^\ast$ is the algebraic pullback of the map
\begin{equation}
    \label{eqn: poset param map}
    \begin{split}
    &\hat\varphi: \Theta \to \mathbb{R}^{2^{[m]}\setminus\{\emptyset\}};\\ 
    &\hat\varphi: (\theta_1,\ldots, \theta_n)\mapsto{\left(\bigwedge_{i\in S}c_i \cdot \mathbb{T}\right)_{S \in 2^{[m]}\setminus\{\emptyset\}}}. 
    \end{split}
\end{equation}
The map $\hat \varphi$ is a coordinate extension of the map $\varphi$ into a (generally higher dimensional) affine space.  
That is, $\hat\varphi(\Theta)$ is a lifting of the semialgebraic set $\varphi(\Theta)$ from $\mathbb R^m$ to $\mathbb R^{P_\varphi}\simeq \mathbb R^{2^{[m]}\setminus\{\emptyset\}}$ that is parameterized by coordinate functions given by the elements of the poset $P_\varphi$.
In particular, $I_{\hat\varphi}$ is the vanishing ideal of $\hat\varphi(\Theta)$. 

\begin{lemma}
    \label{lem: existence of lifted poset invariants}
    Suppose that there exists $s\in P_\varphi$ with residue vector $r(s) = 0$.  Then there exist (linear) $P_{\hat\varphi}$-invariants.
\end{lemma}

\begin{proof}
    For every $s'\preceq s$, the rational functions $\beta_s = 0$ and $\alpha_{s'} = z_{S'}$ satisfies the conditions of Definition~\ref{def: poset invariants} so long as $S'\in 2^{[m]}\setminus \{\emptyset\}$ satisfies $s' = \bigwedge_{i\in S'}c_i \cdot \mathbb{T}$. 
\end{proof}

\begin{remark}
    \label{rem: lifted poset invariants via birationality}
    Aside from the $P_{\hat\varphi}$-invariants identified in Lemma~\ref{lem: existence of lifted poset invariants}, we can also obtain potentially nonlinear $P_{\hat\varphi}$-invariants.  
    One situation in which this can happen is when $\varphi(\Theta)$ satisfies global rational identifiability. In this case, we take $\alpha_{s'} = z_{S'}$ as in Lemma~\ref{lem: existence of lifted poset invariants}, $\beta_s = \psi^\ast(f_s(\theta))$ and the polynomial $p$ that clears the denominator of $\beta_s$. 
    This method applies, for example, to colored DAG models. 
\end{remark}

Note that the posets $P_\varphi$ and $P_{\hat\varphi}$ are equal.  
However, the $P_\varphi$-invariants are polynomials in $I_\varphi$ while the $P_{\hat\varphi}$-invariants are polynomials in $I_{\hat{\varphi}}$, with equality occurring when $\hat \varphi = \varphi$; e.g., when $P_\varphi$ is a $\pi$-system.  
The $P_{\hat\varphi}$-invariants are the natural way to identify the linear forms defining the linear span of $\varphi(\Theta)$. 
This is accomplished via the relationship between the vanishing ideals $I_\varphi$ and $I_{\hat\varphi}$.

\begin{lemma}
    \label{lem: elimination}
    $I_{\varphi} = I_{\hat\varphi}\cap \mathbb{R}[x_1,\ldots, x_m]$. 
\end{lemma}

\begin{proof}
    Let $f \in I_{\varphi}$.  
    Then $f\in \ker(\varphi^\ast)$, and hence $f\in \ker(\hat\varphi^\ast)\cap \mathbb R[x_1,\ldots, x_m]$ since $\hat\varphi$ is a coordinate extension of $\varphi$. 
    Hence, $f\in I_{\hat\varphi}\cap \mathbb R[x_1,\ldots, x_m]$.
    Conversely, if $f\in I_{\hat\varphi}\cap \mathbb R[x_1,\ldots, x_m]$, then $f$ is a polynomial function in the variables $x_1,\ldots,x_m$ that evaluates to $0$ on all points $\hat\varphi(\theta)$ for $\theta \in \Theta$. 
    Thus, $f$ also evaluates to $0$ at all points $\varphi(\theta)$ for $\theta \in \Theta$.
    Therefore, $f\in I_{\varphi}$. 
\end{proof}

The following corollary is an immediate consequence of Lemma~\ref{lem: elimination}, and it has consequences for the linear equivalence problem  (Problem~\ref{prob: linear equivalence}). 

\begin{corollary}
    \label{cor: degree 1 components}
    The degree $1$ component of the ideal $I_{\varphi}$ equals the degree $1$ component of the ideal $I_{\hat\varphi}\cap \mathbb{R}[x_1,\ldots, x_m]$. 
\end{corollary}

\subsection{Solving the linear equivalence problem combinatorially}
\label{subsec: combinatorial linear equiv}

By Theorem~\ref{thm: min lin subspace}, the kernel of the matrix $M_{\hat\varphi}^t$ yields the set of linear forms
defining 
the linear span of $\hat\varphi(\Theta)$. 
For a family of semialgebraic sets $\mathcal F= \{\varphi_i(\Theta_i): \Theta_i\in \mathbb R^{n_i}, \varphi: \Theta_i\to \mathbb R^{m_i}, i\in\mathbb{Z}_{\geq0}\}$, 
the following theorem allows us to reduce the linear equivalence problem for $\mathcal F$ (Problem~\ref{prob: linear equivalence}) to the combinatorial problems~\ref{prob: P-invariants}~(2) and~(3). 
Specifically, Problem~\ref{prob: linear equivalence} is solved using combinatorial data; namely, a formula for the Möbius function of $P_{\varphi}$ and a description of the residue vectors $r(s)$.  

\begin{theorem}
    \label{thm: lin independence}
    Let $\varphi(\Theta)$ be a semialgebraic set, with $\Theta$ containing an open subset of $\mathbb R^n$ and $\varphi:\theta\to (f_1(\theta),\ldots, f_m(\theta))$, where the coordinate functions have coefficient vectors $c_1,\ldots, c_m\in\mathbb Z_{\geq 0}^\mathcal{T}$.  
    If the nonzero residue vectors $r(s)$ for $s\in P_{\varphi}$ are linearly independent, then the linear span of $\hat\varphi(\Theta)$ is defined by 
    \begin{itemize}
        \item the linear $P_{\hat\varphi}$-invariants for which $r(s)= 0$, and
        \item the relations $z_S - z_{S'}$ for subsets $S, S'\in2^{[m]}\setminus\{\emptyset\}$ satisfying $\bigwedge_{i\in S}c_i = \bigwedge_{i\in S'}c_i\in P_{\varphi}$.
    \end{itemize} 
\end{theorem}

\begin{proof}
    Recall that the rows of the matrix $M_{\hat\varphi}$ are indexed by the subsets $S \in 2^{[m]}\setminus\{\emptyset\}$. 
    Consequently, the left nullspace of $M_{\hat\varphi}$ decomposes naturally into two types of linear relations: trivial relations arising from duplicate rows, and combinatorial relations arising from the poset structure. 
    First, any subsets $S, S'$ that evaluate to the same element in $P_{\varphi}$ correspond to identical rows in $M_{\hat\varphi}$, yielding the relations $z_S - z_{S'} = 0$. 

    Consider the submatrix formed by the unique rows of $M_{\hat\varphi}$ indexed by the elements $s \in P_{\varphi}$. 
    Since $s = \sum_{s'\preceq s}r(s')$ according to Proposition~\ref{prop: poset sum}, the Möbius inversion formula in Theorem~\ref{thm: mobius} amounts to applying elementary row operations, collected in a matrix $E$, to the matrix $M_{\hat\varphi}$.  
    Let $F_{\hat\varphi} = EM_{\hat\varphi}$.
    The matrix $F_{\hat\varphi}$ whose rows are the residue vectors $r(s)$ for $s\in P_{\varphi}$ and zero vectors (from duplicate rows) is row equivalent to $M_{\hat\varphi}$.  
    Since the nonzero rows are linearly independent, it follows that the kernel of $M_{\hat\varphi}^t$ corresponds to the rows of zeros in $F_{\hat\varphi}$. 
    Let $d = \dim(\ker(M_{\hat\varphi}^t))$.  
    Since there are $d$ rows of zeros in $F_{\hat\varphi}$, each corresponding to a $P_{\hat\varphi}$-invariant vanishing on $\hat\varphi(\Theta)$, it follows that these $P_{\hat\varphi}$-invariants have a solution set equal to the minimal linear subspace containing $\hat\varphi(\Theta)$. 
\end{proof}

When $P_{\varphi}$ is a $\pi$-system, we have $\R[P_{\varphi}] = \R[x_1,\ldots, x_m]$, and the relations in Theorem~\ref{thm: lin independence} define the linear span of $\varphi(\Theta)$. 
Otherwise, by Corollary~\ref{cor: degree 1 components}, a simple (linear) elimination step applied to the ideal generated by the relations in Theorem~\ref{thm: lin independence} recovers the ideal of the linear span of $\varphi(\Theta)$.
When the linear independence assumption in Theorem~\ref{thm: lin independence} fails, one can apply the combinatorial procedure iteratively using the poset on the residue vectors to further row-reduce the matrix $M_{\hat\varphi}$ (see Section~\ref{subsec: iterated}).

\begin{example}[Colored DAG models]
    \label{ex: elimination}
    We first consider some concrete examples for colored DAG models. 
    Consider the four colored DAGs in Figure~\ref{fig: colored DAGs}.\\
    
    (a) The colored DAG in Figure~\ref{fig: colored a} is a $\pi$-graph (see Example \ref{ex: pi-graph and non-pi-graph}).  
    Hence, $\R[P_{\varphi_{\GG,c^\ast}}] = \R[\sigma_{ij}: 1\leq i \leq j\leq m]$. 
    It follows that the linear $P_{\hat\varphi_{\GG,c^\ast}}$-invariants are simply the linear $P_{\varphi_{\GG,c^\ast}}$-invariants. 
    Moreover, its set of nonzero residue vectors are linearly independent. 
    Hence, by Theorem~\ref{thm: lin independence}, the linear $P_{\varphi_{\GG,c^\ast}}$-invariants for the $t(i,j)\in P_{\GG,c^\ast}$ with $r(t(i,j)) = 0$ appear when row-reducing $M_{\varphi_{\GG,c^\ast}}$ according to the Möbius inversion formula in Theorem~\ref{thm: mobius}.  
    They are seen to define the linear span of $\mathcal{M}(\GG,c^\ast)$ by performing this row-reduction on the augmented matrix $[M_{\varphi_{\GG,c^\ast}}\mid \Sigma]$, where $\Sigma = (\sigma_{ij} : 1\leq i \leq j \leq m ) = \varphi_{\mathcal{T}}(\theta)$ for arbitrary $\theta = (\lambda_1,\ldots, \lambda_e,\omega_1,\ldots, \omega_n)\in\R^{e}\times \R^n_{>0}$. 
    The poset $P_{\varphi_{\GG,c^\ast}}$, with nodes labeled by residue vectors in red, and matrix the $[M_{\varphi_{\GG,c^\ast}}\mid \Sigma]$ after the row-reduction (that is, $[F_{\varphi_{\GG,c^\ast}} \mid E \Sigma]$ using the notation in the proof of Theorem \ref{thm: lin independence}) are in Figure~\ref{fig: poset and matrix a}.  
    Exactly one row is equal to $0$ (which is the single $0$ residue in the poset to the left), and its linear equation in the $\sigma_{ij}$-coordinates is the associated linear $P_{\varphi_{\GG,c^\ast}}$-invariant.  
    Note that coefficients of this linear form are precisely the entries of the single kernel vector for $M_{\varphi_{\GG,c^\ast}}^t$ found in Example~\ref{ex: motivation}.\\
    
    \begin{figure}[t]
    \begin{subfigure}[b]{0.5\textwidth}
    \centering
    {\tiny
    \begin{tikzpicture}[scale = 1.1]
            \node (0) at (0, 0) {(0,0,0,0,0)};
            \node (1) at (-2, 1) {(0,1,0,0,0)};
            \node (2) at (0, 1) {(0,0,1,0,0)};
            \node (3) at (2, 1) {(1,0,0,0,0)};

            \node (4) at (-2, 2) {(0,1,0,1,0)};
            \node (5) at (2, 2) {(1,0,1,0,0)};

            \node (6) at (2, 3) {(1,0,1,0,1)};

             \draw[-,  >=stealth] (0) edge (1);
             \draw[-,  >=stealth] (0) edge (2);
             \draw[-,  >=stealth] (0) edge (3);

             \draw[-,  >=stealth] (1) edge (4);
             \draw[-,  >=stealth] (2) edge (5);
             \draw[-,  >=stealth] (3) edge (5);

             \draw[-,  >=stealth] (5) edge (6);


            \node (r1) at (-2 - 0.55, 1 - 0.3) {\textcolor{red}{(0,1,0,0,0)}};
            \node (r2) at (0 + 0.55, 1 - 0.3) {\textcolor{red}{(0,0,1,0,0)}};
            \node (r3) at (2 + 0.55, 1 - 0.3) {\textcolor{red}{(1,0,0,0,0)}};
            \node (r4) at (-2 - 0.55, 2 - 0.3) {\textcolor{red}{(0,0,0,1,0)}};
            \node (r5) at (2 + 0.55, 2 - 0.3) {\textcolor{red}{(0,0,0,0,0)}};
            \node (r6) at (2 + 0.55, 3 - 0.3) {\textcolor{red}{(0,0,0,0,1)}};
             
    \end{tikzpicture}
    }
    \caption{}
    \label{fig: poset a}
    \end{subfigure}
    \hfill
    \begin{subfigure}[b]{0.4\textwidth}
    \centering
    {\small
    \[
    \begin{pmatrix}
        1 & 0 & 0 & 0 & 0 & \sigma_{11}\\
        0 & 1 & 0 & 0 & 0 & \sigma_{12}\\
        0 & 0 & 0 & 0 & 0 & \sigma_{22} - \sigma_{13} - \sigma_{11}\\
        0 & 0 & 1 & 0 & 0 & \sigma_{13}\\
        0 & 0 & 0 & 1 & 0 & \sigma_{23} - \sigma_{12}\\
        0 & 0 & 0 & 0 & 1 & \sigma_{33} - \sigma_{11} - \sigma_{13}
    \end{pmatrix}
    \]
    }
    \caption{}
    \label{fig: matrix a}
    \end{subfigure}

\caption{The poset $P_{\varphi_{\GG,c^\ast}}$ for $(\GG,c)$ in Figure~\ref{fig: colored a}, and the augmented matrix $[M_{\varphi_{\GG,c^\ast}}\mid \Sigma]$ following row-reduction according to the Möbius inversion formula in Theorem~\ref{thm: mobius}.}
\label{fig: poset and matrix a}
\end{figure}
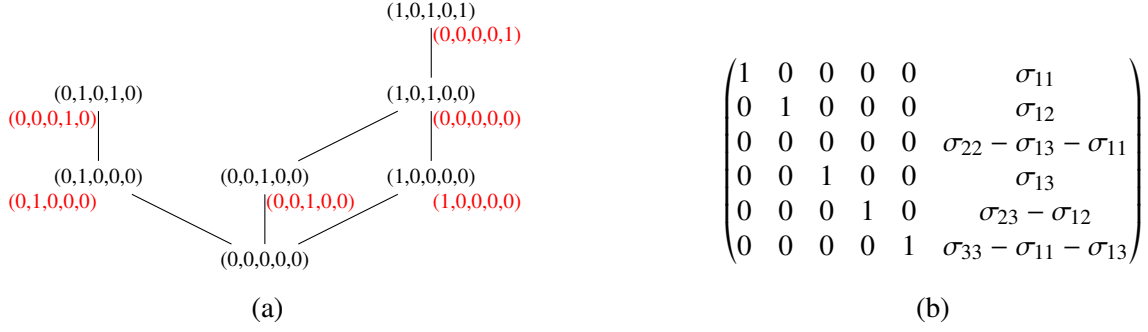

(b) The colored DAG in Figure~\ref{fig: colored b} has the same poset as in Figure~\ref{fig: poset a}, with the exception that the top-most element $(1, 0, 1, 0, 1)$ is replaced by $s = t(5,5) =  (1, 0, 2, 0, 2)$.  
This new element has residue vector $r(s) = (0, 0, 1, 0, 2)$. 
Since the residue vectors are linearly independent, Theorem~\ref{thm: lin independence} again applies, and the minimal linear subspace containing $\mathcal{M}(\GG,c^\ast)$ is the zero locus of the linear forms 
\begin{gather*}
    \sigma_{33} - \sigma_{11} - \sigma_{15}, \quad \sigma_{11} - \sigma_{22},\quad \sigma_{13} - \sigma_{24},\quad\sigma_{33} - \sigma_{44},\quad \sigma_{35} - \sigma_{45},\\ 
    \sigma_{15} - \sigma_{25}, \quad
    \sigma_{34}, \quad 
    \sigma_{23}, \quad 
    \sigma_{14}, \quad 
    \sigma_{12}.
\end{gather*}

Unlike the graph in (a), we have binomial relations in addition to the linear $P_{\varphi_{\GG,c^\ast}}$-invariant. 
These relations are easily read from the poset by noting equalities $t(i,j) = t(k,\ell)$ between trek vectors.\\

(c) The graph in Figure~\ref{fig: colored c} is not a $\pi$-graph, meaning that the ground set contains the elements $z_{22,13} =t(2,2)\wedge t(1,3) = (0,0,1,0,0)$ and $z_{23, 33} = t(3,3)\wedge t(2,3) = (0,0,1,1,0)$ in addition to the vectors $t(i,j)$ for $i,j\in[m]$. 
These two elements correspond, respectively, to the variables $z_{22,13}, z_{23, 33}\in \mathbb R[P_{\varphi_{\GG,c^\ast}}]$ that are not in $\mathbb R[\sigma_{ij}: 1\leq i\leq j\leq m]$.
The poset and the associated augmented matrix are depicted in Figure~\ref{fig: poset and matrix c}. 
Again, the nonzero residue vectors are linearly independent, and in this case there are no relations $z_{S} = z_{S'}$. 
So, by Theorem~\ref{thm: lin independence}, the three linear forms defining the minimal linear subspace containing $\hat\varphi_{{\GG,c^\ast}}(\R^{e}\times \R^n_{>0})$ are the $P_{\hat{\varphi}_{\GG,c^\ast}}$-invariants 
\[
\sigma_{22} - \sigma_{11} - z_{22,13},\quad \sigma_{13} - \sigma_{12} - z_{22,13}, \quad \sigma_{23} - \sigma_{12} - z_{23, 33}, 
\]
which include variables $z_{22,13},z_{23, 33}$ for elements of the poset not corresponding to $t(i,j)$ for $i,j\in[m]$. 
By Corollary~\ref{cor: degree 1 components}, the linear ideal defining the linear span of $\mathcal{M}(\GG,c^\ast)$ is obtained by eliminating these variables from the ideal generated by these three linear forms.  This yields a single generator
\[
\sigma_{22} - \sigma_{11} - (\sigma_{13} - \sigma_{12}). 
\]

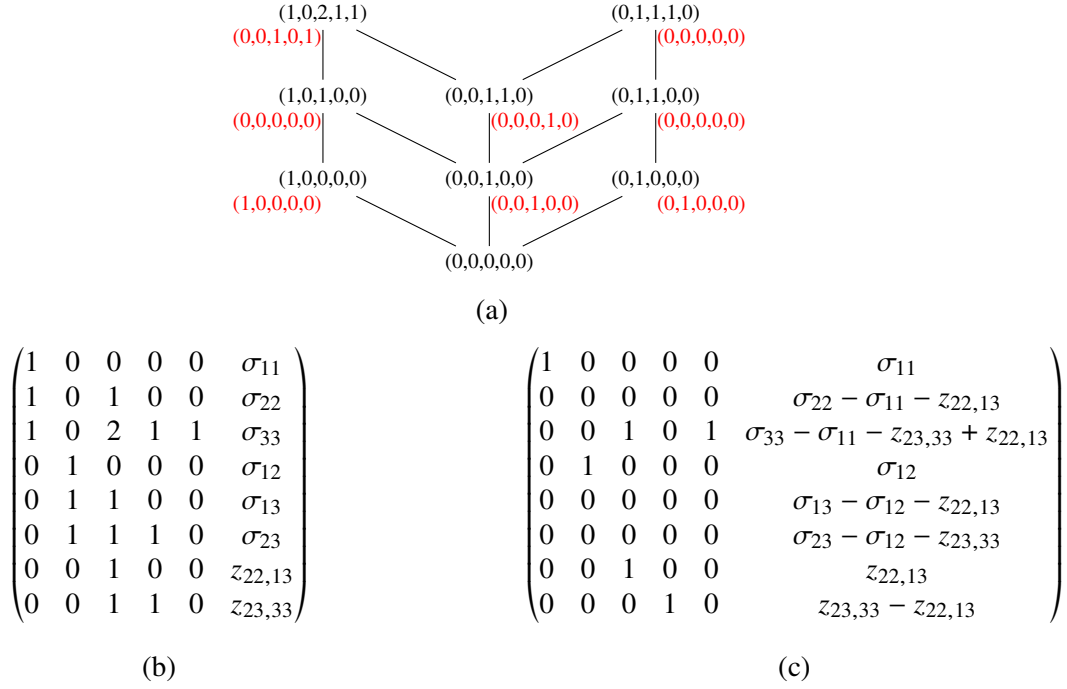
\begin{figure}[t]
    \begin{subfigure}[b]{0.9\textwidth}
    \centering
    {\tiny
    \begin{tikzpicture}[scale = 1.1]
            \node (0) at (0, 0) {(0,0,0,0,0)};
            \node (1) at (-2, 1) {(1,0,0,0,0)};
            \node (2) at (0, 1) {(0,0,1,0,0)};
            \node (3) at (2, 1) {(0,1,0,0,0)};

            \node (4) at (-2, 2) {(1,0,1,0,0)};
            \node (7) at (0, 2) {(0,0,1,1,0)};
            \node (5) at (2, 2) {(0,1,1,0,0)};

            \node (6) at (2, 3) {(0,1,1,1,0)};
            \node (8) at (-2, 3) {(1,0,2,1,1)};

             \draw[-,  >=stealth] (0) edge (1);
             \draw[-,  >=stealth] (0) edge (2);
             \draw[-,  >=stealth] (0) edge (3);

             \draw[-,  >=stealth] (1) edge (4);
             \draw[-,  >=stealth] (2) edge (4);
             \draw[-,  >=stealth] (2) edge (5);
             \draw[-,  >=stealth] (2) edge (7);
             \draw[-,  >=stealth] (3) edge (5);

             \draw[-,  >=stealth] (5) edge (6);
             \draw[-,  >=stealth] (4) edge (8);
             \draw[-,  >=stealth] (7) edge (8);
             \draw[-,  >=stealth] (7) edge (6);


            \node (r1) at (-2 - 0.55, 1 - 0.3) {\textcolor{red}{(1,0,0,0,0)}};
            \node (r2) at (0 + 0.55, 1 - 0.3) {\textcolor{red}{(0,0,1,0,0)}};
            \node (r7) at (0 + 0.55, 2 - 0.3) {\textcolor{red}{(0,0,0,1,0)}};
            \node (r3) at (2 + 0.55, 1 - 0.3) {\textcolor{red}{(0,1,0,0,0)}};
            \node (r4) at (-2 - 0.55, 2 - 0.3) {\textcolor{red}{(0,0,0,0,0)}};
            \node (r5) at (2 + 0.55, 2 - 0.3) {\textcolor{red}{(0,0,0,0,0)}};
            \node (r6) at (2 + 0.55, 3 - 0.3) {\textcolor{red}{(0,0,0,0,0)}};
            \node (r8) at (-2 - 0.55, 3 - 0.3) {\textcolor{red}{(0,0,1,0,1)}};
             
    \end{tikzpicture}
    }
    \caption{}
    \label{fig: poset c}
    \end{subfigure}
    
    \begin{subfigure}[b]{0.45\textwidth}
    \centering
    {\small
    \[
    \begin{pmatrix}
        1 & 0 & 0 & 0 & 0 & \sigma_{11}\\
        1 & 0 & 1 & 0 & 0 & \sigma_{22}\\
        1 & 0 & 2 & 1 & 1 & \sigma_{33}\\
        0 & 1 & 0 & 0 & 0 & \sigma_{12}\\
        0 & 1 & 1 & 0 & 0 & \sigma_{13}\\
        0 & 1 & 1 & 1 & 0 & \sigma_{23}\\
        0 & 0 & 1 & 0 & 0 & z_{22,13}\\
        0 & 0 & 1 & 1 & 0 & z_{23,33}\\
    \end{pmatrix}
    \]
    }
    \caption{}
    \label{fig: matrix c}
    \end{subfigure}
    \hfill
    \begin{subfigure}[b]{0.5\textwidth}
    \centering
    {\small
    \[
    \begin{pmatrix}
        1 & 0 & 0 & 0 & 0 & \sigma_{11} \\
        0 & 0 & 0 & 0 & 0 & \sigma_{22} - \sigma_{11} - z_{22,13}\\
        0 & 0 & 1 & 0 & 1 & \sigma_{33} - \sigma_{11} - z_{23,33} + z_{22,13}\\
        0 & 1 & 0 & 0 & 0 & \sigma_{12}\\
        0 & 0 & 0 & 0 & 0 & \sigma_{13} - \sigma_{12} - z_{22,13}\\
        0 & 0 & 0 & 0 & 0 & \sigma_{23} - \sigma_{12} - z_{23, 33}\\
        0 & 0 & 1 & 0 & 0 & z_{22,13}\\
        0 & 0 & 0 & 1 & 0 & z_{23,33} - z_{22,13}\\
    \end{pmatrix}
    \]
    }
    \caption{}
    \label{fig: matrix c reduced}
    \end{subfigure}

\caption{The poset $P_{\varphi_{\GG,c^\ast}}$ for $(\GG,c)$ in Figure~\ref{fig: colored c}, and the augmented matrix $[M_{\varphi_{\GG,c^\ast}}\mid \Sigma]$ both before and after row-reduction according to the Möbius inversion formula in Theorem~\ref{thm: mobius}.}
\label{fig: poset and matrix c}
\end{figure}

(d) This last example, depicted in Figure~\ref{fig: colored d}, shows how linear dependencies among the nonzero residue vectors can arise.  
This happens when a subset of treks that has already appeared in a minimal form, perhaps as a $t(i,j)$ or an intersection thereof, appears in another $t(k,\ell)$ but with excessive multiplicity. 
In the graph (d), note that there is a unique trek, with trek monomial $\omega\lambda^3$, between nodes $5$ and $11$.  
Hence, the vector $t(5,11)$ will be an atom in $P_{\varphi_{\GG,c^\ast}}$. 
A second atom is $t(3,5)$, with the unique trek $\omega\lambda^2$.  
Both of these atoms are equal to their residue vectors. 
They are also covered by $t(2,5)$ where there are two copies of each trek $\omega\lambda^2$ and $\omega\lambda^3$.  
Hence, the residue $r(t(2,5))$ will contain one copy of each trek, inducing the dependence relation
\[
r(t(2,5)) - r(t(3,5)) - r(t(5,11)) = 0. 
\]
This demonstrates that the combinatorial reduction of Möbius inversion is not always sufficient to fully convert the matrix $M_{\varphi_{\GG,c^\ast}}$ into a fully reduced form.  
On the other hand, these additional relations on the residues are easily seen to be extracted by placing the same partial order on the set of residues, and repeating the process for $P_{\GG,c^\ast}$.
In particular, an iteration of the combinatorial process above may be needed to extract the desired algebraic result for some graphs (see Section~\ref{subsec: iterated}).  
This appears to be mainly a concern for more refined colorings. 
For instance, changing the color of any one edge or node in the graph in Figure~\ref{fig: colored d}, produces a colored DAG with linearly independent nonzero residues. 
\end{example}

\begin{example}[Linear phylogenetic invariants]
    \label{ex: phylogenetic elimination}
    Recall the Jukes-Cantor model on the 4-leaf tree $\mathcal{T}$ introduced in Section~\ref{subsec: phylogenetics}, parameterized by the variables $a_1, \ldots, a_5$ and $b_1, \ldots, b_5$. 
    The reduced map $\varphi_{JC}: \Theta \to \mathbb{R}^{15}$ defines a semialgebraic set $\varphi_{JC}(\Theta)\subset \mathbb{R}^{15}$. 
    
    The poset $P_{\varphi_{JC}}$ constructed from the coefficient vectors of the 15 coordinate polynomials --which have length 32, corresponding to the distinct monomials in $\mathcal{T}$-- has a ground set consisting of 142 elements (including the minimal zero element). The elimination process described in Theorem~\ref{thm: lin independence} recovers exactly two linear constraints for this model, emerging through two distinct combinatorial mechanisms. 
    
    The first linear constraint is obtained directly by applying the Möbius inversion formula from Theorem~\ref{thm: mobius} to elements $s \in P_\varphi$ where the residue vector $r(s)$ vanishes:
    \begin{equation}
        \label{eq: lake1}
        f_1 = p_{\texttt{ACCA}} + p_{\texttt{ACAG}} + p_{\texttt{ACGC}} - p_{\texttt{ACAC}} - p_{\texttt{ACCG}} - p_{\texttt{ACGA}} = 0.
    \end{equation}
    The second constraint, however, is obtained by identifying a linear dependence among the non-zero residue vectors themselves, analogous to the algebraic behavior demonstrated in Example~\ref{ex: elimination}(d):
    \begin{equation}
        \label{eq: lake2}
        f_2 = p_{\texttt{ACGT}} + p_{\texttt{ACCA}} - p_{\texttt{ACGA}} - p_{\texttt{ACCG}} = 0.
    \end{equation}
    
    These two linear forms define the minimal linear subspace containing $\varphi_{JC}(\Theta)$. Since the linear span is invariant under change of basis, this subspace is equivalently spanned by the classical \emph{Lake invariants} for the 4-leaf topology, which are given by $f_2$ and $f_1 - f_2$ \cite{lake1987}. These invariants are fundamentally important in phylogenetics because they successfully distinguish between the three possible 4-leaf tree topologies and remain robust to data generated by mixtures of trees. By applying our framework, we bypass traditional algebraic elimination algorithms, providing a purely combinatorial derivation of the invariants required to solve the tree distinguishability problem. The code verifying the poset dimension and elimination ideal are provided here: \url{https://github.com/marinagarrote/poset-parametrizations}. 
\end{example}

\begin{example}
    \label{ex: linear relations for graphic matroids}
    Let $\mathcal V_{\varphi_{M,\mathfrak F'}}$ be an ideal variety, as defined in Section~\ref{subsec: ideal varieties}.  
    Then $P_{\varphi_{M, \mathfrak F'}}$ is isomorphic to $\mathfrak F'$ as a subposet of the lattice of flats $P_M$ of $M$.
    Moreover, $P_{\varphi_{M,\mathfrak F'}}$ is a $\pi$-system, meaning that $\hat\varphi_{M, \mathfrak F'} = \varphi_{M, \mathfrak F'}$. 
    By Example~\ref{ex: poset for graphic matroid}, the nonzero residue vectors of $P_{\varphi_{M,\mathfrak F'}}$ are $r(F)$ where $F\in \mathfrak F'$ is an atom in $P_M$, and these vectors are all linearly independent. 
    Generalizing Example~\ref{ex: matroid poset invariants}, the linear $P_{\varphi_{M, \mathfrak F'}}$-invariants are $\sum_{F'\preceq F \textrm{ in } \mathfrak F'}\mu(F',F)x_{F'}\in I_{\varphi_{M, \mathfrak F'}}$ for $F\in \mathfrak F'$ a non-atom in $P_M$. 
    By Theorem~\ref{thm: lin independence}, these polynomials define the linear span of the ideal variety $\mathcal V_{\varphi_{M,\mathfrak F'}}$. 
\end{example}

\begin{example}
    \label{ex: linear relations for secants}
    Let $\mathcal V_{\varphi_1}$ be the variety corresponding to the map $\varphi_1$ introduced in Example \ref{ex: poset for a secant subvariety}. The non-zero residue vectors are linearly independent as they are standard basis vectors. By Theorem \ref{thm: lin independence} it follows that the linear span of the image of $\hat{\varphi}_1$ is given by linear $P_{\hat{\varphi_1}}$-invariants.

    To find the exact relations we first note that the elements $s\in P_{\hat{\varphi_1}}$ with $r(s)$ are exactly the coefficient vectors coming from the map $\varphi_1$. Denote the coefficient vector of the $ij$-th coordinate function of $\varphi_1$ as $c_{ij}$. Then, for example, $c_{11} = (1, 0, 0, 1, 0, 0)$ and $r(c_{ij})=0$ for any $i, j\in [3]$.
    
    Each $c_{ij}$ has two standard basis vectors in the lower cover. The meet of any two standard basis vectors is $\hat{0}$. Thus, the $P_{\hat{\varphi}_1}$-invariants have the following form: for every $c_{ij}$ we have $x_{ij} - z_{ij, ik} - z_{ij, kj}=0$ for any $k\neq i$ or $ k\neq j$. Indeed, the elements $c_{ij}\wedge c_{ik}$ and $c_{ij}\wedge c_{kj}$ are exactly the two vectors in the lower cover of $c_{ij}$. Besides the $P_{\hat{\varphi}_1}$-invariants, there are also trivial relations of the form $z_{ij, ik}-z_{ij, il}$ and $z_{ij, kj} - z_{ij, lj}$ for $i, j, k, l\in [3]$, since $c_{ij}\wedge c_{ik} = c_{ij}\wedge c_{il}$ and $c_{ij}\wedge c_{kj} = c_{ij}\wedge c_{lj}$. 

    The linear relations above include variables $z_{ij, kl}$. By  Corollary~\ref{cor: degree 1 components}, to find the linear ideal defining the span of $\mathcal V_{\varphi_1}$ these variables need to be eliminated. The elimination yields generators of the form 
    \[
    x_{ij} + x_{kl} - x_{il} - x_{kj}=0, \text{ for } i, j, k, l\in [3].
    \]

    The analogous reasoning also works for the variety $\mathcal V_{\varphi_2}$. There are only four elements $s\in P_{\varphi_2}$ with $r(s)=0$, namely, $c_{11}$, $c_{12}$, $c_{21}$ and $c_{22}$. The corresponding four linear $P_{\hat{\varphi}_2}$-invariants involve variables $z_{ij, kl}$. After the elimination there is a single generator of the linear span of $\mathcal V_{\varphi_2}$:
    \[
    x_{11}+x_{22}-x_{12}-x_{21}=0.
    \]
\end{example}

\subsection{Combinatorial toric reparameterization}
\label{subsec: combinatorial toric}

This section summarizes the consequences of the theory developed in Sections~\ref{subsec: minimal linear subspace} and~\ref{sec: posets} in regard to the toric reparameterization problem (Problem~\ref{prob: toric}). 
Recall that the map $\hat\varphi$ defined in~\eqref{eqn: poset param map} can be factored as $\hat\varphi(\theta) = M_{\hat\varphi}\phi_{\mathcal{T}}(\theta)$, where $\theta \in\Theta$ and $\phi_{\mathcal{T}}$ is the toric (monomial) map in~\eqref{eqn: nonlinear part}. 
Combining this observation with Lemma~\ref{lem: elimination} yields the following in regard to Problem~\ref{prob: toric}. 

\begin{theorem}
    \label{thm: toric condition}
    Let $\varphi(\Theta)$ be a semialgebraic set as in~\eqref{eqn: rational map} with $f_1,\ldots, f_m\in\mathbb Z_{\geq0}[\theta_1,\ldots, \theta_n]$ for which the residue vectors $r(s)$ for $s\in P_\varphi$ are either the $0$-vector or a scalar multiple of a standard basis vector.  
    Then $I_{\varphi}$ is a toric ideal up to a linear change of coordinates.  
    Hence, $\mathcal V_\varphi$ is a toric variety. 
\end{theorem}

\begin{proof}
    Proposition~\ref{prop: poset sum} implies the existence of a matrix $E_{\hat\varphi}$ corresponding to a sequence of elementary row operations on $M_{\hat\varphi}$ such that the rows of $E_{\hat\varphi}M_{\hat\varphi}$ are the residue vectors $r(s)$ for $s\in P_{\varphi}$. 
    Since $E_{\hat\varphi}$ is invertible, the kernel of the pullback of the polynomial map $E_{\hat\varphi}M_{\hat\varphi}\phi_{\mathcal{T}}(\theta)$ is equal to $I_{\hat\varphi}$. 
    Since the rows of $E_{\hat\varphi}M_{\hat\varphi}$ are either $0$-vectors or scalar multiples of standard basis vectors then $E_{\hat\varphi}M_{\hat\varphi}\phi_{\mathcal{T}}(\theta)$ is a monomial map.  
    Hence, up to the linear change of coordinates given by $E_{\hat\varphi}$, the ideal $I_{\hat\varphi}$ is toric. 
    In particular, in these new coordinates it has a Gröbner basis with respect to any term order that consists of only binomials. 
    By \cite[Theorem 2, Chapter 3.1]{cox1997ideals}, a generating set for $I_{\hat\varphi}\cap \mathbb{R}[x_1,\ldots, x_m]$ is given by the elements of such a Gröbner basis (for appropriately chosen term orders) that belong to the ring $\mathbb{R}[x_1,\ldots, x_m]$.  
    By Lemma~\ref{lem: elimination}, these binomials are a generating set for $I_{\varphi}$. 
    Since $I_{\varphi}$ is a prime ideal, it follows that it is toric up to a linear change of coordinates. 
\end{proof}

\begin{example}
    \label{ex: toric graphic matroids}
    Consider an ideal variety $\mathcal V_{\varphi_{M, \mathfrak F'}}$ as defined in Section~\ref{subsec: ideal varieties}.  
    By Example~\ref{ex: poset for graphic matroid}, the nonzero residue vectors of $P_{\varphi_{M, \mathfrak F'}}$ are the standard basis vectors $b_{\{e\}}\in \mathbb R^{\mathfrak F'}$ such that $\{e\}$ is a singleton in $\mathfrak F'$. 
    Hence, by Theorem~\ref{thm: toric condition} every ideal variety is toric.  
    For instance, when $M$ is the graphic matroid for the undirected graph $G = (V,E)$, then $\mathcal V_{\varphi_{M, \mathfrak F'}}$ is the toric variety of the edge polytope of the subgraph $H = (V, E')$ of $G$ with $E' = \bigcup_{i = 1}^sF_i$ where $F_1,\ldots, F_s$ is the (unique) antichain in $P_M$ satisfying $\mathfrak F' = \langle F_1,\ldots, F_s\rangle$. 
    By Example~\ref{ex: linear relations for graphic matroids}, this toric variety is embedded in the subspace of $\mathbb R^{\mathfrak F'}$ defined by the linear relations $\sum_{F'\preceq F \textrm{ in } \mathfrak F'}\mu(F',F)x_{F'}$ for $F\in \mathfrak F'$ a non-atom in $P_M$. 

    We note that matroid flat varieties $\mathcal V_{\varphi_{M, \mathfrak F'}}$ that are not ideal varieties can also be toric.  Take, for instance, $M$ to be the graphic matroid of the $4$-cycle in Example~\ref{ex: matroid poset invariants}, and let 
    \begin{equation}
    \label{eqn: matroid flat variety toric}
    \mathfrak F' = \{\{\{1,2\}\}, \{\{3,4\}\},\{\{1,2\},\{3,4\}\}, \{\{1,2\},\{2,4\}\}, \{\{1,3\},\{3,4\}\}\}. 
    \end{equation}
    The resulting matroid flat variety $\mathcal V_{\varphi_{M, \mathfrak F'}}$ has poset isomorphic to the subposet of $P_M$ depicted in Figure~\ref{fig: subposet graphic matroid}. 
    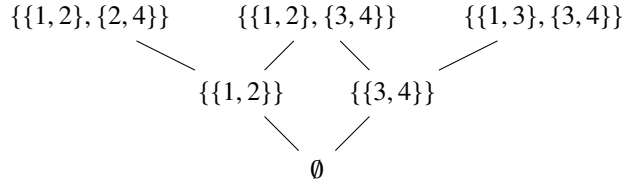
\begin{figure}[t]
    \centering
        \begin{tikzpicture}
            \node (0) at (0, 0) {\footnotesize$\emptyset$};
            \node (1) at (-1, 1) {\footnotesize$\{\{1,2\}\}$};
            \node (4) at (1, 1) {\footnotesize$\{\{3,4\}\}$};

            \node (5) at (-3, 2) {\footnotesize$\{\{1,2\}, \{2,4\}\}$};
            \node (6) at (0, 2) {\footnotesize$\{\{1,2\}, \{3,4\}\}$};
            \node (7) at (3, 2) {\footnotesize$\{\{1,3\}, \{3,4\}\}$};

             \draw[-,  >=stealth] (0) edge (1);
             \draw[-,  >=stealth] (0) edge (4);
             
             \draw[-,  >=stealth] (1) edge (6);
             \draw[-,  >=stealth] (4) edge (6);

             \draw[-,  >=stealth] (1) edge (5);
             \draw[-,  >=stealth] (4) edge (7);
             
        \end{tikzpicture}
    \caption{The subposet of the poset $P_M$ in Figure~\ref{fig: 4-cycle poset} on the set $\mathfrak F'$ in~\eqref{eqn: matroid flat variety toric}.}
    \label{fig: subposet graphic matroid}
    \end{figure}
    The nonzero residue vectors of $P_{\varphi_{M,\mathfrak F'}}$ correspond to the elements covering precisely one element, and a quick computation of the Möbius function of the above poset reveals that these are all standard basis vectors. 
    Hence, the matroid flat variety $\mathcal V_{\varphi_{M, \mathfrak F'}}$ is toric by Theorem~\ref{thm: toric condition}. 
\end{example}

Example~\ref{ex: toric graphic matroids} shows that ideal varieties for matroids are always toric, but they are not the only toric matroid flat varieties.  
It would be interesting to clarify when a matroid flat variety is toric, for instance, in the case of graphic matroids. 

\begin{problem}
    \label{prob: toric matroid flat varieties}
    What are the toric matroid flat varieties? In particular, when $M$ is a graphic matroid, for which subsets of flats $\mathfrak F'$ of $M$ is  $\mathcal V_{\varphi_{M, \mathfrak F'}}$ toric?
\end{problem}

\begin{example}
    \label{ex: toric subvarieties of secants}
    From Example \ref{ex: poset for a secant subvariety} we see that there are degenerate subvarieties of secant varieties that are toric. In the poset $P_{\varphi_1}$, vectors with non-zero residues are those in the upper cover of $\hat{0}$. Those form the standard basis of $\mathbb R^6$. It follows that the variety $\mathcal V_{\varphi_1}$ is toric. This example generalizes for larger $m, n$. One can check that by adding constraints $s_i=s_j$ for $i, j\in [m]$ and $q_i=q_j$ for $i, j\in [n]$ the resulting image is a toric variety. On the other hand, the condition is not satisfied for the poset $P_{\varphi_2}$. Using the computational method introduced in \cite{kahle2025efficiently}, one can check that there is no linear change of coordinates that makes the vanishing ideal $I_{\varphi_2}$ toric.
\end{example}

\begin{example}
    \label{ex: toric}
    It is seen from the matrices in Figures~\ref{fig: matrix a} and~\ref{fig: matrix c reduced} that the colored graphs in Figures~\ref{fig: colored a} and~\ref{fig: colored b} have toric vanishing ideals up to a linear change of coordinates, according to Theorem~\ref{thm: toric condition}. 
    Since the graphs~\ref{fig: colored a} and~\ref{fig: colored b} are $\pi$-graphs, their colored DAG models $\mathcal{M}(\GG,c^\ast)$ have as their Zariski closures a toric variety whose characters are given by the residue vectors. 
    The Möbius inversion in Theorem~\ref{thm: mobius} simply recovers these by a linear transformation of the ambient space, undoing the obfuscation of the toric embedding created by the (statistically) natural parameterization of the variety via~\eqref{eqn: trek map}. 
\end{example}

\subsection{Iterated combinatorial row reduction}
\label{subsec: iterated}

The preceding sections show that combinatorial properties of the poset $P_\varphi$ can help us deduce the linear span of $\varphi(\Theta)$ and provide a sufficient condition for $\mathcal V_\varphi$ to be toric. 
Underlying these observations is the Möbius inversion formula in Theorem~\ref{thm: mobius} which specifies a combinatorial row reduction of the matrix $M_{\hat\varphi}$.  
This combinatorial row reduction need not fully reduce $M_{\hat\varphi}$.  
However, iterating the process by constructing a poset on the residue vectors for $P_\varphi$ can further reduce $M_{\hat\varphi}$ to reveal solutions not exposed by $P_\varphi$ alone. 
We now describe this iterative procedure, starting with an example. 

\begin{example}
    \label{ex: combinatorial row reduction}
    For the graph in Figure~\ref{fig: colored c}, the residue vectors for $P_{\varphi_{\GG,c^\ast}}$, are linearly independent but not standard basis vectors.  
    However, we can iterate the combinatorial row reduction by constructing the poset $P_{\varphi_{\GG,c^\ast}}^{(1)}$ on the residue vectors for $P_{\varphi_{\GG,c^\ast}}$ using the same partial order $\preceq$; e.g., the coordinate-wise dominance order. 
    This new poset $P_{\varphi_{\GG,c^\ast}}^{(1)}$ is depicted in Figure~\ref{fig: graph c second iteration poset}.
    Applying Theorem~\ref{thm: mobius} to this poset we obtain the following residue of a residue vector:
    \[
    (0,0,0,0,1) = r(r(t(3,3)) = r(t(3,3)) - r(t(2,2)\wedge t(1,3)) = (0,0,1,0,1) - (0,0,1,0,0)
    \]
    Substituting this equation into the row for $\sigma_{33}$ in the augmented matrix in Figure~\ref{fig: matrix c reduced} amounts to an additional row reduction steps that result in a matrix, depicted in Figure~\ref{fig: matrix c second iteration}, whose nonzero rows are standard basis vectors.  
    This shows how the combinatorial row reduction via the posets $P_\varphi$ can be iterated twice to reveal the toric structure of the variety. 
    Since, for this graph, we introduced the additional variables $z_{22,13}, z_{23,33}$, the Zariski closure of $\mathcal{M}(\GG,c^\ast)$ is a toric variety obtained by projecting the toric variety defined by these (twice iterated) residue vectors in Figure~\ref{fig: matrix c second iteration} along the coordinates $z_{22, 13}, z_{23, 33}$.  
\end{example}

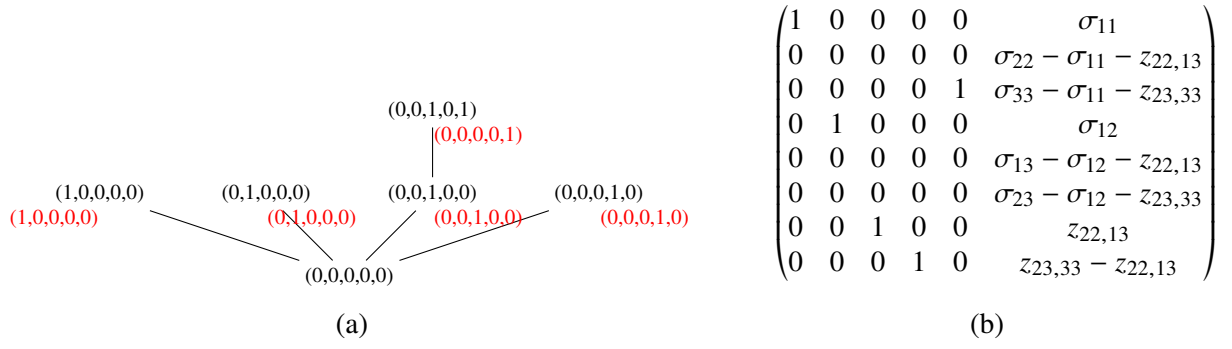
\begin{figure}[t]
    \begin{subfigure}[b]{0.6\textwidth}
    \centering
    {\tiny
    \begin{tikzpicture}[scale = 1.1]
            \node (0) at (1, 0) {(0,0,0,0,0)};
            \node (1) at (-2, 1) {(1,0,0,0,0)};
            \node (2) at (0, 1) {(0,1,0,0,0)};
            \node (3) at (2, 1) {(0,0,1,0,0)};
            \node (4) at (2, 2) {(0,0,1,0,1)};
            \node (5) at (4, 1) {(0,0,0,1,0)};

             \draw[-,  >=stealth] (0) edge (1);
             \draw[-,  >=stealth] (0) edge (2);
             \draw[-,  >=stealth] (0) edge (3);
             \draw[-,  >=stealth] (0) edge (5);

             \draw[-,  >=stealth] (3) edge (4);


            \node (r1) at (-2 - 0.55, 1 - 0.3) {\textcolor{red}{(1,0,0,0,0)}};
            \node (r2) at (0 + 0.55, 1 - 0.3) {\textcolor{red}{(0,1,0,0,0)}};
            \node (r7) at (4 + 0.55, 1 - 0.3) {\textcolor{red}{(0,0,0,1,0)}};
            \node (r3) at (2 + 0.55, 1 - 0.3) {\textcolor{red}{(0,0,1,0,0)}};
            \node (r5) at (2 + 0.55, 2 - 0.3) {\textcolor{red}{(0,0,0,0,1)}};
             
    \end{tikzpicture}
    }
    \caption{}
    \label{fig: graph c second iteration poset}
    \end{subfigure}
    \hfill
    \begin{subfigure}[b]{0.35\textwidth}
    \centering
    {\small
    \[
    \begin{pmatrix}
        1 & 0 & 0 & 0 & 0 & \sigma_{11} \\
        0 & 0 & 0 & 0 & 0 & \sigma_{22} - \sigma_{11} - z_{22,13}\\
        0 & 0 & 0 & 0 & 1 & \sigma_{33} - \sigma_{11} - z_{23,33}\\
        0 & 1 & 0 & 0 & 0 & \sigma_{12}\\
        0 & 0 & 0 & 0 & 0 & \sigma_{13} - \sigma_{12} - z_{22,13}\\
        0 & 0 & 0 & 0 & 0 & \sigma_{23} - \sigma_{12} - z_{23, 33}\\
        0 & 0 & 1 & 0 & 0 & z_{22,13}\\
        0 & 0 & 0 & 1 & 0 & z_{23,33} - z_{22,13}\\
    \end{pmatrix}
    \]
    }
    \caption{}
    \label{fig: matrix c second iteration}
    \end{subfigure}

\caption{(a) The poset $P_{\varphi_{\GG,c^\ast}}^{(1)}$ on the residue vectors of $P_{\varphi_{\GG,c^\ast}}$ for $(\GG,c^\ast)$ in Figure~\ref{fig: colored c}, and (b) the augmented matrix $[M_{\GG,c^\ast}\mid \Sigma]$ after further row-reducting the matrix in Figure~\ref{fig: matrix c reduced} according to the Möbius inversion formula in Theorem~\ref{thm: mobius} applied to $P_{\varphi_{\GG,c^\ast}}^{(1)}$.}
\label{fig: poset and matrix c second iteration}
\end{figure}

Example~\ref{ex: combinatorial row reduction} raises an interesting question related to the distinction between row reduction on $M_{\hat\varphi}$ and the combinatorial row reduction on $M_{\hat\varphi}$ performed by the Möbius inversion formula in Theorem~\ref{thm: mobius}. 
As we saw in Section~\ref{subsec: combinatorial linear equiv}, replacing each row $s$ in $M_{\hat\varphi}$ with its residue vector $r(s)$ amounts to applying a set of row operations to the matrix $M_{\hat\varphi}$.  
However, as shown in Example~\ref{ex: combinatorial row reduction}, we may need to iterate this combinatorial row reduction to reveal our sufficient condition for toricity of the variety $\mathcal V_\varphi$; i.e., that the nonzero rows of the reduced matrix are scalars of standard basis vectors. 

To formalize this, 
recall that the partial order $\preceq$ on $P_\varphi$ is the coordinate-wise dominance order. 
Set $P_\varphi^{(0)}:= P_\varphi$, and let $P_\varphi^{(k)}:= (r(P_\varphi^{(k-1)}),\preceq)$ for $k>0$ where 
\[
r(P_\varphi^{(k-1)}) = \{r\in \mathbb Z_{\geq 0}^{\mathcal T} : r = \sum_{s'\preceq s \textrm{ in } P_\varphi^{(k-1)}}\mu(s',s)s' \textrm{ for some } s\in P_{\varphi}^{(k-1)}\}. 
\]
In other words, $P_\varphi^{(k)}$ is the poset with ground set the residue vectors of $P_\varphi^{(k-1)}$ under the coordinate-wise dominance order $\preceq$. 
The contents of this paper show that $\mathcal V_\varphi=\overline{\varphi(\Theta)}$ will be isomorphic to a toric variety if there exists $k\geq 0$ such that 
$r(P_\varphi^{(k)})$ is a collection of (nonnegative) scalars of standard basis vectors. 
Given a family of semialgebraic sets $\mathcal F$ with coordinate functions having nonnegative integral coefficients, one way to identify subfamilies $\mathcal F'\subset \mathcal F$ with toric reparametrizations is to characterize the $\varphi(\Theta)\in \mathcal F$ for which such a $k\geq 0$ exists. 
So it would be interesting to see a solution to the following problem for families of semialgebraic sets that are of interest in the literature. 

\begin{problem}
    \label{prob: stratified toric DAGs}
    Fix a nonnegative integer $k$, and consider the family of semialgebraic sets $\mathcal F= \{\varphi_i(\Theta_i): \Theta_i\in \mathbb R^{n_i}, \varphi: \Theta_i\to \mathbb R^{m_i}, i\in\mathbb{Z}_{\geq0}\}$.  Characterize $\varphi_i(\Theta_i)\in \mathcal F$ such that $P_{\varphi}^{(k)}$ is a poset on nonnegative scalars of standard basis vectors. 
\end{problem}

A solution to Problem~\ref{prob: stratified toric DAGs} amounts to a characterization of the $\varphi(\Theta)\in \mathcal F$ with $\overline{\varphi(\Theta)}$ toric where the toric structure is observable from a combinatorial row reduction of the coordinate functions of $\varphi$.
Solving Problem~\ref{prob: stratified toric DAGs} for the colored DAG models 
$\mathcal F = \{\mathcal M(\GG,c) = \varphi_{\GG,c}(\Theta) : (\GG,c) \textrm{ a colored DAG}\}$
would be of interest to algebraic statistics.  
The same problem is also interesting to study for other families of semialgebraic sets; such as the matroid flat varieties or degenerate subvarieties of secants in Section~\ref{section:prelim}. 
For group-based phylogenetic models, this approach to proving toricity of $\mathcal V_\varphi$ gives a simple, purely combinatorial way of recovering well-known results. 
This is demonstrated in the following example.

\begin{example}[Toric phylogenetic varieties]
    \label{ex: iterated jukes cantor}
    Recall the 4-leaf Jukes-Cantor model from Section~\ref{subsec: phylogenetics}. 
    The vanishing ideal $\mathcal V_{\varphi_{JC}}$ of this model is known to be toric by classic results in group-based phylogenetic models, where toric structure is classically revealed via the discrete Fourier transform.
    However, in the initial poset $P_{\varphi_{JC}}^{(0)}$, eight of the residue vectors are not standard basis vectors. 
    On the other hand, the poset $P_{\varphi_{JC}}^{(1)}$, depicted in Figure~\ref{fig: P1 phylo}, consists of $36$ elements whose residue vectors are all scalar multiples of standard basis vectors.
    Thus, a single iteration fully diagonalizes the system. This provides a purely combinatorial proof that the Zariski closure of the model is a toric variety, without the need for the discrete Fourier transform.
\end{example}

\begin{figure}[t]
    \centering
    \includegraphics[width=0.4\textwidth]{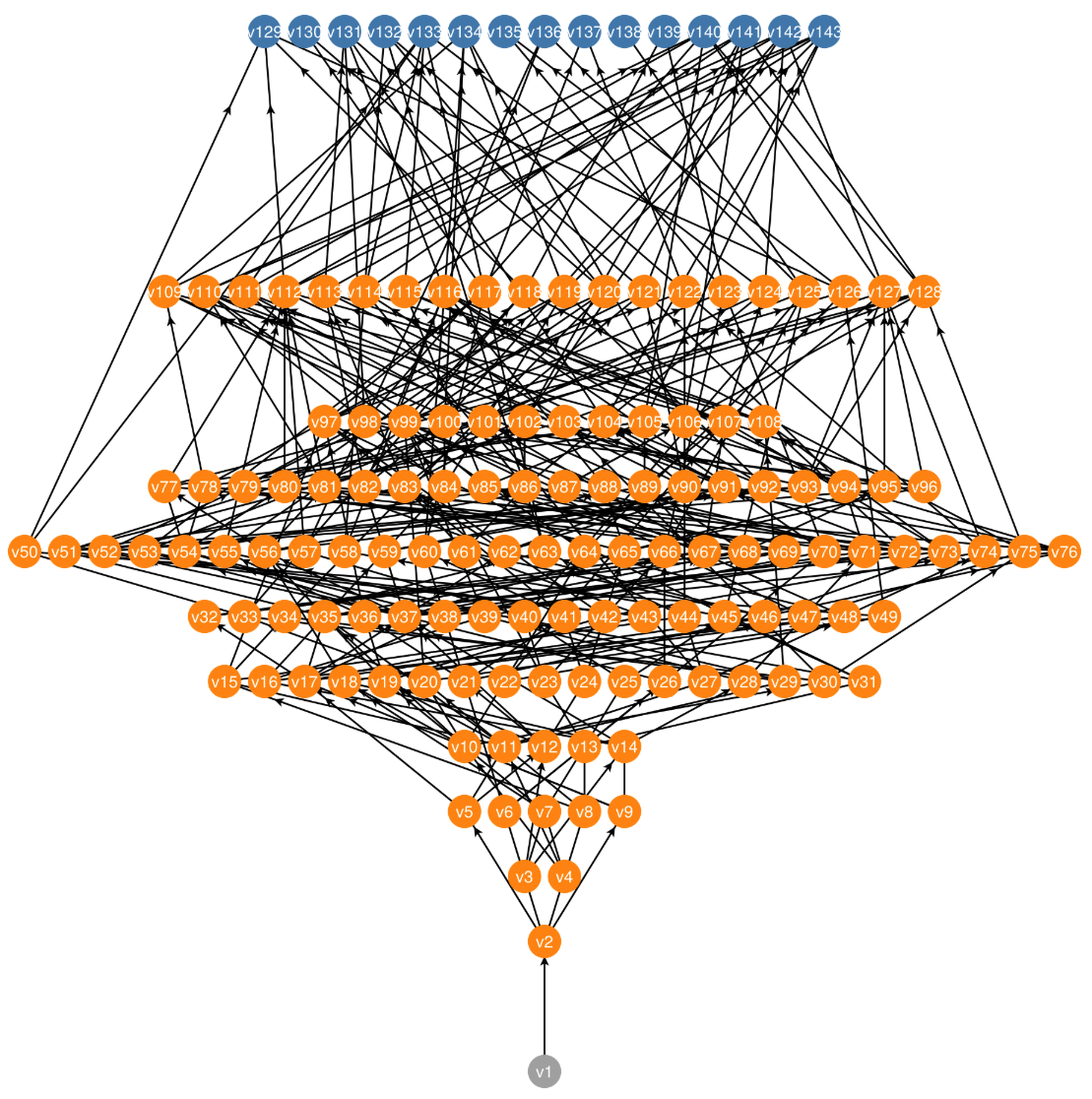}
    \includegraphics[width=0.4\textwidth]{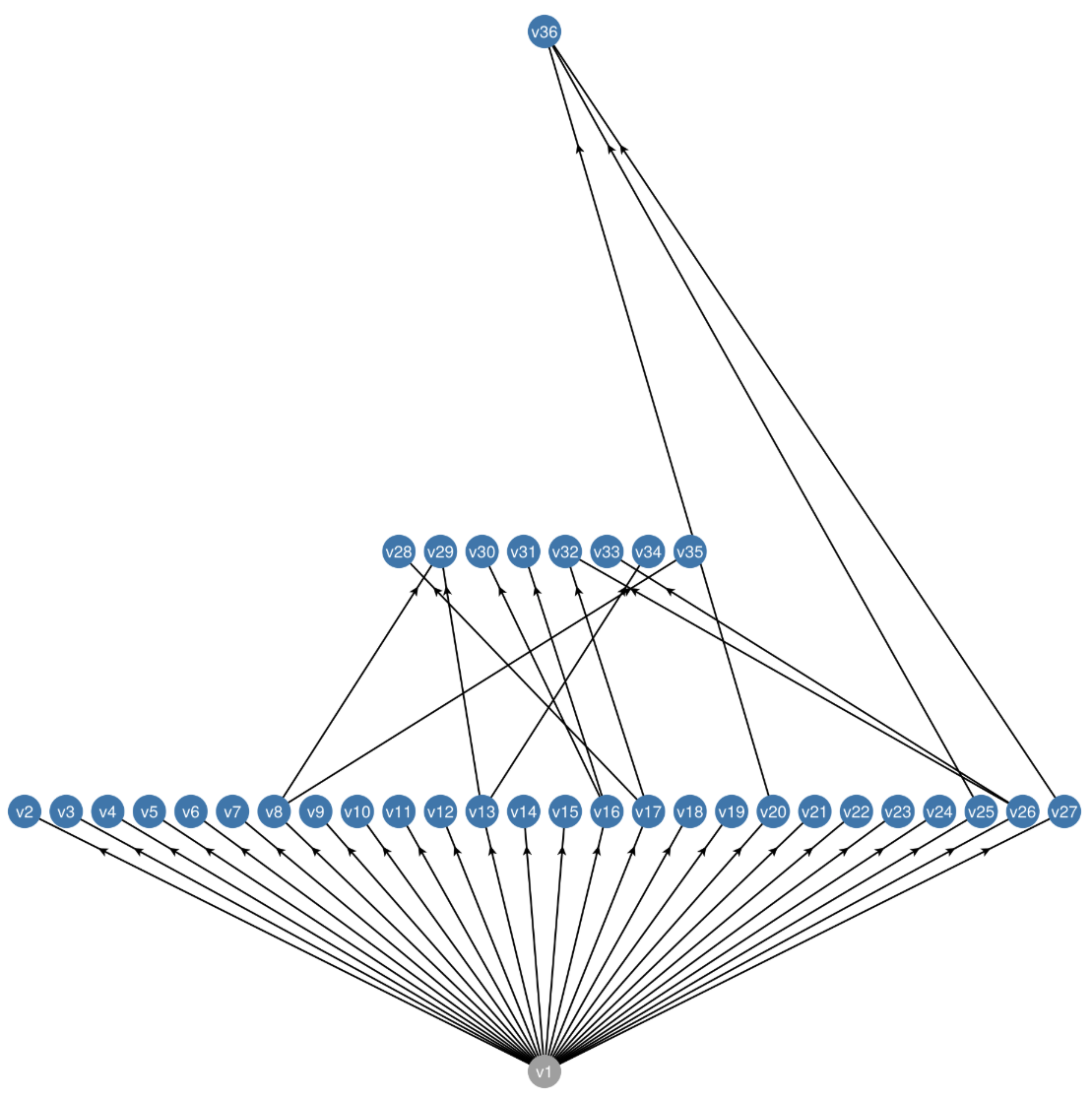}
    \caption{Posets $P_{\varphi_{JC}}^{(0)}$ (left) and $P_{\varphi_{JC}}^{(1)}$ (right) for the 4-leaf tree under the Jukes Cantor model. Nodes in blue represent the coefficient vectors of $\varphi$ while in orange are the meet elements. 
    } 
    \label{fig: P1 phylo}
\end{figure}

\subsection{Applications to implicitization}
\label{rem: toric implicitization}
    We end by describing how the poset $P_\varphi$ allows for complete solutions to the implicitization problem (Problem~\ref{prob: colored implicitization})  when $P_\varphi$ fulfills the conditions of Theorem~\ref{thm: lin independence} and Theorem~\ref{thm: toric condition}. 
    
    Given a generating set for $\ker(\phi^\ast_{\mathcal{T}})$, a solution to the implicitization problem for $I_{\varphi}$ fulfilling the conditions in Theorem~\ref{thm: toric condition} may be obtained by tracing the generators through the linear map $M_{\hat\varphi}$ and eliminating. Note that when $P_{\varphi}$ is a $\pi$-system the elimination step is not required.
    
    \begin{example}
    Consider the ideal variety $\mathcal V_{\varphi_{M_G, \mathfrak F'}}$ for a graphic matroid $M_G$ and $\mathfrak F'=\langle F_1,\ldots, F_s\rangle$, and let $H = (V,E')$ be the subgraph of $G$ with edges $E' = \bigcup_{i=1}^sF_s$.  
    Then, by Example~\ref{ex: toric graphic matroids}, we have that
    \[
    I_{\varphi_{M_G,\mathfrak F'}} = \langle \sum_{F'\preceq F \textrm{ in } \mathfrak F'}\mu(F',F)x_{F'}\rangle + \left\langle \prod_{k=1}^qx_{i_{2k-1}} - \prod_{k=1}^qx_{i_{2k}} : (e_1,\ldots, e_{i_{2q}}) \textrm{ a closed walk in } H\right\rangle,
    \]
    where the nonlinear generators are the generators of the edge ideal of $H$ computed in \cite[Chapter 5.3]{herzog2018binomial}. 
    \end{example}

    In general, when the variety is toric according to Theorem~\ref{thm: toric condition}, the poset reveals the toric reparametrization, which is the linear transformation given by Möbius inversion.  

    \begin{example}
    The matroid flat variety $\mathcal V_{\varphi_{M_G, \mathfrak F'}}$ for the graphic matroid $M_G$ of the $4$-cycle and $\mathfrak F'$ in~\eqref{eqn: matroid flat variety toric} lives in the linear space defined by the linear $P_{\varphi_{M_G,\mathfrak F'}}$-invariant 
    \[
    x_{12,34} - x_{12} - x_{34}
    \]
    given by the single element $\{12,34\}$ with residue vector $0$ in the poset in Figure~\ref{fig: subposet graphic matroid}. It has toric reparametrization
    \[
    \begin{split}
        \varphi'(\theta):\theta \mapsto (t_{12} :=\theta_1\theta_2 &= x_{12},\\
          t_{34} :=\theta_3\theta_4 &= x_{34},\\
          t_{24} :=\theta_2\theta_4 &= \mu(\{12, 24\},\{12, 24\})x_{12,24} +  \mu(\{12\},\{12, 24\})x_{12} = x_{12, 24} - x_{12},\\
          t_{13} :=\theta_1\theta_3 &= \mu(\{13, 34\},\{13, 34\})x_{13,34} +  \mu(\{34\},\{13, 34\})x_{34} = x_{13, 34} - x_{34}),
    \end{split}
    \]
    given by the elements with nonzero residue vectors.   
    Since $I_{\varphi'} =\langle t_{12}t_{34} - t_{13}t_{24}\rangle$, we obtain a complete solution to the implicitization problem for $\mathcal V_{\varphi_{M_G, \mathfrak F'}}$: 
    \[
    I_{\varphi_{M_G,\mathfrak F'}} = \langle x_{12,34} - x_{12} - x_{34}\rangle + \langle x_{12}x_{34} - (x_{13,34} - x_{34})(x_{12, 24} - x_{12})\rangle. 
    \]
    \end{example}

    More generally, for a family of semialgebraic sets $\mathcal F= \{\varphi_i(\Theta_i): \Theta_i\in \mathbb R^{n_i}, \varphi: \Theta_i\to \mathbb R^{m_i}, i\in\mathbb{Z}_{\geq0}\}$, if the posets $P_{\varphi_i}$ are understood and a basis of the toric ideals $I_{\phi_{\mathcal T_i}}$ are known, then the techniques in this paper recover complete solutions to the implicitization problem.  
    This reduces problems in algebraic geometry to poset and lattice point combinatorics (for understanding $P_{\varphi_i}$ and $I_{\phi_{\mathcal T_i}}$, respectively).
    
    In a follow-up paper, we demonstrate this approach when $\mathcal F$ is a subfamily of colored DAG models. 
    In this case, these techniques allow us to completely solve all five problems~\ref{prob: colored implicitization},~\ref{prob: toric},~\ref{prob: linear span},~\ref{prob: colored distinguishability},~\ref{prob: linear equivalence} described in the introduction.

\subsection*{Data availability statement}
All computational examples presented in this paper can be verified using the Julia code available at our GitHub repository: \url{https://github.com/marinagarrote/poset-parametrizations}.

\subsection*{Acknowledgements}
M.~G-L was partially funded by the Beatriu de Pinós postdoctoral programme of the Department of Research and Universities of the Generalitat de Catalunya (ref. 2024BP00235).
N.~K. was partially supported by the research visit grant awarded by the Foundation for Aalto University Science and Technology.
L.~S. was partially supported by the 3-year Prize for Young Researchers awarded by the G\"oran Gustafsson Foundation, a Mercator Fellowship funded by the DFG priority programme SPP 2458 \emph{Combinatorial Synergies}, and the KTH Center for Digital Futures.
L.~S. and M.~G-L. were partially supported by the Wallenberg AI, Autonomous Systems and Software Program (WASP) funded by the Knut and Alice Wallenberg Foundation.

\bibliographystyle{plainnat}
\bibliography{main}

\end{document}